\documentclass[12pt,twoside,british]{report}
\usepackage[utopia]{mathdesign}

\usepackage[T1]{fontenc}
\usepackage[latin9]{inputenc}
\usepackage[a4paper]{geometry}
\usepackage{xcolor}
\usepackage{babel}
\usepackage{verbatim}
\usepackage{units}
\usepackage{textcomp}
\usepackage{mathtools}
\usepackage{enumitem}
\usepackage{amsmath}
\usepackage{amsthm}
\usepackage{stmaryrd}
\usepackage{makeidx}
\makeindex
\usepackage{graphicx}
\usepackage{setspace}
\usepackage{xargs}[2008/03/08]
\PassOptionsToPackage{normalem}{ulem}
\usepackage{ulem}
\usepackage{nomencl}
\providecommand{\printnomenclature}{\printglossary}
\providecommand{\makenomenclature}{\makeglossary}
\makenomenclature
\usepackage[unicode=true,pdfusetitle,
 bookmarks=true,bookmarksnumbered=false,bookmarksopen=false,
 breaklinks=false,pdfborder={0 0 1},backref=false,colorlinks=false]
 {hyperref}

\makeatletter

\newenvironment{subproof}[1][\proofname]{%
  \begin{proof}[#1]%
}{%
  \end{proof}%
}

\renewcommand{\slash}{/\penalty\exhyphenpenalty\hspace{0pt}}

\newcommand{\noun}[1]{\textsc{#1}}

\numberwithin{equation}{section}
\newlength{\lyxlabelwidth}      
\theoremstyle{plain}
\newtheorem{thm}{\protect\theoremname}[section]
\theoremstyle{definition}
\newtheorem{defn}[thm]{\protect\definitionname}
\theoremstyle{remark}
\newtheorem{notation}[thm]{\protect\notationname}
\theoremstyle{remark}
\newtheorem{rem}[thm]{\protect\remarkname}
\theoremstyle{plain}
\newtheorem{lem}[thm]{\protect\lemmaname}
	\newenvironment{elabeling}[2][]%
	{\settowidth{\lyxlabelwidth}{#2}
		\begin{description}[font=\normalfont,style=sameline,
			leftmargin=\lyxlabelwidth,#1]}
	{\end{description}}
\theoremstyle{plain}
\newtheorem{cor}[thm]{\protect\corollaryname}
\theoremstyle{plain}
\newtheorem{prop}[thm]{\protect\propositionname}
\theoremstyle{remark}
\newtheorem*{note*}{\protect\notename}
\theoremstyle{remark}
\newtheorem{note}[thm]{\protect\notename}
\theoremstyle{definition}
\newtheorem{example}[thm]{\protect\examplename}
\newcommand{\lyxaddress}[1]{
	\par {\raggedright #1
	\vspace{1.4em}
	\noindent\par}
}
\theoremstyle{plain}
\newtheorem{question}[thm]{\protect\questionname}
\theoremstyle{plain}
\newtheorem{conjecture}[thm]{\protect\conjecturename}

\usepackage{tikz}

\usetikzlibrary{matrix}
\usetikzlibrary{arrows}
\usetikzlibrary{patterns}

\usepackage{fancyhdr}

\makeatother

\def\nomname{Symbols and Notation}
\providecommand{\conjecturename}{Conjecture}
\providecommand{\corollaryname}{Corollary}
\providecommand{\definitionname}{Definition}
\providecommand{\examplename}{Example}
\providecommand{\lemmaname}{Lemma}
\providecommand{\notationname}{Notation}
\providecommand{\notename}{Note}
\providecommand{\propositionname}{Proposition}
\providecommand{\questionname}{Question}
\providecommand{\remarkname}{Remark}
\providecommand{\theoremname}{Theorem}

\begin{document}
\selectlanguage{british}

\global\long\def\m#1{\operatorname{#1}}%
\global\long\def\ml#1{\operatorname*{#1}}%

\newcommandx\ubdelim[6][usedefault, addprefix=\global, 1=., 2=.]{\left #1 #3\smash{\underbrace{#4}_{#5}}\vphantom{#4}#6\right #2 \vphantom{\underbrace{#4}_{#5}}}
\newcommandx\obdelim[6][usedefault, addprefix=\global, 1=., 2=.]{\left #1 #3\smash{\overbrace{#4}^{#5}}\vphantom{#4}#6\right #2 \vphantom{\overbrace{#4}^{#5}}}

\global\long\def\expfrac#1#2{{\scriptstyle #1}/\raisebox{-1pt}{\ensuremath{{\scriptstyle #2}}}}%

\DeclarePairedDelimiter{\paren}{\lparen}{\rparen}%
\DeclarePairedDelimiter{\pbrace}{\lbrace}{\rbrace}%
\DeclarePairedDelimiter{\abs}{\lvert}{\rvert}%
\DeclarePairedDelimiter{\norm}{\lVert}{\rVert}%
\global\long\def\group#1{\m{#1}}%

\global\long\def\GL{\group{GL}}%
\global\long\def\SL{\group{SL}}%
\global\long\def\GO{\group O}%
\global\long\def\SO{\group{SO}}%
\global\long\def\GU{\group U}%
\global\long\def\SU{\group{SU}}%
\global\long\def\GS{\group S}%

\newcommandx\CH[1][usedefault, addprefix=\global, 1=\bullet]{H^{#1}}%

\global\long\def\nt{\trianglelefteq}%
\global\long\def\tn{\triangleright}%
\global\long\def\operp{\mathbin{\textcircled\ensuremath{\perp}}}%
\global\long\def\pmmg{\mathbin{\begin{subarray}{c}
\pm\\
\cdot:
\end{subarray}}}%

\global\long\def\subqed{\qedhere\qquad\text{}}%
\global\long\def\subsubqed{\qedhere\qquad\qquad\text{}}%

\global\long\def\End{\m{End}}%
\global\long\def\Aut{\m{Aut}}%
\global\long\def\Hom{\m{Hom}}%
\global\long\def\Eig{\m{Eig}}%
\global\long\def\Spec{\m{Spec}}%
\global\long\def\Specm{\m{Spec_{max}}}%

\global\long\def\Re{\m{Re}}%
\global\long\def\Im{\m{Im}}%
\global\long\def\id{\m{id}}%
\global\long\def\tr{\m{tr}}%
\global\long\def\rk{\m{rk}}%
\global\long\def\sgn{\m{sgn}}%
\global\long\def\diag{\m{diag}}%
\global\long\def\dist{\m{dist}}%
\global\long\def\abst{\m{abst}}%
\global\long\def\rot{\m{rot}}%
\global\long\def\im{\m{im}}%
\global\long\def\nil{\m{nil}}%
\global\long\def\ggt{\m{ggT}}%
\global\long\def\Var{\m{Var}}%
\global\long\def\supp{\m{supp}}%
\global\long\def\spann{\m{span}}%
\global\long\def\ord{\m{ord}}%
\global\long\def\inv{\m{inv}}%
\global\long\def\pt{\m{pt}}%
\global\long\def\Sym{\m{Sym}}%
\global\long\def\Mat{\m{Mat}}%
\global\long\def\Gr{\m{Gr}}%
\global\long\def\Stab{\m{Stab}}%

\global\long\def\Pow{\m{Pow}}%
\global\long\def\Res{\m{Res}}%
\global\long\def\Jac{\m{Jac}}%
\global\long\def\Transp{\mathsf{T}}%

\pagenumbering{roman}
\title{Fixed points and dynamics of holomorphic maps}
\author{Josias Reppekus}

\maketitle
\cleardoublepage{}

\fancyhead[LE]{\nouppercase\leftmark}
\fancyhead[RO]{\nouppercase\rightmark}
\fancyhead[LO]{}
\fancyhead[RE]{}
\pagestyle{fancy}

\pagestyle{plain}

\phantomsection\addcontentsline{toc}{chapter}{Introduction}%

\chapter*{Introduction}

\selectlanguage{british}

In this manuscript we systematically review known results of local
dynamics of discrete local holomorphic dynamics near fixed points
in one and several complex variables as well as the consequences in
global dynamics. 

Discrete holomorphic dynamics is the study of the behaviour of the
sequence of iterates $\{F^{\circ n}\}_{n\in\mathbb{N}}$ of a holomorphic
endomorphism $F:M\to M$ of a complex manifold $M$, or the behaviour
of the \emph{orbits} $\{F^{\circ n}(p)\}_{n}$ of points $p\in M$.
We are mainly interested in invariant subsets of $M$, their geometric
properties, and the behaviour of the contained orbits. We will examine
these questions from a local and a global point of view. 

In local holomorphic dynamics, we focus on a fixed point of $F$ and
examine the behaviour of orbits of points in a small neighbourhood
of that fixed point. 

In global holomorphic dynamics, our central objects of interest is
the \emph{Fatou set} of all points in $M$ that admit a neighbourhood
on which $\{F^{\circ n}\}_{n\in\mathbb{N}}$ is normal, and its connected
components, the \emph{Fatou components} of $F$. A Fatou component
can be thought of as a maximal connected subsets of $M$ on which
the behaviour of orbits of $F$ is stable under small changes of the
starting point. 

In dimension one, the local dynamics near fixed points are almost
completely understood and classified, depending on the derivative
of $F$ at the fixed point. The currently final major piece in the
puzzle was provided by results of Pérez-Marco, in \cite{PerezMarco1997FixedPointsandCircleMaps}
and other publications, with the concept of hedgehogs, that shed light
on the most complicated case of Cremer fixed points. Moreover, in
one complex variable, every (interesting) invariant Fatou component
is hyperbolic, so by Denjoy-Wolff theory, all orbits either accumulate
inside or at a single fixed point on the boundary. Hence, local dynamics
near fixed points played a major part in the classification of these
components by Fatou, Julia and others. For rational maps, Sullivan
famously ruled out the existence of Fatou components whose orbits
never land in the same component twice, called \emph{wandering} \emph{domains,}
completing the classification of all Fatou components in this case.
Sullivan's non-wandering result does not hold for transcendental functions,
where wandering domains are known to occur. 

In several variables, the local situation is made more complicated
by the presence of multiple eigenvalues of the linear part of $F$
at a fixed point and mixing of the one-dimensional behaviours associated
to each eigenvalue. In particular, non-trivial multiplicative relations
between eigenvalues, called resonances, lead to new problems and phenomena.
Especially in the Cremer-case little is known about local dynamics,
as the hedgehog-techniques do not generalise. 

In global dynamics of general holomorphic maps little is known, but
there are strong results on suitable subclasses such as automorphisms
of $\mathbb{C}^{2}$ or endomorphisms of projective space. For polynomial
automorphisms of $\mathbb{C}^{2}$,  Bedford and Smillie \cite{BedfordSmillie1991PolynomialdiffeomorphismsofmathbbC2IIStablemanifoldsandrecurrence}
and Fornæss and Sibony \cite{FornessSibony1992ComplexHenonmappingsinmathbbC2andFatouBieberbachdomains}
classified the invariant Fatou components that have orbits accumulating
inside. If the Jacobian of the map is moreover sufficiently small,
then Lyubich and Peters \cite{LyubichPeters2014ClassificationofinvariantFatoucomponentsfordissipativeHenonmaps}
show that the only other case is again convergence to a fixed point
in the boundary with well-understood local dynamics, completing the
classification of invariant Fatou components in this case.

These global arguments reducing the problem to a local one exemplify
the role of local dynamics in classifying global stable dynamics,
but they are not the focus of this work. We instead start from purely
local considerations and observe the global objects that emerge: We
can use local dynamics to obtain local coordinates conjugating the
map to a map with simple dynamics, e.g. a linear map. Such coordinates
shed further light on the orbit behaviour near the fixed point, but
for automorphisms in particular, these local coordinates extend to
the whole Fatou component and let us identify its biholomorphic type.
This was the basis of the first constructions of proper subdomains
of $\mathbb{C}^{2}$ biholomorphic to $\mathbb{C}^{2}$, so-called
Fatou-Bieberbach domains.

\subsubsection*{Structure}

In Chapter~\ref{sec:BasicsLocal}, we introduce the basic objects
and tools in local dynamics followed by the Fatou and Julia set in
global dynamics.

Chapter~\ref{sec:Local1Ddyn} describes the classification of local
dynamics and Fatou components in one complex variable.

Chapter~\ref{chap:2DlocDyn} is the main part of this work: a review
of local dynamics in complex dimension two and above, including the
theory of formal Poincaré-Dulac normal forms and conditions for convergence
of (partial) normalisations. In particular, we systematically review
known results on the local dynamics near different types of fixed
points.

Lastly in Chapter~\ref{chap:FatouAut}, we survey known classification
results and examples of Fatou components in several complex variables.
We then relate what we know about local dynamics to the classification
of attracting Fatou components of automorphisms of complex affine
space and close with a number of open questions and conjectures.

\cleardoublepage\phantomsection\addcontentsline{toc}{chapter}{\contentsname}\tableofcontents{}\cleardoublepage{}

\selectlanguage{british}
\global\long\def\Transp{\mathsf{T}}%
\global\long\def\Pow{\m{Pow}}%
\global\long\def\ord{\m{ord}}%

\nomenclature[N]{$\mathbb N$}{The set of natural numbers including zero $\{0,1,2,3,\ldots\}$.}

\nomenclature[C]{$\mathbb{C}$}{The set of complex numbers.}

\nomenclature[R]{$\mathbb{R}$}{The set of Real numbers.}

\nomenclature[Q]{$\mathbb{Q}$}{The set of rational numbers.}

\nomenclature[Aut(M,p)]{$\Aut(M,p)$}{The set of germs of invertible holomorphic self-maps of a manifold $M$ fixing the point $p \in M$.} 

\nomenclature[Aut(N,p)]{$\Aut_L(\mathbb{C}^d,0)$}{The set of germs in $\Aut(\mathbb{C}^d,0)$ with linear part $L\in \mathbb{C}^{d\times d}$ at the fixed point $0$.}

\nomenclature[End(M,X)]{$\End(M,X)$}{The set of germs of holomorphic self-maps of a manifold $M$ mapping the closed subset $X\subseteq M$ into itself.}

\nomenclature[End(M,p)]{$\End(M,p)$}{The set of germs of holomorphic self-maps of a manifold $M$ fixing the point $p \in M$.} 

\nomenclature[End(M,q)]{$\End_L(\mathbb{C}^d,0)$}{The set of germs in $\End(\mathbb{C}^d,0)$ with linear part $L\in \mathbb{C}^{d\times d}$ at the fixed point $0$.}

 \nomenclature[F:(M,X)->(N,Y)]{$F:(M,X)\to (N,Y)$}{A germ of maps from $M$ to $N$ at $X \subseteq M$ sending $X$ into $Y\subseteq N$.}

\nomenclature[F:(M,X)->(N,Y)]{$F:(M,p)\to (N,q)$}{A germ of maps from $M$ to $N$ at $p \in M$ sending $p$ to $q \in N$.}

\nomenclature[Fn]{$F^{\circ n}$}{The $n$-th iterate of an endomorphism $F$.}

\nomenclature[Pow(Cd,0)]{$\Pow(\mathbb{C}^{d},0)= \mathbb{C}_{0}[[z_{1},\ldots,z_{d}]]^{d}$}{The space of $d$-tuples of formal power series in $z_1,\ldots,z_d$ with complex coefficients and vanishing constant term.}

\nomenclature[Pow(Cd,1)]{$\Pow_L(\mathbb{C}^{d},0)$}{The space of $d$-tuples of formal power series with complex coefficients, vanishing constant term, and linear part $L \in \mathbb{C}^{d \times d}$.}

\nomenclature[Projn]{$\mathbb{P}^{n}$}{The complex projective space of dimension $n$}

\nomenclature[Proj(V)]{$\mathbb{P}(V)$}{The complex projective space of one-dimensional subspaces of  $V$.}

\nomenclature[Pkkp]{$p_{k}\xrightarrow[k\to\infty]{}p$}{The sequence $(p_k)_{k\in \mathbb N}$ converges to $p$.}

\nomenclature[Brp]{$B_{r}(p)$}{The euclidean ball of radius $r$ around $p \in \mathbb C^{d}$.}

\nomenclature[V]{$\overline{V}$}{The closure of a set $V$ in the surrounding topological space.}

\nomenclature[ej]{$e_{j}$}{The unit vector with $1$ in the $j$-th component.}

\nomenclature[dFp]{$dF_{p}$}{The linear part of a function, germ or power series at the point $p$.}

\nomenclature[diagA1Ak]{$\diag (A_{1},\ldots,A_{k})$}{The block-diagonal matrix with blocks $A_1, \ldots, A_k$.}

\nomenclature[z]{{$\lVert z\rVert$}}{The euclidean norm of a vector $z\in\mathbb{C}^{d}$}

\nomenclature[SigmaFU]{$\Sigma_{F}(U)$}{The local stable set containing all orbits under the map $F$ that never leave the set $U$.}

\nomenclature[AFS]{$A_{F}(p)$}{The realm of attraction to a point $p$ containing all orbits under the map $F$ that converge to $p$.}

\nomenclature[AFS]{$A_{F}(S)$}{The realm of attraction to a set $S$ containing all orbits under the map $F$ whose distance to $S$ converges to $0$.}

\nomenclature[dFp]{$dF_{p}$}{The linear part of a germ $F\in\End(\mathbb{C}^d,p)$.}

\phantomsection\addcontentsline{toc}{chapter}{\nomname}\printnomenclature[5.1em]{}

\subsection*{Landau notation}

For a topological space $D$ and maps $f,g:D\to\mathbb{C}$, we use
Bachmann-Landau notation for semi-local behaviour:
\begin{itemize}
\item $f(x)=O(g(x))$ for $x\in D$, if $|f(x)|\le C|g(x)|$ for all $x\in D$
for some $C>0$,
\item $f(x)\approx g(x)$ for $x\in D$, if $f(x)=O(g(x))$ and $g(x)=O(f(x))$
(also denoted $f(x)=\Theta(g(x))$ in the literature),
\end{itemize}
and for asymptotic behaviour near $x_{0}\in D$:
\begin{itemize}
\item $f(x)=O(g(x))$ as $x\to x_{0}$, if $\limsup_{x\to x_{0}}\frac{|f(x)|}{|g(x)|}=C<+\infty$,
\item $f(x)\approx g(x)$ as $x\to x_{0}$, if $f(x)=O(g(x))$ and $g(x)=O(f(x))$
as $x\to x_{0}$,
\item $f(x)=o(g(x))$ as $x\to x_{0}$, if $\lim_{x\to x_{0}}\frac{|f(x)|}{|g(x)|}=0$,
\item $f(x)\sim g(x)$ as $x\to x_{0}$, if $\lim_{x\to x_{0}}\frac{f(x)}{g(x)}=1$
or $f(x)=g(x)(1+o(1))$ as $x\to x_{0}$.
\end{itemize}
We use the same symbol $O(g(x))$ to denote a vector of functions
of the same type with the appropriate dimension for its context. Multiple
arguments denote a sum: $O(f(x),g(x)):=O(f(x))+O(g(x))$. If the coordinates
are clear, then we use $O(k)$ as short-hand for terms of total order
at least $k$ in all occurring variables.

\subsection*{Indices}

As we use lower indices to indicate the orbit $\{z_{n}\}_{n}=\{F^{\circ n}(z)\}_{n}$
of a point $z\in\mathbb{C}^{d}$, we often use upper indices for the
components of $z=(z^{1},\ldots,z^{d})$. We try to always indicate
this clearly to avoid confusion with exponents.

\cleardoublepage{}

\pagenumbering{arabic}

\pagestyle{fancy}

\chapter{\label{sec:BasicsLocal}Framework of holomorphic iteration}

In this chapter we set up the central objects and tools in dynamics
of holomorphic maps. Sections~\ref{sec:Germs} through \ref{subsec:Normal-Forms-andInvariants}
set up the framework and central questions for local dynamics, based
on the first section of Abate's comprehensive survey article \cite{Abate2010Discreteholomorphiclocaldynamicalsystems}.
Then we define normality as a measure of dynamical stability in global
dynamics in Section~\ref{sec:FatouDef}, leading to a decomposition
of the domain into the Fatou set of points with locally stable dynamics
and the Julia set of points with chaotic behaviour. 

\selectlanguage{british}
\global\long\def\Transp{\mathsf{T}}%
\global\long\def\Pow{\m{Pow}}%
\global\long\def\ord{\m{ord}}%

\section{\label{sec:Germs}Germs}

For the rest of this chapter let $M$ and $N$ be complex manifolds
and $p\in M$ and $q\in N$. Given an open subset $U\subseteq M$
and a holomorphic self map $F:U\to M$ fixing a point $p\in U$, we
are concerned only with the local dynamics close to the fixed point
$p$. That is to say, we wish to freely replace $F$ by its restriction
to suitable neighbourhoods of $p$. This is made rigorous by passing
to germs.
\begin{defn}
Let $M$ and $N$ be complex manifolds (or topological spaces) and
$p\in M$. Let $U$ and $V$ be open neighbourhoods of $p$ in $M$.
Then we call continuous maps $F:U\to N$ and $G:V\to N$ \emph{equivalent}
if they coincide on some (possibly small) open neighbourhood $W\subseteq U\cap V$
of $p$. The equivalence classes of this relation are called \emph{germs\index{germ of maps!at a point}
(of continuous maps)} or \emph{local (continuous) maps\index{local map!at a point}
}from $M$ to $N$ at $p$. An element of such an equivalence class
is a \emph{representative\index{representative@\emph{representative}}}
of the containing germ.
\end{defn}

A germ inherits all local properties of its representatives, in particular
their regularity (near $p$), so it makes sense to talk about \emph{local
holomorphic maps\index{local holomorphic map}} or \emph{holomorphic
germs}\index{holomorphic germ}.
\begin{notation}
By abuse of notation, we use the same symbol for a germ $F$ and a
representative $F:U\to N$. If we want to emphasise the specific representative
on $U$, we write $F|_{U}$. We also denote a germ $F$ at $p$ by
$F:(M,p)\to N$ or $F:(M,p)\to(N,q)$, where $q=F(p)\in N$ and we
let $\End(M,p)$ denote the set of holomorphic germs $F:(M,p)\to(M,p)$
fixing a point $p\in M$ and $\Aut(M,p)\subseteq\End(M,p)$ the subset
of locally biholomorphic germs.
\end{notation}

\section{Local Dynamics}

For a germ $F\in\End(M,p)$ we are interested in the dynamics around
the fixed point $p$, that is, in the behaviour of the iterates $F^{\circ n}$
for large $n$. Clearly, arbitrary iterates are defined at $p$, but
not necessarily on any fixed neighbourhood of $p$.

\begin{defn}
Let $F:U\to M$ be a representative of $F\in\End(M,p)$.
\begin{enumerate}[noitemsep]
\item The (\emph{forward})\emph{ orbit\index{orbit} }of a point $q\in U$
under $F$ is the sequence 
\[
\{q_{n}\}_{n}:=\{F^{\circ n}(q)\mid n\in\mathbb{N},F^{\circ(n-1)}(q)\in U\}.
\]
An orbit $\{q_{n}\}_{n}$ is \emph{stable\index{stable orbit@\emph{stable orbit}}}
(in $U$) or \emph{\index{non-escaping orbit@\emph{non-escaping orbit}}non-escaping}
(from $U$), if it is infinite (or $\{q_{n}\}_{n}\subseteq U$\}).
An orbit $\{q_{n}\}_{n}$ \emph{escapes\index{escaping orbit@\emph{escaping orbit}}}
$U$ if it is finite (or $\{q_{n}\}_{n}\not\subseteq U$).

\item The set 
\[
\Sigma_{F}:=\Sigma_{F}(U):=\bigcap_{j=0}^{\infty}F^{\circ(-j)}(U)=\{q\in U\mid F^{\circ n}(q)\in U\text{ for all }n\in\mathbb{N}\}
\]
of all stable orbits in $U$ is called the \emph{(local) stable set\index{stable set}}
of $F$ (in $U$).
\item The set 
\[
A_{F}:=A_{F}(p):=A_{F}(p,U):=\{q\in\Sigma_{F}(U)\mid F^{\circ n}(q)\xrightarrow[n\to\infty]{}p\}
\]
 of all orbits in $U$ converging to $p$ is called the \emph{realm
of attraction\index{realm of attraction@\emph{realm of attraction}}}
of $F$ (to $p$ in $U$).
\item A subset $V\subseteq U$ is \emph{invariant\index{invariant set}}
under $F$ (or \emph{$F$-invariant), }if $F(V)\subseteq V$, and
\emph{completely invariant\index{completely invariant set}} under
$F$ (or \emph{completely} \emph{$F$-invariant), }if $F^{-1}(V)=V$.

\end{enumerate}
\end{defn}

\begin{rem}
The stable set $\Sigma_{F}(U)$ and the realm of attraction $A_{F}(p)$
are $F$-invariant and completely $F|_{U}$-invariant.
\end{rem}

\begin{rem}[Terminology]
In hyperbolic settings, the realm of attraction $A_{F}$ is commonly
called the \emph{(global) stable set}, as for small enough $U$, it
coincides with the local stable set $\Sigma_{F}(U)$ (see Theorem~\ref{thm:StableMnfThm}).
$A_{F}$ is also called the \emph{basin of attraction }or \emph{stable
manifold }when it is an open set or a lower dimensional submanifold
respectively (see Theorems~\ref{thm:StableMnfThm}. We adopt the
term ``realm of attraction'', that J.~Milnor introduced in \cite{Milnor1985OntheConceptofAttractor}
to avoid confusion.
\end{rem}

At first glance, all interesting local dynamics happen in the set
$\Sigma_{F}(U)$, as points outside of it have finite orbits, but
additional information can be obtained from their possible backward
orbits:
\begin{defn}
Let $F:U\to M$ be a representative of $F\in\End(M,p)$.
\begin{enumerate}[noitemsep]
\item A \emph{backward orbit\index{backward orbit@\emph{backward orbit}}
of $q\in M$ under $F$ is a sequence} $\{q_{-k}\}_{k\in\mathbb{N}}$
with $q_{0}=q$ and $q_{-k+1}=F(q_{-k})$ for all $k\in\mathbb{N}$.
\item The \emph{unstable set}\index{unstable set} $\Sigma_{F}^{-}=\Sigma_{F}^{-}(U)$
of $F$ (in $U$) is the set of all $q\in U$ that admit a (infinite)
backward orbit contained in $U$.
\item The \emph{realm of repulsion\index{realm of repulsion@\emph{realm of repulsion}}}
$A_{F}^{-}$ of $F$ (from $p$ in $U$) is the set of all $q\in\Sigma_{F}^{-}$
that admit a backward orbit $\{q_{-k}\}_{k\in\mathbb{N}}\subseteq U$
such that $q_{-k}\xrightarrow[k\to+\infty]{}0$.
\end{enumerate}
\end{defn}

\begin{rem}
If $F$ is invertible, then the unstable set $\Sigma_{F}^{-}$ is
the stable set $\Sigma_{F^{-1}}$ of the inverse $F^{-1}$ and the
realm of repulsion $A_{F}^{-}$ is the realm of attraction $A_{F^{-1}}$
of the inverse $F^{-1}$.
\end{rem}

Points outside $\Sigma_{F}(U)$ and $\Sigma_{F}^{-}(U)$ have finite
forward and backward orbits, so there is no long-term dynamical behaviour
to study, but nevertheless, weather such points exists can be an interesting
question.

The topology and orbit behaviour of the sets $\Sigma_{F}$ and $A_{F}$
as well as $\Sigma_{F}^{-}$ and $A_{F}^{-}$ is of central interest
in discrete local holomorphic dynamics.

$\Sigma_{F}(U)$ depends strongly on the neighbourhood $U$ and we
will be concerned with features that persist for arbitrarily small
neighbourhoods. For example, for an entire function $F:\mathbb{C}\to\mathbb{C}$,
the stable set with respect to $\mathbb{C}$ is just $\Sigma_{F}(\mathbb{C})=\mathbb{C}$,
which does not give us any local information.

\section{Conjugation}

A first step towards understanding the dynamics of a germ $F\in\End(M,p)$,
is to conjugate $F$ to a germ of a simpler form, such as a linear
self-map of $\mathbb{C}^{d}$. There are three main types of conjugacy
we discuss here, following Section 1 of \cite{Abate2010Discreteholomorphiclocaldynamicalsystems}.
\begin{defn}
Let $F\in\End(M,p)$ and $G\in\End(N,q)$. Then $F$ and $G$ are
called \emph{holomorphically (topologically, smoothly) conjugate}\index{conjugation},
if there exists a local biholomorphism (homeomorphism, (real) diffeomorphism)
$\varphi:(M,p)\to(N,q)$ such that $F=\varphi^{-1}\circ G\circ\varphi$.
\end{defn}

We can immediately observe that if $F$ and $G$ are conjugate, the
local dynamics of $F$ and $G$ (as described in the previous section)
are equivalent as conjugation preserves topological properties. Moreover,
taking a local chart $\varphi:(M,p)\to(\mathbb{C}^{d},0)$ centred
at $p$, any holomorphic germ $F\in\End(M,p)$ is holomorphically
conjugate to a germ $G\in\End(\mathbb{C}^{d},0)$, where $d=\dim M$.
Hence, for our purposes, it is enough to consider germs in $\End(\mathbb{C}^{d},0)$
or, equivalently, ($d$-tuples of) convergent power series without
constant term. Therefore, we can consider $\End(\mathbb{C}^{d},0)$
as a subspace of the space $\Pow(\mathbb{C}^{d},0)$ of ($d$-tuples
of) formal power series with vanishing constant terms. In particular,
we can conjugate via elements $\Phi\in\Pow(\mathbb{C}^{d},0)$ that
are invertible with respect to formal composition of power series
($\Phi$ is invertible if and only if its linear part is). We also
call an element $\Phi\in\Pow(\mathbb{C}^{d},0)$ a \emph{\index{formal germ}formal
germ}.
\begin{defn}
Two formal germs $F,G\in\Pow(\mathbb{C}^{d},0)$ are \emph{formally
conjugate}\index{conjugation!formal}, if there exists an invertible
formal power series $\Phi\in\Pow(\mathbb{C}^{d},0)$ such that $F=\Phi^{-1}\circ G\circ\Phi$.
\end{defn}

Clearly, if $F$ and $G$ are holomorphically conjugate germs, they
are both topologically and formally conjugate as well. On the other
hand, neither formal nor topological conjugation imply the other (compare
Theorem~\ref{thm:GrobmanHartman} versus Theorem~\ref{thm:PoincareDulac}).

Note that unlike topological conjugation, formal conjugation does
not necessarily preserve dynamical properties. However, formal conjugates
are often easier to construct systematically. From there, one can
try to prove convergence of the formal conjugating series. If this
is not possible, one can always truncate the conjugating series according
to the following lemma:
\begin{lem}
\label{lem:FormalConjIffAnyOrderPolyConj}Two germs $F,G\in\Pow(\mathbb{C}^{d},0)$
are formally conjugate if and only if they are polynomially conjugate
up to every finite order, i.e.\ for every $k\in\mathbb{N}$ there
exists a polynomial $P_{k}\in\Aut(\mathbb{C}^{d},0)$ such that 
\begin{equation}
P_{k}\circ F\circ P_{k}^{-1}(z)=G(z)+O(\norm z^{k})\label{eq:conjToHighOrders}
\end{equation}
 for $z\in\mathbb{C}^{d}$ near $0$.
\end{lem}

Hence, formal conjugacy ensures holomorphic conjugacy up to arbitrarily
high orders. To derive dynamical properties of $F\in\End(\mathbb{C}^{d},0)$
from those of a formal conjugate $G$, we have to control the tail
term $O(\norm z^{k})$ in (\ref{eq:conjToHighOrders}), for example
via analytic estimates. 

\section{\label{subsec:Normal-Forms-andInvariants}Normal Forms and invariants}

These three notions of conjugacy each induce an equivalence relation.
The \emph{\index{classification problem@\emph{classification problem}}classification
problem} for an equivalence relation $\sim$ on a class $\mathcal{C}$
consists in describing the quotient class $\mathcal{C}/\sim$.

A \emph{normal form}\index{normal form!concept} of an object $A$
of $\mathcal{C}$ is a selected representative of the same equivalence
class. To get useful classification results, we would like a normal
form to have 
\begin{elabeling}{(NF3)00}
\item [{(NF1)}] Simple form,
\item [{(NF2)}] Uniqueness in its equivalence class.
\end{elabeling}
In local holomorphic dynamics, we are dealing with the class $\End(\mathbb{C}^{d},0)$
of holomorphic germs fixing $0$. So for (NF1) the best case would
be linear or at least polynomial maps.
\begin{defn}
We say $F\in\End(\mathbb{C}^{d},0)$ is \emph{formally\index{formally linearisable germ},
topologically\index{topologically linearisable germ} }or\emph{ holomorphically
linearisable}\index{holomorphically linearisable germ}, if $F$ is
formally, topologically or holomorphically conjugated to its linear
part $dF_{0}$ respectively.
\end{defn}

Unfortunately, a general germ is not linearisable with respect to
any of our conjugations, and more complicated normal forms are often
not unique. Since such a classification quickly gets too complicated,
one may start looking for invariants instead: 
\begin{defn}
A \emph{\index{system of invariants}system of invariants} $\mathcal{S}$
for $\mathcal{C}$ is a class $\mathcal{S}(A)$ of objects (\emph{invariants}\index{invariant!concept})
associated to each object $A$ of $\mathcal{C}$, such that two objects
of the same equivalence class have the same invariants. A system of
invariants $\mathcal{S}$ for $\mathcal{C}$ is \emph{complete}, if
two objects $A$ and $B$ of $\mathcal{C}$ are equivalent if and
only if they have the same invariants $\mathcal{S}(A)=\mathcal{S}(B)$.
\end{defn}

Again, we would like our system of invariants to have:
\begin{elabeling}{(I3)0}
\item [{(I1)}] Small classes $\mathcal{S}(A)$,
\item [{(I2)}] (Easily) computable invariants,
\item [{(I3)}] Completeness
\end{elabeling}
Finding complete systems of invariants for all of $\End(\mathbb{C}^{d},0)$
in general is hard, just like finding general normal forms. We instead
look at reasonable results on suitable subclasses. Normal forms and
invariants show up in most results on holomorphic dynamics in one
way or another.

\section{\label{sec:FatouDef}Normality and the Fatou set}

\selectlanguage{british}
\global\long\def\Transp{\mathsf{T}}%
\global\long\def\Pow{\m{Pow}}%
\global\long\def\ord{\m{ord}}%

In the global study of iterated holomorphic self-maps of a complex
manifold, an important notion of stability of orbits is provided by
normality of the sequence $\{F^{\circ n}\}_{n\in\mathbb{N}}$.
\begin{defn}
\label{def:normality}Let $X$ and $Y$ be metric spaces. A family
$\mathcal{F}\subseteq\mathcal{C}(X,Y)$ of continuous maps from $X$
to $Y$ is \emph{normal} (in\emph{ $\mathcal{C}(X,Y)$)} \emph{at
$x\in X$}, if every sequence $\{f_{n}\}_{n}\subseteq\mathcal{F}$
admits an open neighbourhood $U\subseteq X$ of $x$ and a subsequence
$\{f_{n_{j}}\}_{j}$ that either 
\begin{enumerate}[noitemsep]
\item converges uniformly on $U$ to a continuous function $F:U\to Y$
or
\item diverges uniformly from $Y$ on $U$, i.e.\ for every compact $K\subseteq Y$,
we have $f_{n_{j}}(U)\cap K=\emptyset$ for $j$ sufficiently large.
\end{enumerate}
The family $\mathcal{F}$ is a \emph{normal}\index{normal family@\emph{normal family}}
in\emph{ $\mathcal{C}(X,Y)$}, if $\mathcal{F}$ is normal at every
point $x\in X$.
\end{defn}

\begin{rem}[Characterisations]
\label{rem:CharNormality}For locally compact spaces, locally uniform
convergence is equivalent to compact convergence and convergence in
the compact-open topology on $\mathcal{C}(X,Y)$. Hence normality
is a topological property that does not depend on the metrics on $X$
and $Y$, but only on the induced topologies.

If $X\subseteq Y$ and $Y$ is compact, then for $F\in\mathcal{C}(X,Y)$,
the family $\mathcal{F}=\{F^{\circ n}\}_{n}$ is normal, if and only
if orbits of nearby points stay close to each other for all time.
Conversely, $\{F^{\circ n}\}_{n}$ is not normal at a point $x\in X$,
if and only if small changes of the starting point $x$ can lead to
arbitrarily large changes of the orbit, in other words, orbits display
sensitive dependence on the starting point or the dynamics are \emph{chaotic.}
\end{rem}

If $X$ and $Y$ are complex manifolds and $\mathcal{F}\subseteq\mathcal{C}(X,Y)$
is a normal family of holomorphic mappings, then Weierstraß' convergence
theorem ensures that all locally uniform limit functions are again
holomorphic. In the holomorphic category, a fundamental criterion
for normality is Montel's compactness principle:
\begin{thm}[Montel \cite{Montel1927LeconsSurLesFamillesNormalesDeFonctionsAnalytiquesEtLeursApplicationsRecueilliesEtRedigeesParJBarbotte}]
\label{thm:MontelBdd}Let $M$ be a complex manifold and $\mathcal{F}\in\mathcal{C}(M,\mathbb{C}^{d})$
a family of holomorphic maps from $M$ to $\mathbb{C}^{d}$. Then
every sequence $\{f_{n}\}_{n}\subseteq\mathcal{F}$ admits a locally
uniformly convergent subsequence, if and only if $\mathcal{F}$ is
locally uniformly bounded in $M$.
\end{thm}

This implies an immediate relation between normality of iterates and
stability of orbits:
\begin{cor}
\label{cor:StabilityAndNormality}Let $F\in\End(M,p)$. Then $\{F^{\circ n}\}_{n}$
is normal at $p$ if and only if the fixed point $p$ is stable, i.e.\ the
stable set $\Sigma_{F}(U)$ contains a neighbourhood of $p$ for any
neighbourhood $U\subseteq M$ of $p$.
\end{cor}

More generally, for non-fixed points, we get:
\begin{cor}
\label{cor:StableCompFatou}Let $F:U\to\mathbb{C}^{d}$ be a representative
of $F\in\End(\mathbb{C}^{d},0)$ on a bounded domain $U\subseteq\mathbb{C}^{d}$.
Then $\{F^{\circ n}\}_{n}$ is normal at $x\in U$, if and only if
$x$ is in the interior of the stable set $\Sigma_{F}(U)$.
\end{cor}

Normality of local orbits near a fixed point is hence completely determined
be the topology of the stable set. In global dynamics, this is no
longer true in general and a fundamental task is finding the domain
of normality of the sequence of iterates:
\begin{defn}
Let $X$ be a metric space and $F\in\mathcal{C}(X,X)$. Then the \emph{Fatou
set\index{Fatou set@\emph{Fatou set}}} of $F$ is the set of all
points $x\in X$ such that $\{F^{\circ n}\}_{n\in\mathbb{N}}$ is
normal at $x$. Its connected components are the \emph{Fatou components
}of\emph{ }$F$. The \emph{Julia set\index{Julia set@\emph{Julia set}}}
of $F$ is the complement in $X$ of the Fatou set of $F$.
\end{defn}

Via Remark~\ref{rem:CharNormality}, the Fatou set contains the points
of stable dynamics, in that the orbits of nearby points in the same
Fatou component stay close to each other (or uniformly large) for
all time, whereas on a neighbourhood of a point in the Julia set,
orbits display sensitive dependence on initial conditions or chaotic\emph{
}behaviour\emph{.}
\begin{rem}
The Fatou set is open and the Julia set is closed by definition. The
Fatou and Julia sets are the same for any iterate $F^{\circ k}$ with
$k\in\mathbb{N}_{>0}$. The Fatou set is $F^{-1}$-invariant, the
Julia set $F$-invariant and if $F$ is an open map, then both sets
are completely invariant. Holomorphic maps in one variable are always
open, unless they are constant. In several variables, this is the
case for locally injective maps and endomorphisms of projective space.
\end{rem}

\begin{rem}
\label{rem:StableCompFatou}Let $F\in\End(\mathbb{C}^{d})$ fix a
point $p\in\mathbb{C}^{d}$, and let $U\subseteq\mathbb{C}$ be a
bounded neighbourhood of $p$. Then, by Corollaries~\ref{cor:StabilityAndNormality}
and \ref{cor:StableCompFatou}, any connected open subset of the stable
set $\Sigma_{F}(U)$ is contained in a Fatou component of $F$. The
dynamics on these parts of the stable set are relatively well understood.
The structure of non-escaping orbits in the Julia set can be much
more complicated (see e.g. Theorem~\ref{thm:PerezMarcoCremerHedgehog}).

Conversely, if an $F$-invariant Fatou component $\Omega\subseteq\mathbb{C}^{d}$
is attracting to $p$, i.e. $F^{\circ n}\xrightarrow[n\to\infty]{}p$
uniformly on $\Omega$, then there exists an open subset $V\subseteq\Sigma_{F}(U)$
such that $\Omega=\bigcup_{n=0}^{\infty}F^{\circ(-n)}(V)$. In this
way, all attracting Fatou components can be identified locally.
\end{rem}

\cleardoublepage{}

\chapter{\label{sec:Local1Ddyn}One complex variable}

\selectlanguage{british}
\global\long\def\Pow{\m{Pow}}%
\global\long\def\Res{\m{Res}}%
\global\long\def\Jac{\m{Jac}}%
\global\long\def\Transp{\mathsf{T}}%

\section{The Multiplier and formal linearisation}

The local dynamics of a holomorphic map $f\in\End(\mathbb{C},0)$
are determined at the first order by the derivative of $f$, which
will be our first important invariant:
\begin{defn}
Let $f\in\Pow(\mathbb{C},0)$ be given by 
\[
f(z)=\lambda z+f_{2}z^{2}+f_{3}z^{3}+\cdots.
\]
Then the number $\lambda\in\mathbb{C}$ is called the \emph{multiplier\index{multiplier@\emph{multiplier}}}
of $f$ (at $0$). The \emph{order\index{order@\emph{order}} }of
$f$ (at $0$) is $\inf\{j\ge2\mid f_{j}\neq0\}\in\mathbb{N}\cup\{\infty\}$.
\end{defn}

\begin{notation}
For $\lambda\in\mathbb{C}$ let $\Pow_{\lambda}(\mathbb{C},0)$, $\End_{\lambda}(\mathbb{C},0)$,
and $\Aut_{\lambda}(\mathbb{C},0)$ denote the subsets of elements
with multiplier $\lambda$ in $\Pow(\mathbb{C},0)$, $\End(\mathbb{C},0)$,
and $\Aut(\mathbb{C},0)$ respectively.
\end{notation}

If $f\in\End(\mathbb{C},0)$, then the multiplier is just the derivative
$\lambda=f'(0)$. That is, for $z$ near $0$ the best linear approximation
of $f(z)$ is $\lambda z$ and the best linear approximation of the
orbit $\{f^{\circ n}(z)\}_{n}$ is given by $\lambda^{n}z$. Hence
our first classification is according to the properties of $\lambda^{n}$
for $n\in\mathbb{N}$:
\begin{defn}
\label{def:1DMultiplierTypes}Let $f\in\Pow_{\lambda}(\mathbb{C},0)$.
Then the fixed point $0$ and, for convenience, also the multiplier
$\lambda$ and the (formal) germ $f$, are called
\begin{enumerate}[noitemsep]
\item \emph{hyperbolic}\index{hyperbolic fixed point/germ}, if $|\lambda|\neq1$,
\begin{enumerate}[noitemsep]
\item \emph{attracting}\index{attracting fixed point/germ}, if $|\lambda|<1$,
\begin{enumerate}[noitemsep]
\item \emph{geometrically attracting}, if $0<|\lambda|<1$.
\item \emph{superattracting}\index{superattracting fixed point/germ}, if
$\lambda=0$.
\end{enumerate}
\item \emph{repelling}\index{repelling fixed point/germ}, if $|\lambda|>1$.
\end{enumerate}
\item \emph{neutral}\index{neutral fixed point/germ}, if $|\lambda|=1$,
\begin{enumerate}[noitemsep]
\item \emph{parabolic}\index{parabolic fixed point/germ}, if $\lambda$
is a root of unity.
\item \emph{elliptic}\index{elliptic fixed point/germ}, if $|\lambda|=1$,
but $\lambda$ is not a root of unity.
\end{enumerate}
\end{enumerate}
A major part of the formal classification is covered by the Poincaré-Dulac
theorem on formal elimination of non-resonant terms (see Definition~\ref{def:resonance}):
\end{defn}

\begin{prop}[Poincaré \cite{Poincare1879SurLesProprietesDesFonctionsDefiniesParLesEquationsAuxDifferencesPartielles},
Dulac \cite{Dulac1904RecherchesSurLesPointsSinguliersDesEquationsDifferentielles}]
\label{prop:1DPoincareDulac}Let $f\in\Pow(\mathbb{C},0)$ with multiplier
$\lambda$. Then $f$ is formally conjugated to a power series $g\in\Pow(\mathbb{C},0)$
with multiplier $\lambda$ in normal form, that is, $g$ contains
only monomials $z^{j}$ for $j\ge1$ with $\lambda^{j}=\lambda$.
The conjugating series $h\in\Pow(\mathbb{C},0)$ is uniquely determined
by prescribing (arbitrary) coefficients for the monomials $z^{j}$
for $j\ge1$ with $\lambda^{j}=\lambda$. This form of $g\in\Pow(\mathbb{C},0)$
is preserved by conjugation with $h\in\Pow(\mathbb{C},0)$, if and
only if $h$ also contains only monomials $z^{j}$ for $j\ge1$ with
$\lambda^{j}=\lambda$.
\end{prop}

\begin{proof}
A series $h\in\Pow(\mathbb{C},0)$ conjugates $f$ to $g$ if it is
invertible solution of the \emph{homological equation\index{homological equation@\emph{homological equation}}}
\[
f\circ h=h\circ g.
\]
With expansions $f(z)=\sum_{j=1}^{\infty}f_{j}z^{j}$, $g(z)=\sum_{j=1}^{\infty}g_{j}z^{j}$
and $h(z)=\sum_{j=1}^{\infty}h_{j}z^{j}$ for $f$, $g$, and $h$,
the homological equation becomes 
\begin{equation}
\sum_{j=1}^{\infty}f_{j}\paren[\bigg]{\sum_{k=1}^{\infty}h_{k}z^{k}}^{j}=\sum_{j=1}^{\infty}h_{j}\paren[\bigg]{\sum_{k=1}^{\infty}g_{k}z^{k}}^{j}\label{eq:1DhomolEqExpl}
\end{equation}
and comparing coefficients of $z^{j}$ we have $f_{1}h_{1}=h_{1}g_{1}$
for $j=1$ and for $j\ge2$ 
\begin{equation}
(\lambda^{j}-\lambda)h_{j}=h_{1}^{j}f_{j}-h_{1}g_{j}+\sum_{2\le k<j}\sum_{l_{1}+\cdots+l_{k}=j}(f_{k}h_{l_{1}}\cdots h_{l_{k}}-h_{k}g_{l_{1}}\cdots g_{l_{k}}).\label{eq:homol1DexplByDegree}
\end{equation}
For each $j\ge1$, all indices $k,l_{1},\ldots,l_{k}$ in (\ref{eq:homol1DexplByDegree})
are smaller than $j$ and $f_{j}$ is given. Hence iteratively for
$j\ge1$, if $\lambda^{j}=\lambda$, we can choose any value of $h_{j}$,
(resulting in a possibly non-zero value of $g_{j}$), and if $\lambda^{j}\neq\lambda$,
we can prescribe $g_{j}=0$ on the right hand side and solve (\ref{eq:homol1DexplByDegree})
uniquely for $h_{j}$.
\end{proof}
\begin{rem}
\label{rem:multiplierIsInvariant}Proposition~\ref{prop:1DPoincareDulac}
and its proof give some immediate insights into the formal and holomorphic
classification:
\begin{enumerate}
\item The multiplier is a formal and holomorphic invariant. This is just
the linear part of (\ref{eq:1DhomolEqExpl}). 
\item While topological conjugation does not preserve the multiplier in
general, the dynamical properties will imply that it at least preserves
the classification of Definition~\ref{def:1DMultiplierTypes}.
\end{enumerate}
\end{rem}

\begin{cor}
\label{cor:1DformalLinearisation}Let $f\in\End(\mathbb{C},0)$ with
multiplier $\lambda$. If $\lambda^{j}\neq\lambda$ for all $j\ge2$
(i.e.~$f$ is not parabolic or superattracting), then $f$ is formally
linearisable.
\end{cor}

\begin{rem}
If $\lambda^{j}\neq\lambda$ for some $j\ge2$ (i.e.~$\lambda\neq0,1$),
then the order $\min\{k\ge2\mid f_{k}\neq0\}$ is not preserved under
formal or holomorphic conjugation, as for any given value of $g_{2}$
(zero or non-zero) we can solve (\ref{eq:homol1DexplByDegree}) for
$h_{2}$. On the other hand, if $\lambda=1$ the order is equal to
the \emph{multiplicity\index{multiplicity@\emph{multiplicity}}} (of
the fixed point $0$) and if $\lambda=0$ to the \emph{local degree\index{local degree@\emph{local degree}}}
(of $f$ near $0$), both of which are formal, holomorphic and topological
invariants.
\end{rem}

\section{Hyperbolic fixed points}

The hyperbolic case is the generic case and also the most easily understood.
The basic dynamics can be derived quickly from the multiplier. In
fact, these topological dynamics characterise attracting and repelling
germs:
\begin{lem}
\label{lem:1DhypDynamics} A germ $f\in\End(\mathbb{C},0)$ with multiplier
$\lambda$ is attracting ($|\lambda|<1$) if and only if it is topologically
attracting, i.e. there exists a neighbourhood $U\subseteq\mathbb{C}^{d}$
of $0$ such that $f^{\circ n}\xrightarrow[n\to\infty]{}0$ uniformly
on $U$.  It is repelling ($|\lambda|>1$) if and only if it is
topologically repelling, i.e. there exists a neighbourhood $U\subseteq\mathbb{C}^{d}$
of $0$ such that all orbits except $\{0\}$ escape from $U$ (i.e.\ $\Sigma_{f}(U)=\{0\}$).
\end{lem}

\begin{proof}
Let $f$ be attracting. We have $f(z)=\lambda z+O(z^{2})$ near $0$.
Hence there exists $C>0$ such that for $|z|$ small enough, we have
\begin{equation}
|f(z)|\le|\lambda|z+C|z|^{2}.\label{eq:attr1D}
\end{equation}
Let $r>0$ small enough that (\ref{eq:attr1D}) holds for $z$ in
the disc $\mathbb{D}_{r}$ of radius $r$ around $0$ and such that
$|\lambda|+Cr<0$. Then for $z\in\mathbb{D}_{r}$, we have 
\[
|f(z)|\le(|\lambda|+Cr)|z|<|z|,
\]
hence $f(z)\in\mathbb{D}_{r}$ and by induction on $n\in\mathbb{N}$,
we have 
\[
|f^{\circ n}(z)|\le(|\lambda|+Cr)^{n}|z|\xrightarrow[n\to\infty]{}0
\]
showing uniform convergence to $0$ on $\mathbb{D}_{r}$.

Conversely, let $\mathbb{D}_{r}\subseteq\mathbb{C}^{d}$, $r>0$ such
that $f^{\circ n}\xrightarrow[n\to\infty]{}0$ uniformly on $\mathbb{D}_{r}$.
Then for some $n\in\mathbb{N}$, $f^{\circ n}(\mathbb{D}_{r})\subseteq\mathbb{D}_{r/2}$
and by Schwarz's lemma $|\lambda^{n}|<1$.

If $|\lambda|>1$, then $f$ is invertible and $f^{-1}$ is attracting.
If $f$ is topologically repelling, then the above shows $\lambda\neq0$,
so $f$ is invertible and by the following topological Lemma~\ref{lem:TopRepellIsInversTopAttr},
$f^{-1}$ is topologically attracting. Now the repelling case follows
from the attracting one.

\end{proof}
\begin{lem}
\label{lem:TopRepellIsInversTopAttr}A germ of homeomorphisms $g:(X,p)\to(X,p)$
of a locally compact metric space $X$ is topologically repelling
if and only if its inverse $g^{-1}:(X,p)\to(X,p)$ is topologically
attracting.
\end{lem}

\begin{proof}
Let $g$ be topologically repelling and $U\subseteq X$ a neighbourhood
of $0$ such that all orbits in $U\backslash\{0\}$ escape. Let $K\subseteq U$
be a compact neighbourhood of $p$. Then 
\[
K_{n}:=K\cap g^{\circ(-1)}(K)\cap\cdots\cap g^{\circ(-n)}(K),\quad n\in\mathbb{N}
\]
defines a decreasing nested sequence of such neighbourhoods. Since
all orbits in $K\backslash\{p\}$ escape, we have $\bigcap_{n=0}^{\infty}K_{n}=\{p\}$
and hence $\m{diam}(K_{n})\xrightarrow[n\to\infty]{}p$. Thus for
$n\in\mathbb{N}$ large enough, $K_{n}\subseteq g(K)$ and 
\[
g(K_{n+1})=K_{n}\cap g(K)=K_{n}
\]
and we can assume $g$ is bijective on $K_{n}$. Hence $g^{\circ(-m)}(K_{n})=K_{n+m}$
for $m\in\mathbb{N}$ and $g^{\circ(-m)}(x)\to p$ uniformly on the
neighbourhood $K_{m}\subseteq U$ of $0$.

Conversely, let $g^{-1}$ be topologically attracting and $U\subseteq X$
a neighbourhood of $p$ such that $g^{\circ(-n)}\to p$ uniformly
on $U$. Let $x\in U$ and assume $g^{\circ n}(x)\in U$ for all $n$.
Then for every $\varepsilon>0$ there is an $n$ large enough so that
$g^{\circ(-n)}(U)\subseteq B_{\varepsilon}(p)$ and hence 
\[
x=g^{\circ(-n)}(g^{\circ n}(x))\in B_{\varepsilon}(p).
\]
Thus $x=p$ and all other $g$-orbits escape from $U$.
\end{proof}

Moreover, this correspondence of inverses immediately suggests close
relations of normal forms in the repelling and geometrically attracting
case. For the superattracting case on the other hand, we can already
see the difference in the formal classification in Remark~\ref{rem:multiplierIsInvariant}.
The hyperbolic invertible case was the first for which unique holomorphic
and topological normal forms were found by Koenigs:
\begin{thm}[Koenigs, \cite{Koenigs1884RecherchesSurLesIntegralesDeCertainesEquationsFonctionnelles}]
\label{thm:Koenigs}Let $f\in\Aut(\mathbb{C},0)$ be hyperbolic and
invertible with multiplier $\lambda$. Then:
\begin{enumerate}
\item $f$ is holomorphically conjugate to its linear part $L_{\lambda}:z\mapsto\lambda z$
and the conjugation $\varphi\in\Aut(\mathbb{C},0)$ is uniquely determined
up to multiplication by a non-zero constant.
\item $f$ is topologically conjugate to the map $L_{1/2}:z\mapsto z/2$,
if $|\lambda|<1$ and the map $L_{2}:z\mapsto2z$, if $|\lambda|>1$.
\end{enumerate}
\end{thm}

\begin{proof}[Proof sketch]
Proposition~\ref{prop:1DPoincareDulac} provides a linearising series
$\varphi\in\Pow(\mathbb{C},0)$ for $f$, that is unique up to choosing
the (non-zero) linear part (since $\lambda^{j}=\lambda$ if and only
if $j=1$). The convergence of $\varphi$ follows from the more general
Theorem~\ref{thm:BrjunoYoccoz1D} of Brjuno.

For a more direct proof in the geometrically attracting case, let
$\varphi_{k}:=f^{\circ k}/\lambda^{k}$ for $k\in\mathbb{N}$. Then
for $r>0$ small enough, one can show that the sequence $\{\varphi_{k}\}_{k\in\mathbb{N}}$
converges uniformly on $\mathbb{D}_{r}$ to a holomorphic map $\varphi$.
By definition for $k\in\mathbb{N}$ we have $\varphi_{k}\circ f=L_{\lambda}\circ\varphi_{k+1}$,
and passing to the limit $k\to\infty$ shows that the limit map $\varphi$
conjugates $f$ to $L_{\lambda}$.

The topological conjugation from $L_{\lambda}$ to $L_{1/2}$ can
be constructed by glueing suitable homeomorphisms of annuli covering
a neighbourhood of $0$ (see e.g. \cite{Abate2010Discreteholomorphiclocaldynamicalsystems}).
\end{proof}
\begin{rem}
The method of composing the $k$-th iterate of the map $f$ with the
$k$-th iterate of a (right-) inverse of the target map $g$ and showing
convergence for $k\to\infty$ to arrive at a conjugating local biholomorphism
is an important technique that we will see in more detail in the construction
of Fatou coordinates.
\end{rem}

\begin{rem}
The normal form $L_{\lambda}$ in the holomorphic category is unique
by Remark~\ref{rem:multiplierIsInvariant}. The two possible topological
normal forms can not be topologically conjugated, since the stable
set $\Sigma_{L_{1/2}}$ is a neighbourhood of $0$, but $\Sigma_{L_{2}}=\{0\}$
by Lemma~\ref{lem:1DhypDynamics}.

Hence invertible hyperbolic germs are uniquely determined up to holomorphic
conjugation by their multiplier $\lambda$ and up to topological conjugation
by the dichotomy $0<|\lambda|<1$ or $|\lambda|>1$. This gives us
a complete set of invariants in both categories.
\end{rem}

The classification in the superattracting case was first stated by
Böttcher in \cite{Boettcher1899PrinciplesofIterationalCalculuspartOneandTwo},
better known from \cite{Boettcher1904ThePrincipalLawsofConvergenceofIteratesandTheirApplicationtoAnalysis},
where he gives a sketch of the proof. First complete proofs were given
independently in \cite{Fatou1919SurLesEquationsFonctionnellesI}
and \cite{Ritt1920OntheIterationofRationalFunctions}.
\begin{thm}[Böttcher \cite{Boettcher1899PrinciplesofIterationalCalculuspartOneandTwo}]
\label{thm:Boettcher}Let $f\in\End(\mathbb{C},0)$ be superattracting
of order $2\le k<\infty$, i.e.
\[
f(z)=a_{k}z^{k}+a_{k+1}z^{k+1}+\cdots
\]
with $a_{k}\neq0$. Then:
\begin{enumerate}
\item $f$ is holomorphically conjugate to $E_{k}:z\mapsto z^{k}$ and the
conjugation $\varphi\in\Aut(\mathbb{C},0)$ is unique up to multiplication
by $k-1$-st roots of unity.
\item For $k,l\in\mathbb{N}$, $E_{k}$ and $E_{l}$ are topologically conjugate
if and only if $k=l$.
\end{enumerate}
\end{thm}

\begin{proof}[Proof sketch]
After a linear change of coordinates, we can assume $a_{k}=1$, so
\[
f(z)=z^{k}g_{1}(z),
\]
where $g_{1}(z)=1+O(z)$. For $r>0$ small enough, $g_{1}$ never
vanishes on $\mathbb{D}_{r}$ and $f$ maps $\mathbb{D}_{r}$ into
itself. Therefore also the $n$-th iterate $f^{\circ n}$ of $f$
has a representation 
\[
f^{\circ n}(z)=z^{k^{n}}g_{n}(z)
\]
 with $g_{n}(z)=1+O(z)$ never vanishing on $\mathbb{D}_{r}$. Hence
a holomorphic $k^{n}$-th root $h_{n}$ of $g_{n}$ exists on $\mathbb{D}_{r}$
and is well defined by imposing $g_{n}(0)=1$. This allows us to choose
a $k^{n}$-th root $\varphi_{n}(z)=zh_{n}(z)$ of $f^{\circ n}$ and
as in the proof of Koenigs' thoerem~\ref{thm:Koenigs}, we have an
expression of the form $\varphi_{n}\circ f=E_{k}\circ\varphi_{n+1}$
for all $n\in\mathbb{N}$. Hence the proof can be completed by showing
locally uniform convergence of $\{\varphi_{n}\}_{n}$ to a local biholomorphism
$\varphi\in\Aut(\mathbb{C},0)$. 
\end{proof}

It follows that the order $k\in\mathbb{N}$ is a complete invariant
for the superattracting case in the topological, holomorphic and formal
categories. Therefore, Theorems~\ref{thm:Koenigs} and \ref{thm:Boettcher}
provide a complete classification and description of the dynamics
for hyperbolic germs in one complex variable.

\section{\label{subsec:1D-parabolic}Parabolic fixed points}

Let $f\in\Pow(\mathbb{C},0)$ be parabolic. Then the multiplier of
$f$ is a rational rotation $\lambda=e^{2\pi ip/q}$, with $p\in\mathbb{N}$,
$q\in\mathbb{N}^{*}$ and the $q$-th iterate $f^{\circ q}$ of $f$
is tangent to the identity in the following sense:

\begin{defn}
If $f\in\Pow(\mathbb{C},0)$ has multiplier $1$, we say $f$ is \emph{tangent
to the identity}.
\end{defn}

Germs tangent to the identity will be the model case for parabolic
germs.

\subsection{Dynamics}

Let $f\in\End(\mathbb{C},0)$ be tangent to the identity of order
$k+1$, i.e. 
\begin{equation}
f(z)=z(1+az^{k})+O(z^{k+2})\label{eq:1DparaSimple}
\end{equation}
for $z\in\mathbb{C}$ near $0$ and $a\neq0$. Since the first order
term is the identity, the dynamics of $f$ are dominated by the next
non-vanishing term $az^{k}$. In particular, the factor $(1+az^{k})$
will act contracting, whenever $az^{k}$ is a (small) negative real
number, and expanding, whenever $az^{k}$ is a positive number, motivating
the following definition:
\begin{defn}
Let $f\in\End(\mathbb{C},0)$ be tangent to the identity of order
$k+1$ of the form (\ref{eq:1DparaSimple}). Then a vector $v\in\mathbb{C}^{*}$
is an \emph{\index{attracting direction@\emph{attracting direction}}attracting
direction} (for $f$ at $0$) if $kav^{k}=-1$ and a \emph{\index{repelling direction}repelling
direction}, if $kav^{k}=1$.
\end{defn}

\begin{rem}
Let $f\in\End_{1}(\mathbb{C},0)$ be of the form (\ref{eq:1DparaSimple}).
\begin{enumerate}
\item The set of attracting\slash repelling directions is invariant under
local holomorphic (or formal) changes of coordinates by our choice
of normalisation $kav^{k}=\pm1$.
\item $f$ has $k$ equally spaced attracting directions alternating with
$k$ equally spaced repelling directions around $0$ namely the $k$-th
roots of $-1/(ka)$ and $1/(ka)$ respectively.
\item Since $f^{-1}(z)=z(1-az^{k})+O(z^{k+2})$ for $z$ near $0$, the
attracting and repelling directions of $f$ exchange roles for the
inverse $f^{-1}\in\End_{1}(\mathbb{C},0)$.
\end{enumerate}
\end{rem}

\begin{defn}
Let $f\in\End(\mathbb{C},0)$ and $v\in\mathbb{C}^{*}$. An orbit
$\{z_{n}\}_{n}=\{f^{\circ n}(z)\}_{n}$ is said to converge to $0$
\emph{along\index{convergence along a real direction@\emph{convergence along a real direction}}
the direction $v$}, if $z_{n}\to0$ and $z_{n}/|z_{n}|\to v/|v|$
(in particular $z_{n}\neq0$ for all $n\in\mathbb{N}$).
\end{defn}

\begin{defn}
Let $v\in\mathbb{C}\backslash\{0\}$ and $f:U\to\mathbb{C}$ be a
representative of $f\in\End_{1}(\mathbb{C},0)$. 
\begin{enumerate}
\item An $f$-invariant open set $P_{v}\subseteq\mathbb{C}^{*}$ is an \emph{\index{attracting petal@\emph{attracting petal}}attracting
petal} for $f$ centred at $v$, if for every  point $z\in U$ its
orbit $\{z_{n}\}_{n}=\{f^{\circ n}(z)\}_{n}$ converges to $0$ along
the direction $v$ if and only if $z_{n}\in P_{v}$ eventually. A
\emph{repelling petal\index{repelling petal@\emph{repelling petal}}}
for $f$ is an attracting petal for $f^{-1}$. 
\item The (\emph{local}) \emph{parabolic basin} of $f$ \emph{(centred)
at $v$\index{parabolic basin@\emph{parabolic basin}}} (in $U$)
is 
\[
\Omega_{v}=\{z\in U\mid z_{n}/|z_{n}|\to v/|v|\}.
\]
\item A \emph{parabolic flower\index{parabolic flower@\emph{parabolic flower}}}
for $f$ is an assignment of an attracting petal $P_{v}$ to each
attracting direction $v$.
\end{enumerate}
\end{defn}

\begin{rem}
Let $P_{v}$ be an attracting petal centred at $v\in\mathbb{C}^{*}$
for $f:U\to\mathbb{C}$ tangent to the identity. Then 
\[
\Omega_{v}=\bigcup_{n\in\mathbb{N}}f^{-n}(P_{v})\subseteq A_{f}(0)
\]
 is the unique maximal attracting petal centred at $v$.
\end{rem}

In fact all stable orbits end up in an attracting petal centred at
an attracting direction. This is part of the celebrated Leau-Fatou
parabolic flower theorem, proved in its initial form by Leau \cite{Leau1897EtudeSurLesEquationsFonctionellesaUneOuPlusieursVariables},
and improved upon by Julia \cite{Julia1918MemoireSurLiterationDesFonctionsRationnelles}
and Fatou \cite{Fatou1919SurLesEquationsFonctionnellesI,Fatou1920SurLesEquationsFonctionnellesII,Fatou1920SurLesEquationsFonctionnellesIII}.
\begin{thm}[Parabolic flower theorem]
\label{thm:LeauFatou}Let $f\in\End(\mathbb{C},0)$ be tangent to
the identity of finite order $k+1\ge2$. There exist parabolic flowers
for $f$ and $f^{-1}$, such that 
\begin{enumerate}
\item \label{enu:LeauFatou1petalsNbh}The union of the attracting and repelling
petals and $\{0\}$ is a neighbourhood of $0$. Attracting petals
(for $f$) are pairwise disjoint and intersect precisely the repelling
petals centred at neighbouring repelling directions.
\item \label{enu:LeauFatou2rateOfConv}For points $z\in P_{v}$, we have
$z_{n}\sim vn^{-1/k}$ (i.e.\ $\lim_{n\to\infty}\sqrt[k]{n}z_{n}=v$).
\item \label{enu:LeauFatou3StableSet}For small neighbourhoods $U\subseteq\mathbb{C}$
of $0$, every stable orbit in $U\backslash\{0\}$ converges to $0$
along an attracting direction $v$ or the stable set is $\Sigma_{f}(U)=\{0\}\cup\bigcup_{v\text{ attracting}}\Omega_{v}$.
In particular, any orbit converging to $0$ does so along an attracting
direction.
\end{enumerate}
\end{thm}

To simplify our set-up, after a linear change of coordinates $z\mapsto\sqrt[k]{-ka}z$
(for an arbitrary choice of the $k$-th root) to the form (\ref{eq:1DparaSimple}),
we may assume 
\[
f(z)=z\paren[\Big]{1-\frac{z^{k}}{k}}+O(z^{k+2}),
\]
so the attracting directions are precisely the $k$-th roots of unity
$v_{j}=\exp(2\pi ij/k)$ for $j=0,\ldots,k-1$. A main tool in the
following proofs is the coordinate $w=\varphi(z):=z^{-k}$ defined
for $z\in\mathbb{C}^{*}$. For each $j=0,\ldots,k-1$, the restriction
of $\varphi$ to the sector
\begin{align*}
\Sigma_{j} & :=\{z\in\mathbb{C}\mid|\arg z-2\pi j/k|<\pi/k\}
\end{align*}
is a biholomorphisms onto $\mathbb{C}^{-}:=\mathbb{C}\backslash(-\infty,0]$.

Let $f:U\to\mathbb{C}$ be a representative of $f$. For a point $z\in\Sigma_{f}(U)\backslash\{0\}$
in the stable set and $n\in\mathbb{N}$ let $z_{n}:=f^{\circ n}(z)$
and $w_{n}:=\varphi(z_{n})$. Then for all $n\in\mathbb{N}$ (and
$z_{n}$ near $0$), we have 
\begin{align}
w_{n+1} & =w_{n}\paren[\bigg]{1-\frac{z_{n}^{k}}{k}+O(z_{n}^{k+1})}^{-k}\nonumber \\
 & =w_{n}\paren{1+w_{n}^{-1}+O(z_{n}^{k+1})}\label{eq:preFatouTransform1D}\\
 & =w_{n}+1+O(z_{n}).\nonumber 
\end{align}
Theorem~\ref{thm:LeauFatou} is a consequence of the following proposition:
\begin{prop}
\label{prop:Petals}For every $\theta\in(0,\pi)$, there exists $R_{\theta}>0$
such that for any $R>R_{\theta}$ the sets
\[
P_{j}(R,\theta):=\{z\in\Sigma_{j}\mid|\arg(w-R)|<\theta\},\quad j=0,\ldots,k-1
\]
form a parabolic flower for $f$ and for $z\in P_{j}(R,\theta)$,
we have 
\begin{enumerate}
\item $z_{n}=f^{\circ n}(z)\sim v_{j}/\sqrt[k]{n}$ (i.e.\ $\lim_{n\to\infty}\sqrt[k]{n}z_{n}=v_{j}$)
and if $\theta\le\pi/2$, there is a uniform constant $C>0$ such
that $|z_{n}|<Cn^{-1/k}$ for $z\in P_{j}(R,\theta)$.
\item for any backwards orbit $\{z_{-n}\}_{n}$ of $z$ (i.e.\ $f(z_{-n})=z_{-n+1}$
and $f^{\circ n}(z_{-n})=z$ for all $n\in\mathbb{N}$), we have $z_{-n}\notin P_{j}(R,\theta)$
eventually.
\end{enumerate}
\end{prop}

\begin{rem}
\label{rem:tangencyOfPetalBoundary}
\begin{enumerate}
\item The $P_{j}(R,\theta)$ are the connected components of the preimage
of 
\begin{align*}
H(R,\theta) & :=\{w\in\mathbb{C}^{*}\mid|\arg w-R|<\theta\}\\
 & =\{\Re w>R+|\Im w|/\tan(\theta)\}
\end{align*}
 under $\varphi$ and since for each $j$, $\varphi$ is injective
on $\Sigma_{j}$, the map $\varphi:P_{j}(R,\theta)\to H(R,\theta)$
is a biholomorphic coordinate on $P_{j}(R,\theta)$. 
\item The boundary of $P_{j}(R,\theta)$ for $j\in\{0,\ldots,k-1\}$ is
tangent at $0$ to $\exp(i(2\pi j/k\pm\theta/k))$. So for $\theta\nearrow\pi$
the lemma shows that the basin $\Omega_{v_{j}}$ cannot be contained
in a sector smaller than $\Sigma_{j}$. Alternatively, \cite[§II.5]{CarlesonGamelin1993ComplexDynamics}
gives an explicit construction of smoothly bounded petals tangent
to $\Sigma_{j}$.
\end{enumerate}
\end{rem}

\begin{proof}
Fix $\theta\in(0,\pi)$ and let $\varepsilon_{0}\in(0,1)$ such that
$\frac{\varepsilon_{0}}{1-\varepsilon_{0}}<|\tan\theta|$. For $R>0$
large enough, for all $z\in P_{j}(R,\theta)$, we have 
\begin{equation}
|w_{1}-w-1|<\varepsilon_{0}\label{eq:preFatouError}
\end{equation}
and hence 
\[
|\Im(w_{1}-w)|<\varepsilon_{0}<\frac{\varepsilon_{0}}{1-\varepsilon_{0}}\Re(w_{1}-w)
\]
Therefore the line segment from $w$ to $w_{1}$ cannot intersect
the boundary of $H(R,\theta)$, a line of slope $|\tan(\theta)|>\frac{\varepsilon}{1-\varepsilon}$
and we have $w_{1}\in H(R,\theta)$ and hence $z_{1}$ lies in a petal
$P_{j'}(R,\theta)$ for some $j'$. On the other hand, we have 
\[
d(z,\partial\Sigma_{j})>C_{\theta}|z|\gg|z_{1}-z|
\]
for some $C_{\theta}>0$ as $z\to0$ in $P_{j}(R,\theta)$. So for
$R>0$ large enough, the image $z_{1}$ of $z\in P_{j}(R,\theta)$
cannot leave $\Sigma_{j}$ and has to lie in the same petal.

By induction for all $n\in\mathbb{N}$ we then have $z_{n}\in P_{j}(R,\theta)$
and with (\ref{eq:preFatouError})
\[
\Re w_{n}>\Re w+(1-\varepsilon_{0})n,
\]
so $|w_{n}|\to+\infty$ and hence $z_{n}\to0$. With this (\ref{eq:preFatouTransform1D})
implies that $w_{n+1}-w_{n}\xrightarrow[n\to\infty]{}+1$ and hence
also 
\[
\frac{w_{n}-w}{n}=\frac{1}{n}\sum_{l=0}^{n-1}(w_{l+1}-w_{l})\xrightarrow[n\to+\infty]{}+1.
\]
 Since $w/n\to0$, we thus have $z_{n}^{k}=w_{n}^{-1}\sim n^{-1}$.
Since $z_{n}\in\Sigma_{j}$ for all $n\in\mathbb{N}$, we can extract
the unique $k$-th root with values in $\Sigma_{j}$ and it follows
$z_{n}\sim v_{j}/\sqrt[k]{n}$. In particular, $\{z_{n}\}_{n}$ converges
to $0$ along the direction $v_{j}$. Since $P_{j}(R,\theta)$ is
a neighbourhood of the ray $\mathbb{R}^{+}v_{j}$ near $0$, it eventually
contains every orbit converging to $0$ from that direction and thus
$P_{j}(R,\theta)$ is an attracting petal.

For Part~2 assume $z_{-n}\in P_{j}(R,\theta)$ for all $n\in\mathbb{N}$
and set $w_{-n}=z_{-n}^{-k}$. Then from (\ref{eq:preFatouError})
by induction for all $n\in\mathbb{N}$ we have 
\begin{align*}
\Re w_{-n}-|\Im w_{-n}|/\tan\theta & <\Re w-(1-\varepsilon_{0})n-|\Im w|/\tan\theta+n\varepsilon_{0}/|\tan\theta|\\
 & =\Re w-|\Im w|/\tan\theta-\underbrace{(1-\varepsilon_{0}-\varepsilon_{0}/|\tan\theta|)}_{>0}n
\end{align*}
and since $1-\varepsilon_{0}-\varepsilon_{0}/|\tan\theta|>0$ by
our choice of $\varepsilon_{0}$, for large $n\in\mathbb{N}$ the
right hand side is less then $R$, hence $w_{-n}\notin H(R,\theta)$,
contradicting $z_{-n}\in P_{j}(R,\theta)$.
\end{proof}
For $f^{-1}$ the picture is rotated by $\pi/k$, so for $\theta>\pi/2$,
by Remark~\ref{rem:tangencyOfPetalBoundary}, the petals of the parabolic
flowers for $f$ and $f^{-1}$ constructed in Proposition~\ref{prop:Petals}
overlap in simply connected regions near $0$ and hence form a punctured
neighbourhood $U\backslash\{0\}$ of $0$. This proves Part~\ref{enu:LeauFatou1petalsNbh}
of Theorem~\ref{thm:LeauFatou}.  Moreover, Part~2 of Proposition~\ref{prop:Petals}
shows that all forward orbits have to leave the repelling petals eventually,
so they either end up in an attracting petal or leave $U$, showing
Part~\ref{enu:LeauFatou3StableSet} of Theorem~\ref{thm:LeauFatou}.
.

Finally, on each parabolic petal, we can find a conjugation to a simple
normal form (see \cite[Thm.~10.9]{Milnor2006Dynamicsinonecomplexvariable}):
\begin{thm}[Parabolic linearisation theorem]
\label{thm:FatouCoord1D}Let $f\in\End(\mathbb{C},0)$ be tangent
to the identity. Then each parabolic basin $\Omega_{v}$ of $f$ centred
at an attracting direction $v$ admits a holomorphic map $\psi:\Omega_{v}\to\mathbb{C}$
satisfying the \emph{Abel functional equation}\index{Abel functional equation}
\[
\psi\circ f=\psi+1.
\]
This map $\psi$ is unique up to addition by a complex constant. On
a petal $P_{v}=P_{j}(R,\pi/2)\subseteq\Omega_{v}$ (in local coordinates
such that $v=v_{j}$) from Proposition~\ref{prop:Petals}, the map
$\psi$ is injective, and hence conjugates $f$ to a translation $\zeta\mapsto\zeta+1$.
Moreover, the image $\psi(P_{v})$ contains a right half-plane $\{\Re\zeta>R'\}$.
\end{thm}

\begin{defn}
We call a map $\psi$ as in Theorem~\ref{thm:FatouCoord1D} a \emph{Fatou
coordinate} for $f$.
\end{defn}

Let now $f\in\End(\mathbb{C},0)$ be a general parabolic germ with
multiplier $\lambda$, a (primitive) $p$-th root of unity. Then $f^{\circ p}$
is tangent to the identity and we can deduce the dynamical behaviour
of $f$ from that of $f^{\circ p}$, which is given by Theorem~\ref{thm:LeauFatou}
unless $f^{\circ q}=\id$. This is precisely the case when $f$ is
linearisable:
\begin{prop}
Let $p\in\mathbb{N}^{*}$. 
\begin{enumerate}
\item An invertible germ $f\in\Pow(\mathbb{C},0)$ is holomorphically\slash formally
linearisable, if and only if $f^{\circ p}$ is. 
\item A parabolic germ $f\in\Pow(\mathbb{C},0)$ with multiplier $\lambda$,
a $p$-th root of unity, is topologically\slash formally\slash holomorphically
linearisable if and only if $f^{\circ p}=\id$.
\end{enumerate}
\end{prop}

\begin{proof}
Let $\varphi\in\Pow(\mathbb{C},0)$ linearise $f^{\circ p}$, so $\varphi\circ f^{\circ p}=\lambda^{p}\varphi$
where $\lambda\in\mathbb{C}^{*}$ is the multiplier of $f$, and define
\[
\psi:=\sum_{j=1}^{p}\lambda^{-j}\varphi\circ f^{\circ j}=\sum_{j=1}^{p}\lambda^{-j+1}\varphi\circ f^{\circ(j-1)}.
\]
Then 
\begin{align*}
\lambda\psi & =\sum_{j=1}^{p-1}\lambda^{-j+1}\varphi\circ f^{\circ j}=\psi\circ f.
\end{align*}
and $\psi'(0)=p\lambda^{p}\varphi'(0)\neq0$, so $\psi\in\Pow(\mathbb{C},0)$
is invertible, conjugates $f$ to $g$, and converges if $\varphi$
does. The converse is clear.

For the second part, note that if $f^{\circ p}$ is conjugate to the
identity, then $f^{\circ p}=\id$, since the identity is the only
element in its conjugacy class (in all three categories), hence we
can proceed as above with $\varphi=\id$. In particular, the topological
case leads to a holomorphic conjugation.
\end{proof}

\begin{defn}
Let $f\in\End(\mathbb{C},0)$ be parabolic with multiplier $\lambda$,
a (primitive) $p$-th root of unity, and $f^{\circ p}\neq\id$. Then
the \emph{\index{parabolic order@\emph{parabolic order}}parabolic
order} of $f$ is the order of $f^{\circ p}$.
\end{defn}

If $\{z_{pn}\}_{n}$ is an orbit of $f^{\circ p}$ converging to $0$
along a direction $v$, then the orbit $\{z_{pn+1}\}_{n}=\{f(z_{pn})\}_{n}$
of $f^{\circ p}$ converges to $0$ along the direction $\lambda v$.
Hence multiplication by $\lambda$ permutes the attracting directions
of $f^{\circ p}$ and it follows:
\begin{cor}
\label{cor:ParabFlowerPermutation}Let $f\in\End(\mathbb{C},0)$ be
parabolic with multiplier $\lambda$, a (primitive) $p$-th root of
unity, and $f^{\circ p}\neq\id$. Then the parabolic order of $f$
is $kp+1$ for some integer $k\ge1$ and $f$ acts on the $kp$ attracting\slash repelling
petals for $f^{\circ p}$ from Theorem~\ref{thm:LeauFatou} as a
permutation composed of $k$ disjoint cycles of period $p$.
\end{cor}

\subsection{Classification}

It is clear from the dynamical behaviour that the multiplier and the
parabolic order are invariants under conjugation with orientation-preserving
homeomorphisms. In fact, these are sufficient for the topological
classification.
\begin{thm}[Camacho \cite{Camacho1978OntheLocalStructureofConformalMappingsandHolomorphicVectorFieldsinC2},
Shcherbakov \cite{Shcherbakov1982TopologicalClassificationofGermsofConformalMappingswithIdenticalLinearPart}]
Let $f\in\End(\mathbb{C},0)$ be parabolic with multiplier $\lambda$,
a (primitive) $p$-th root of unity, and $f^{\circ p}\neq\id$. Then
there exists a unique $k\in\mathbb{N}$ such that $f$ is topologically
conjugate to $z\mapsto\lambda z(1+z^{pk})$ (and $pk+1$ is the parabolic
order of $f$). In fact the conjugating homeomorphism can be chosen
$\mathcal{C}^{\infty}$ outside of $0$.
\end{thm}

See also \cite[§3.2]{Bracci2010Localholomorphicdynamicsofdiffeomorphismsindimensionone}
for a short proof and \cite{Jenkins2008HolomorphicGermsandtheProblemofSmoothConjugacyinaPuncturedNeighborhoodoftheOrigin}
for a more detailed discussion.

For the formal classification, observe that if $\lambda$ is a primitive
$p$-th root of unity, then by Proposition~\ref{prop:1DPoincareDulac},
$f$ is formally conjugate to the form 
\begin{equation}
f(z)=\lambda z+a_{k}z^{pk+1}+a_{k+1}z^{p(k+1)+1}+\cdots\label{eq:parabolicPoincDulac1D}
\end{equation}
with $k\ge1$ and $a_{k}\neq0$, but this form is far from unique.
The complete classification is somewhat folklore, in \cite{Voronin1981AnalyticClassificationofGermsofConformalMappingsCOCOwithIdentityLinearPart}
and \cite{ArnoldGoryunovLyashkoVassiliev1993DynamicalSystemsVIIISingularityTheoryIIClassificationandApplicationsTranslfromtheRussianbyJSJoel}
it is attributed to N.~Venkov without a reference, whereas \cite{OFarrellShort2015ReversibilityinDynamicsandGroupTheory}
conjectures Kasner's articles \cite{Kasner1915ConformalClassificationofAnalyticArcsorElementsPoincaresLocalProblemofConformalGeometry,Kasner1916InfiniteGroupsGeneratedbyConformalTransformationsofPeriodTwoInvolutionsandSymmetries}
as the earliest source.
\begin{prop}[Formal classification]
\label{prop:1DparaFormalClass}Let $f\in\Pow(\mathbb{C},0)$ be parabolic
with multiplier $\lambda$, a (primitive) $p$-th root of unity and
$f^{\circ p}\neq\id$. Then there exist unique $k\in\mathbb{N}$ and
$b\in\mathbb{C}$ such that $f$ is formally conjugate to 
\begin{equation}
z\mapsto\lambda z+z^{pk+1}+bz^{2pk+1}.\label{eq:parabolicNF}
\end{equation}
Moreover, $pk+1$ is the parabolic order of $f$. If $f$ converges,
i.e.~$f\in\End(\mathbb{C},0)$, we have 
\begin{equation}
b=\frac{1}{2\pi i}\int_{\gamma}\frac{dz}{\lambda z-f(z)}\label{eq:ParabConstb}
\end{equation}
for small positively oriented loops $\gamma$ around $0$.
\end{prop}

\begin{proof}
We know $f$ is formally conjugate to the form (\ref{eq:parabolicPoincDulac1D})
and after a linear change of coordinates, we may assume $a_{k}=1$.
A conjugation with 
\[
h(z)=z+cz^{pj+1},\quad c\in\mathbb{C},j\ge1
\]
then conjugates $f$ to 
\[
h^{-1}\circ f\circ h(z)=f(z)+p(k-j)cz^{p(k+j)+1}+O(z^{p(k+j)+2})
\]
(see (\ref{eq:homol1DexplByDegree})), so by suitable choice of $c$,
we can ensure $a_{k+j}=0$, if $j\neq k$. By induction on $j\ge1$,
we can hence ensure $a_{k+j}=0$ for all $j\ge1$, $j\neq k$ and
$j\le j_{{\mathrm{m}ax}}$ up to any finite order $j_{{\mathrm{m}ax}}\in\mathbb{N}$
and by Lemma~\ref{lem:FormalConjIffAnyOrderPolyConj}, we have formal
conjugacy to the form (\ref{eq:parabolicNF}). 

Equation (\ref{eq:ParabConstb}) follows from the residue theorem
and shows uniqueness of the parameter $b$.
\end{proof}
\begin{defn}
The \emph{index\index{index at a parabolic point@\emph{index at a parabolic point}}}
at $0$ of a parabolic germ $f\in\End(\mathbb{C},0)$ with multiplier
$\lambda$ is 
\[
b=\frac{1}{2\pi i}\int_{\gamma}\frac{dz}{\lambda z-f(z)}.
\]
If $\lambda$ is a $p$-th root of unity and $f$ of order $kp+1$,
the \emph{iterative residue\index{iterative residue@\emph{iterative residue}}}
of $f$ at $0$ is 
\[
\m{Resit}(f)=\frac{kp+1}{2}-b.
\]
The holomorphic classification requires an additional much more complicated
functional invariant of infinite dimension, called the \emph{sectorial
invariant}, associated to a germ $f\in\End(\mathbb{C},0)$ tangent
to the identity. It is constructed from the injective holomorphic
maps from Theorem~\ref{thm:LeauFatou} conjugating $f$ to the shift
$\zeta\mapsto\zeta+1$ on each attracting and repelling petal. The
sectorial invariant encodes (see \cite{Abate2010Discreteholomorphiclocaldynamicalsystems}
for details) the difference of the two conjugations on the intersection
of each neighbouring pair of attracting and repelling petals as a
$2k$ holomorphic germs $h_{1}^{\pm},\ldots,h_{k}^{\pm}\in\End(\mathbb{C},0)$
with respective multipliers $\lambda_{1}^{\pm},\ldots,\lambda_{k}^{\pm}$
such that
\begin{equation}
\prod_{j=1}^{k}(\lambda_{j}^{+}\lambda_{j}^{-})=\exp(4\pi^{2}\m{Resit}(f)).\label{eq:VoroninMultiplierCondition}
\end{equation}
Two tuples $(h_{1}^{\pm},\ldots,h_{k}^{\pm})$ and $(g_{1}^{\pm},\ldots,g_{k}^{\pm})$
of germs in $\End(\mathbb{C},0)$ whose multipliers satisfy (\ref{eq:VoroninMultiplierCondition})
are then called equivalent, if, up to cyclic permutation, there exist
$\alpha_{j},\beta_{j}\in\mathbb{C}^{*}$ such that 
\[
\alpha_{j}g_{j}^{-}(z)=h_{j}^{-}(\beta_{j}z)\quad\text{and}\quad g_{j}^{+}(\beta_{j}z)=\alpha_{j+1}h_{j}^{+}(z)
\]
for $z\in\mathbb{C}$ near $0$, and the set of resulting equivalence
classes is denoted $\mathcal{M}_{r}$.
\end{defn}

Then, Écalle and Voronin's classification of parabolic germs has the
form:
\begin{thm}[Holomorphic classification\cite{Ecalle1981LesFonctionsResurgentes1LesAlgebresDeFonctionsResurgentes,Ecalle1981LesFonctionsResurgentes2LesFonctionsResurgentesAppliqueesaLiteration},
\cite{Voronin1981AnalyticClassificationofGermsofConformalMappingsCOCOwithIdentityLinearPart}]
Let $f,g\in\End(\mathbb{C},0)$ be parabolic with the same multiplier
$\lambda$, a (primitive) $p$-th root of unity, and $f^{\circ p},f^{\circ q}\neq\id$.
Then $f$ and $g$ are holomorphically conjugate, if and only if they
have the same order and index and $f^{\circ p}$ and $g^{\circ q}$
have the same sectorial invariant.

Moreover, any combination of multiplier $\lambda$ a $p$-th root
of unity, order $pk+1\ge2$, parabolic index $\iota\in\mathbb{C}$
and sectorial invariant $\mu\in\mathcal{M}_{pk}$ is realised by a
parabolic germ $f\in\End(\mathbb{C},0)$.

\end{thm}

\section{\label{subsec:1D-elliptic}Elliptic fixed points}

If $f\in\End(\mathbb{C},0)$ is elliptic, the multiplier of $f$ is
an irrational rotation $\lambda=e^{2\pi i\theta}$ with $\theta\in\mathbb{R}\backslash\mathbb{Q}$.
This case can lead to the most complicated dynamical behaviour and
has not been fully understood, even though formally any elliptic germ
is linearisable by Corollary~\ref{cor:1DformalLinearisation}. A
fundamental question is then: When does the (unique) formal linearisation
converge? After some topological results, we discuss precise conditions
on the angle $\theta$ in the following sections. First we observe,
that linearisability is characterised by topological dynamics: 
\begin{prop}
\label{prop:linIffStable}Let $f\in\End(\mathbb{C},0)$ be a neutral
germ with multiplier $\lambda\in S^{1}$. Then $f$ is holomorphically\slash topologically
linearisable if and only if $0$ is a stable fixed point for $f$,
i.e. the stable set $\Sigma_{f}$ contains a neighbourhood of $0$.
\end{prop}

\begin{proof}
If $f$ is linearisable, then $0$ is stable. If $0$ is stable, let
\begin{equation}
\varphi_{k}(z):=\frac{1}{k}\sum_{j=0}^{k-1}\lambda^{-j}f^{j}(z).\label{eq:NeutralLin}
\end{equation}
Let $U\subseteq\mathbb{C}$ be a bounded neighbourhood of $0$. By
stability, there exists a neighbourhood $V\subseteq\Sigma_{f}(U)$
of $0$ and hence $f^{j}$ is uniformly bounded on $V$ and so is
$\varphi_{k}$. Montel's Compactness Principle~\ref{thm:MontelBdd}
then yields a locally uniformly convergent subsequence $\varphi_{k_{j}}\xrightarrow[j\to+\infty]{}\varphi$
on $V$. We have $\varphi_{k_{j}}'(0)=1$ by definition \ref{eq:NeutralLin},
so $\varphi$ is invertible and again by uniform boundedness of $f^{k_{j}}$
on $V$, we have 
\[
\varphi_{k_{j}}\circ f-\lambda\varphi_{k_{j}}=\frac{\lambda}{k_{j}}\paren{\lambda^{-k_{j}}f^{k_{j}}-\id}\xrightarrow[j\to+\infty]{}0.
\]
Thus $\varphi\in\Aut(\mathbb{C},0)$ linearises $f$.
\end{proof}

Secondly, Na{\u\i}shul' \cite{Nauishul1983TopologicalInvariantsofAnalyticandAreaPreservingMappingsandTheirApplicationtoAnalyticDifferentialEquationsinC2andCP2}
shows, that two elliptic germs with different multiplier, cannot be
topologically conjugate, even if they are linearisable (see also Pérez-Marco
\cite[\textbackslash\{\}textsection IV.1]{PerezMarco1997FixedPointsandCircleMaps}
or Bracci \cite{Bracci2010Localholomorphicdynamicsofdiffeomorphismsindimensionone}):
\begin{thm}[Na{\u\i}shul' \cite{Nauishul1983TopologicalInvariantsofAnalyticandAreaPreservingMappingsandTheirApplicationtoAnalyticDifferentialEquationsinC2andCP2}]
For neutral germs $f\in\End(\mathbb{C},0)$ (i.e. with multiplier
$\lambda\in S^{1}$), the multiplier is preserved under conjugation
with orientation-preserving homeomorphisms.
\end{thm}

In other words, for neutral germs the unordered pair $\{\lambda,\overline{\lambda}\}$
is a topological invariant.

\subsection{Linearisation and small divisors}

Let $f\in\End(\mathbb{C},0)$ be elliptic with multiplier $\lambda=e^{2\pi i\theta}$
with $\theta\in\mathbb{R}\backslash\mathbb{Q}$. Now we consider the
question when the (unique) formal linearisation of $f$ from Corollary~\ref{cor:1DformalLinearisation}
converges.

Recall that if $f(z)=\sum_{j=1}^{\infty}f_{j}z^{j}$, (\ref{eq:homol1DexplByDegree})
implies that a formal power series $h\in\Pow_{1}(\mathbb{C},0)$ given
by $h(z)=\sum_{j=1}^{\infty}h_{j}z^{j}$ linearising $f$ is uniquely
determined by 
\begin{equation}
h_{j}=\frac{1}{\lambda^{j}-\lambda}\paren[\Big]{\sum_{k=2}^{j}\sum_{l_{1}+\cdots+l_{k}=j}f_{k}h_{l_{1}}\cdots h_{l_{k}}}\label{eq:LinSeries1D}
\end{equation}
for $j\ge2$. The series $h$ diverges if the coefficients $h_{j}$
are too large, which is closely related to the divisors $\lambda^{j}-\lambda$
being too small too often.  Finding suitable notions of divisors
being ``too small too often'' is known as a \emph{\index{small divisors problem@\emph{small divisors problem}}small
divisors} or \emph{small denominators problem}\index{small denominators problem@\emph{small denominators problem}}
and is closely related to the arithmetical properties of the rotation
number $\theta\in\mathbb{R}\backslash\mathbb{Q}$.
\begin{defn}
Then $0$ is a \emph{Siegel point\index{Siegel point@\emph{Siegel point}}}
for $f$, if $f$ is holomorphically linearisable and a \emph{\index{Cremer point@\emph{Cremer point}}Cremer
point} for $f$ otherwise. A \emph{\index{Siegel disk@\emph{Siegel disk}}Siegel
disk} for $f$ at a Siegel point $0$ is an injective holomorphic
map $\phi:\mathbb{D}\to\mathbb{C}$ with $\phi(0)=0$ conjugating
$f$ to an irrational rotation.
\end{defn}

G.A.~Pfeiffer first proved the existence of non-linearisable elliptic
germs in \cite{Pfeiffer1917OntheConformalMappingofCurvilinearAnglestheFunctionalEquationfxa_1x}.
H.~Cremer continued Pfeiffer's work in \cite{Cremer1927ZumZentrumproblem}
and \cite{Cremer1938UberDieHaufigkeitDerNichtzentren} and proved
the following result: 
\begin{thm}[Cremer \cite{Cremer1938UberDieHaufigkeitDerNichtzentren}]
\label{thm:Cremer}Let $\lambda=e^{2\pi i\theta}\in S^{1}$ with
$\theta\in(-1/2,1/2)$. Then $\lambda$ satisfies 
\begin{equation}
\limsup_{n\to+\infty}|\lambda^{n}-1|^{-1/n}=+\infty,\label{eq:Cremer}
\end{equation}
if and only if 
\begin{equation}
\limsup_{n\to+\infty}|\{n\theta\}|^{-1/n}=+\infty,\label{eq:Cremer2}
\end{equation}
where $\pbrace x:=x-\left[x\right]$ is the fractional part and $\left[x\right]$
the integer part of a real number $x$. And if $\lambda$ satisfies
these conditions, then there exists a germ $f\in\End(\mathbb{C},0)$
with multiplier $\lambda$ that is not holomorphically linearisable.
Furthermore, the set of numbers $\lambda\in S^{1}$ satisfying (\ref{eq:Cremer})
contains a countable intersection of dense open sets.
\end{thm}

\begin{rem}
\label{rem:CremerGeneric}A set containing a countable intersection
of dense open subsets of a topological space is known as a \emph{comeagre
}or \emph{generic} subset in the fields of topology and algebraic
geometry. In particular, a generic set is necessarily dense and uncountable
(by the Baire category theorem) and intersections with other generic
sets are generic. Hence Theorem~\ref{thm:Cremer} shows that the
set of multipliers admitting non-linearisable germs is dense and topologically
generic in $S^{1}$.
\end{rem}

\begin{proof}
Note that 
\[
|\lambda^{n}-1|=2|\sin(\pi n\theta)|\in(4|\pbrace{n\theta}|,2\pi\abs{\pbrace{n\theta}}),
\]
so (\ref{eq:Cremer}) and (\ref{eq:Cremer2}) are finite or infinite
at the same time. Let $f(z)=\sum_{j=1}^{\infty}f_{n}z^{n}$ with $f_{1}=\lambda=e^{2\pi i\theta}$
and $f_{j}=e^{i\theta_{j}}$ with 
\[
\theta_{j}=\arg\paren[\Big]{\sum_{k=2}^{j-1}\sum_{l_{1}+\cdots+l_{k}=j}f_{k}h_{l_{1}}\cdots h_{l_{k}}},
\]
where $h_{j}$ are given by $h_{1}=1$ and (\ref{eq:LinSeries1D}),
so 
\[
|h_{j}|=h_{j}=\frac{1}{|\lambda^{j}-\lambda|}\abs{f_{j}+\sum_{k=2}^{j-1}\sum_{l_{1}+\cdots+l_{k}=j}f_{k}h_{l_{1}}\cdots h_{l_{k}}}>\frac{|f_{j}|}{|\lambda^{j}-\lambda|}=\frac{1}{|\lambda^{j}-\lambda|},
\]
and $\limsup_{n\to\infty}|h_{n}|^{1/n}=+\infty$, hence the linearising
series $h$ diverges.
\end{proof}
On the other hand, a few years later C.~L.~Siegel established guaranteed
linearisability for a large class of multipliers:
\begin{thm}[Siegel \cite{Siegel1942Iterationofanalyticfunctions}]
\label{thm:Siegel}If for $\lambda\in\mathbb{C}$ there exist $C,\mu>0$
such that 
\begin{equation}
|\lambda^{m}-1|\ge Cm^{-\mu}\label{eq:Siegel}
\end{equation}
for every $m\ge2$, then every germ $f\in\End_{\lambda}(\mathbb{C},0)$
is holomorphically linearisable. Furthermore, condition (\ref{eq:Siegel})
is satisfied for $\lambda$ in a subset $S\subseteq S^{1}$ of full
Lebesgue-measure.
\end{thm}

\begin{rem}
Milnor \cite[Lemma~C.7]{Milnor2006Dynamicsinonecomplexvariable} shows
that the complement of $S$ has Hausdorff dimension $0$. So this
set is large in the sense of measure theory, but topologically small
as the complement of a generic set by Remark~\ref{rem:CremerGeneric}.
\end{rem}

\begin{rem}
A number $\lambda=e^{2\pi i\theta}$ with $\theta\in\mathbb{R}$ satisfies
the Siegel condition (\ref{eq:Siegel}) if and only if $\theta$ is
\emph{Diophantine,} that is there exist $m,c>0$ such that
\begin{equation}
\abs[\bigg]{\theta-\frac{p}{q}}>\frac{c}{q^{m}}\quad\forall p\in\mathbb{Z},q\in\mathbb{N}_{>0}.\label{eq:Diophantisch}
\end{equation}
In other words, the Siegel condition (\ref{eq:Siegel}) is equivalent
to the rotation number $\theta$ being badly approximated by rationals.
In particular, is is easy to show that every algebraic number $\theta\in\mathbb{R}\backslash\mathbb{Q}$
satisfies (\ref{eq:Diophantisch}).
\end{rem}

Siegel's proof of Theorem~\ref{thm:Siegel} using majorant series
was further refined by Brjuno \cite{Brjuno1965ConvergenceofTransformationsofDifferentialEquationstoNormalForm}
under the following weaker condition:
\begin{defn}
For $\lambda\in\mathbb{C}$ and an integer $m\ge2$, let 
\[
\omega(\lambda,m):=\min\{|\lambda^{j}-\lambda|\mid2\le j\le m\}.
\]
We say $\lambda$ is a \emph{Brjuno number}\index{Brjuno number@\emph{Brjuno number}},
if it satisfies the \emph{Brjuno condition\index{Brjuno condition@\emph{Brjuno condition}}}
\begin{equation}
\sum_{j=0}^{\infty}\frac{1}{p_{j}}\log\frac{1}{\omega(\lambda,p_{j+1})}<\infty\label{eq:1DBrjunoCond}
\end{equation}
for any strictly increasing sequence of positive integers $0<p_{0}<p_{1}<\cdots$.
\end{defn}

\begin{rem}
The Brjuno condition (\ref{eq:1DBrjunoCond}) is independent of the
chosen sequence $\{p_{j}\}_{j}$ and is hence often represented as
\[
\sum_{j=0}^{\infty}\frac{1}{2^{j}}\log\frac{1}{\omega(\lambda,2^{j+1})}<\infty.
\]
\end{rem}

\begin{thm}[Brjuno \cite{Brjuno1965ConvergenceofTransformationsofDifferentialEquationstoNormalForm}]
\label{thm:BrjunoYoccoz1D}If $\lambda\in\mathbb{C}$ satisfies the
Brjuno condition, then every $f\in\End(\mathbb{C},0)$ with multiplier
$\lambda$ is linearisable.
\end{thm}

(See the multi-dimensional version, Theorem~\ref{thm:Brjuno} for
the proof.)

J.-C.~Yoccoz \cite{Yoccoz1988LinearisationDesGermesDeDiffeomorphismesHolomorphesDeC0LinearizationofGermsofHolomorphicDiffeomorphismsofC0,Yoccoz1995TheoremeDeSiegelNombresDeBrunoEtPolynomesQuadratiques}
used geometric methods and quasi-conformal conjugations to show that
the Brjuno condition is sharp:
\begin{thm}[Yoccoz \cite{Yoccoz1988LinearisationDesGermesDeDiffeomorphismesHolomorphesDeC0LinearizationofGermsofHolomorphicDiffeomorphismsofC0,Yoccoz1995TheoremeDeSiegelNombresDeBrunoEtPolynomesQuadratiques}]
\label{thm:Yoccoz}If $\lambda\in S^{1}$ and (\ref{eq:1DBrjunoCond})
diverges, then the polynomial $P_{\lambda}(z):=\lambda z+z^{2}$ is
not linearisable around the elliptic fixed point $0$. 
\end{thm}

In summary: $\lambda$ satisfies (\ref{eq:1DBrjunoCond}) $\Leftrightarrow$
every $f\in\End(\mathbb{C},0)$ with multiplier $\lambda$ is linearisable
$\Leftrightarrow$ $P_{\lambda}$ is linearisable.

\begin{rem}[Continued fractions]
For $\lambda=e^{2\pi i\theta}\in S^{1}$, $\theta\in(0,1)\backslash\mathbb{Q}$
let $\{p_{j}/q_{j}\}_{j\in\mathbb{N}}$ be the sequence of \emph{convergents}
 of the continued fraction expansion of $\theta$ (see e.g. Khinchin
\cite{Khinchin1997ContinuedFractions}). Then $p_{j}/q_{j}\xrightarrow[j\to\infty]{}\theta$
and each $p_{j}/q_{j}$ is the best rational approximation of $\theta$
with denominator at most $q_{j}$. This makes continued fractions
a potent instrument for precise estimates on rational approximability
and small-divisor problems in dimension 1. In particular, the Brjuno
condition (\ref{eq:1DBrjunoCond}) on $\lambda$ is equivalent to:
\begin{equation}
B'(\theta):=\sum_{j=0}^{\infty}\frac{1}{q_{j}}\log q_{j+1}<\infty\label{eq:BrjunoCF}
\end{equation}
and corresponds to slow growth of the denominators $q_{j}$ or bad
approximation by rationals.
\end{rem}

This point of view even allows an estimate of the radius of convergence
of germs with a given multiplier: Let $S_{\lambda}$ denote the space
of (germs of) univalent maps $f:\mathbb{D}\to\mathbb{C}$ with $f(0)=0$
and $f'(0)=\lambda$ and for $f\in S_{\lambda}$ let $h_{f}$ denote
the radius of convergence of a linearisation of $f$. For $\theta\in(0,1)$,
set 
\[
r(\theta):=\inf_{f\in S_{e^{2\pi i\theta}}}r(f).
\]

\begin{thm}[Yoccoz \cite{Yoccoz1995TheoremeDeSiegelNombresDeBrunoEtPolynomesQuadratiques}
(see also \cite{BuffCheritat2006TheBrjunoFunctionContinuouslyEstimatestheSizeofQuadraticSiegelDisks})]
There exists a constant $C>0$, such that 
\[
|\log r(\theta)+B'(\theta)|\le C,
\]
whenever $B'(\theta)$ or $-\log r(\theta)$ is finite.
\end{thm}

\subsection{Small cycles and Hedgehogs}

The local classification and dynamics near Siegel points are clear.
Around a Cremer point, we know the dynamics are not stable (Proposition~\ref{prop:linIffStable}).
In fact the dynamics are extremely complicated and not completely
understood to this day. The obstruction to linearisation employed
by J.-C. Yoccoz in \cite{Yoccoz1988LinearisationDesGermesDeDiffeomorphismesHolomorphesDeC0LinearizationofGermsofHolomorphicDiffeomorphismsofC0,Yoccoz1995TheoremeDeSiegelNombresDeBrunoEtPolynomesQuadratiques}
to prove the necessity of the Brjuno condition is the small-cycles
property:
\begin{defn}
A germ $f\in\End(\mathbb{C},0)$ has the \emph{small-cycles property\index{small cycles@\emph{small cycles}}}
or just \emph{small cycles}, if for every $\varepsilon>0$ there exists
a periodic orbit of $f$ inside $\mathbb{D}_{\varepsilon}$.
\end{defn}

\begin{prop}[Yoccoz \cite{Yoccoz1988LinearisationDesGermesDeDiffeomorphismesHolomorphesDeC0LinearizationofGermsofHolomorphicDiffeomorphismsofC0}]
If $\lambda\in S^{1}$ does not satisfy the Brjuno condition (\ref{eq:1DBrjunoCond}),
then every quadratic polynomial $P_{\lambda}\in\End_{\lambda}(\mathbb{C},0)$
has small cycles.
\end{prop}

\begin{rem}
Neither hyperbolic nor parabolic germs can have the small cycles property.
\end{rem}

On the other hand, if we exclude small cycles, Brjuno's condition
(\ref{eq:BrjunoCF}) can be weakened and still ensure linearisation:
\begin{thm}[Pérez-Marco \cite{PerezMarco1993SurLesDynamiquesHolomorphesNonLinearisablesEtUneConjectureDeVIArnold}]
\label{thm:WeakBrjuno}Let $\{q_{j}\}_{j}$ be the sequence of denominators
of the convergents of $\lambda\in\mathbb{C}$. Then 
\begin{equation}
\sum_{j=0}^{\infty}\frac{1}{q_{j}}\log\log q_{j+1}<\infty,\label{eq:weakBrjuno}
\end{equation}
if and only if every $f\in\End_{\lambda}(\mathbb{C},0)$ without small
cycles is holomorphically linearisable. Otherwise,  there exist
$f\in\End_{\lambda}(\mathbb{C},0)$ and a neighbourhood $U\subseteq\mathbb{C}$
of $0$ such that every stable orbit of $f$ in $U$ has $0$ as an
accumulation point (in particular, $f$ is not linearisable and does
not have small cycles).
\end{thm}

\begin{rem}
Note that the parabolic case is not excluded and in this case the
statement follows from the Leau-Fatou flower theorem~\ref{thm:LeauFatou}.
\end{rem}

\begin{defn}
Let $f\in\End(\mathbb{C},0)$. A domain $U\subseteq\mathbb{C}$ is
\emph{admissible} for $f$, if $U$ is a $\mathcal{C}^{1}$ Jordan
neighbourhood $0$ and both $f$ and $f^{-1}$ are univalent on a
neighbourhood of $\overline{U}$. A \emph{Siegel compactum\index{Siegel compactum@\emph{Siegel compactum}}}
$K$ for $f$ on a neighbourhood $U\subseteq\mathbb{C}$ of $0$ is
a compact subset $K\subseteq\overline{U}$  that
\begin{enumerate}
\item is connected and \emph{full\index{full set@\emph{full set}}} (i.e.\ $\mathbb{C}\backslash K$
is connected),
\item contains the origin: $0\in K$,
\item intersects the boundary: $K\cap\partial U\neq\emptyset$,
\item is completely $f$-invariant: $f(K)=K,f^{-1}(K)=K$.
\end{enumerate}
\end{defn}

\begin{thm}[Pérez-Marco \cite{PerezMarco1995SurUneQuestionDeDulacEtFatou,PerezMarco1997FixedPointsandCircleMaps}]
Let $f\in\End_{\lambda}(\mathbb{C},0)$ with $|\lambda|=1$. Then
for every admissible $\mathcal{C}^{1}$ Jordan domain $0\subseteq U\subseteq\mathbb{C}$,
there exists a Siegel compactum $K\subseteq\overline{U}$ for $f$
on $U$. If $f$ is elliptic, then the unique Siegel compactum $K_{f}(U)$
for $f$ on $U$ is the connected component of the stable set 
\[
\Sigma_{f}(\overline{U})=\{z\in\overline{U}\mid f^{\circ n}(z)\in\overline{U}\enskip\forall n\in\mathbb{N}\}
\]
 containing $0$.

Moreover, if $f^{\circ m}\neq\id$ for all $m\in\mathbb{N}$, then
$f$ is linearisable if and only if $0$ lies in the interior $K^{\circ}$
of any/every Siegel compactum $K$ for $f$.
\end{thm}

\begin{defn}
A Siegel compactum for an elliptic germ $f\in\End(\mathbb{C},0)$
is called a \emph{hedgehog\index{hedgehog@\emph{hedgehog}}} for $f$
if $0$ is a Cremer point.
\end{defn}

\begin{thm}[{Pérez-Marco \cite[Thm.~5]{PerezMarco1994TopologyofJuliaSetsandHedgehogs}}]
\label{thm:PerezMarcoCremerHedgehog}Let $K$ be a hedgehog on $U$
for $f\in\End(\mathbb{C},0)$ with Cremer point $0$. Then
\begin{enumerate}
\item $K$ has empty interior $K^{\circ}=\emptyset$.
\item $K$ is not locally connected at any point except possibly $0$.
\item $K\backslash\{0\}$ has an uncountable number of connected components.
\item There exists a subsequence $\{n_{k}\}_{k\in\mathbb{N}}\subseteq\mathbb{N}$
such that $\{f^{\circ n_{k}}\}_{k}$ converges uniformly on $K$ to
$\id_{K}$.
\item \label{enu:PerMarAccumulation}There are orbits in $K$ accumulating
at all points of $K$ (in fact almost every orbit with respect to
the harmonic measure with respect to $\mathbb{C}\backslash K$ does).
\item Every point in $K$ is an accumulation point of its orbit.
\item No orbit of $f$ in $U\backslash K$ has an accumulation point in
$K$.
\item In particular, there is no orbit of $f$ non-trivially converging
to $0$. 
\end{enumerate}
\end{thm}

\begin{thm}[Biswas \cite{Biswas2010CompleteConjugacyInvariantsofNonlinearizableHolomorphicDynamics}]
Two non-linearisable elliptic germs $f,g\in\End(\mathbb{C},0)$ are
holomorphically conjugate by a germ $\varphi\in\Aut(\mathbb{C},0)$
if and only if they have the same multiplier and $\varphi$ maps a
hedgehog $K_{f}$ of $f$ to a hedgehog $K_{g}$ of $g$.
\end{thm}

Hence the multiplier and the conformal class of (sufficiently small)
hedgehogs form a complete set of holomorphic invariants for non-linearisable
elliptic germs.

\begin{rem}[Semi-local properties]
Hedgehog phenomena can also occur semi-locally. Let $K$ be a Siegel
compactum on $U$ for a neutral germ $f\in\End(\mathbb{C},0)$ with
multiplier $e^{2\pi i\theta}$ for $\theta\in\mathbb{R}$. If $0$
is a Siegel point and $U$ small enough, then $K$ is the closure
of the maximal Siegel disk in $U$. For larger admissible domains
$U$, this may no longer be the case.  If $0$ is a Siegel point
and $K$ is not the closure of a Siegel disk, then $K$ is called
a \emph{linearisable hedgehog} and its boundary again exhibits a rich
and complicated structure. For example, we have the following:
\begin{enumerate}
\item For $\mu_{K}$-almost every $z\in K$, we have $\partial K=\overline{\{z_{n}\}_{n\in\mathbb{Z}}}$,
where $\mu_{K}$ is the harmonic measure with respect to $\mathbb{C}\backslash K$
 (\cite[Thm.~IV.2.3]{PerezMarco1997FixedPointsandCircleMaps}). (Compare
Part~\ref{enu:PerMarAccumulation} of Theorem~\ref{thm:PerezMarcoCremerHedgehog}.)
\item Similarly to Theorems~\ref{thm:BrjunoYoccoz1D} and \ref{thm:WeakBrjuno},
there is a sharp arithmetic condition $\mathcal{H}$ on $\theta$,
called the Herman condition, that ensures that all Siegel compacta
for $f$ are (closures of) Siegel disks (\cite{Herman1985AreThereCriticalPointsontheBoundariesofSingularDomains},
\cite{PerezMarco1997FixedPointsandCircleMaps}) and a weaker condition
$\mathcal{H}'$ ensuring the same if $f$ has no small cycles (\cite{PerezMarco1997FixedPointsandCircleMaps}).
$\mathcal{H}$ is complicated to state, but explicitly known. $\mathcal{H}'$
can be shown to be ``arithmetic'' in a certain sense, but no arithmetic
definition is known (\cite{PerezMarco1997FixedPointsandCircleMaps}).
We have the following implications: \begin{equation*}
    \begin{tikzpicture}[baseline = (current bounding  box.center)]
        \matrix (m) [matrix of math nodes,
        row sep=0em, column sep=1em,
		text height=1.5ex, text depth=0.25ex,text width=1em, align=center]
        {
		  \mathcal{D} & \mathcal{H} & \mathcal{B} & \mathcal{B}' & \theta\text{ irrational},\\
		  && \mathcal{H}'\\
        };
        \path[font=\normalsize]
		(m-1-1) edge[-implies,double equal sign distance] (m-1-2)
		(m-1-2) edge[-implies,double equal sign distance] (m-1-3)
		(m-1-2) edge[-implies,double equal sign distance] (m-2-3)
		(m-1-3) edge[-implies,double equal sign distance] (m-1-4)
		(m-2-3) edge[-implies,double equal sign distance] (m-1-4)
		(m-1-4) edge[-implies,double equal sign distance] (m-1-5)
		;
    \end{tikzpicture}
\end{equation*} where $\mathcal{D}$ is the Diophantine condition (\ref{eq:Diophantisch}),
$\mathcal{B}$ is the Brjuno condition (\ref{eq:1DBrjunoCond}) and
\textit{\emph{$\mathcal{B}'$ }}is condition (\ref{eq:weakBrjuno}).
The reverse implications are known to be false in the first row and
those involving $\mathcal{H}'$ are conjectured to be false in \cite{PerezMarco1997FixedPointsandCircleMaps}.
\end{enumerate}
\end{rem}

\section{\label{sec:1DFatouClassification}Classification of Fatou components}

\selectlanguage{british}

In this section, we obtain Fatou components from local dynamics, and
give an account of the classification of Fatou components of rational
and entire functions.

\subsection{Fatou components from local dynamics}

Several types of Fatou components in one variable arise directly from
local fixed point theory and Montel's theorem~\ref{thm:MontelBdd}.
For this section, let $X$ be a Riemann surface, such as $\mathbb{C}$
or the Riemann sphere $\mathbb{P}$.

\begin{lem}
For an attracting fixed point $p\in X$ of $f\in\End(X)$, the connected
components of the realm of attraction $A_{f}(p)$ are Fatou components
of $f$.
\end{lem}

\begin{proof}
The proof of Lemma~\ref{lem:1DhypDynamics} yields a bounded connected
attracting neighbourhood $U$ of $p$. By Montel's theorem~\ref{thm:MontelBdd},
$\{f^{\circ n}\}_{n}$ is normal on $U$ (compare Remark~\ref{rem:StableCompFatou}).
Normality extends to the the realm of attraction as the iterated preimage
\[
A_{f}(p)=\bigcup_{n\in\mathbb{N}}f^{-n}(U).
\]
Hence each connected component $C$ of $A_{f}(p)$ is contained in
a Fatou component $C^{+}$ of $f$. Since any limit function of $\{f^{\circ n}\}$
on the open set $C$ is the constant function $p$, by the identity
principle, so is every limit function on $C^{+}$, i.e.~$C^{+}\subseteq A_{f}(p)$.
But since $C$ was a connected component of $A_{f}(p)$, we have $C=C^{+}$.
\end{proof}
In the same way we can use Theorem~\ref{thm:LeauFatou} and Montel's
theorem~\ref{thm:MontelBdd} to prove:
\begin{lem}
For a parabolic fixed point $p\in X$ of $f\in\End(X)$ and an attracting
direction $v$ of $f$ at $p$, the connected components of the parabolic
basin $\Omega_{v}$ of $f$ at $p$ centred at $v$ are Fatou components.
\end{lem}

We want to single out the main component near the fixed point:
\begin{defn}
Let $f\in\End(X)$ with fixed point $p\in X$. 
\begin{enumerate}
\item If $p$ is attracting, then the connected component of the realm of
attraction $A_{f}(p)$ containing $p$ is called the \emph{immediate
basin of attraction}\index{immediate basin of attraction} of $p$.
\item If $p$ is parabolic with a parabolic basin $\Omega_{v}$ centred
at an attracting direction $v$ of $f$ at $p$, then then the unique
connected component of $\Omega_{v}$ containing $p$ in its boundary
is called the \index{immediate parabolic basin}\emph{immediate parabolic
basin} of $f$ at $p$ centred at $v$.
\end{enumerate}
\end{defn}

In the elliptic case, we get, again using Montel's theorem~\ref{thm:MontelBdd}
and the identity principle: 
\begin{lem}
For a Siegel point $p\in X$ of $f\in\End(X)$ the maximal Siegel
disk for $f$ at $p$ is a Fatou component.
\end{lem}

For repelling fixed points and Cremer points, Montel's theorem~\ref{thm:MontelBdd}
and, respectively, Lemma~\ref{lem:1DhypDynamics} and Theorem~\ref{thm:PerezMarcoCremerHedgehog}
rule out Fatou components arising from the stable set.

\subsection{Rational functions}

We interpret rational functions as the holomorphic self-maps of $\mathbb{P}=\mathbb{C}\cup\{\infty\}$,
so normality is determined only by the first part of Definition~\ref{def:normality},
i.e.\ by locally uniform convergence. The Fatou components of rational
functions have been classified by Sullivan in \cite{Sullivan1983ConformalDynamicalSystems}
and \cite{Sullivan1985QuasiconformalHomeomorphismsandDynamicsISolutionoftheFatouJuliaProblemonWanderingDomains}.
We give a short summary following \cite[§16]{Milnor2006Dynamicsinonecomplexvariable}.
The classification of invariant Fatou components is the fruit of efforts
of several authors, including Fatou and Julia:
\begin{thm}
\label{thm:FatouCompRational}Let $f:\mathbb{P}\to\mathbb{P}$ be
a rational map of degree $\deg f\ge2$ and $U\subseteq\mathbb{P}$
be an invariant \emph{(}$f(U)=U$\emph{)} Fatou component of $f$.
Then we have one of the following:
\begin{enumerate}
\item $U$ is the immediate basin of attraction of an attracting fixed point
$p\in\mathbb{P}$.
\item $U$ is an immediate parabolic basin of a parabolic fixed point $p\in\mathbb{P}$
with multiplier $1$.
\item $U$ is a Siegel disk.
\item $U$ is a Herman ring\index{Herman ring}, i.e.\ there exists a biholomorphism
of $U$ to an annulus $A$, conjugating $f$ to an irrational rotation
of $A$.
\end{enumerate}
\end{thm}

\begin{rem}
Only Herman rings do not arise from the local theory. As a consequence
of the maximum principle, polynomials (and entire functions) cannot
have Herman rings. More precisely, both connected components of the
complement in $\mathbb{P}$ of a Herman ring must contain at least
one pole.
\end{rem}

\begin{proof}[Proof sketch]
An invariant Fatou component must be hyperbolic. Since rational maps
have only finitely many fixed points and no iterate is the identity,
Denjoy-Wolff theory tells us that $U$ is either 
\begin{enumerate}
\item a rotation domain, i.e. there exists an injective holomorphic map
$\phi:U\to\mathbb{C}$ conjugating $f$ to an irrational rotation.
\item attracting to a fixed point $p\in\overline{U}$, i.e.\ all orbits
in $U$ converge to $p$.
\end{enumerate}
A hyperbolic rotation domain has to be biholomorphic to $\mathbb{D}$,
$\mathbb{D}^{*}=\mathbb{D}\backslash\{0\}$, or an annulus. We can
rule out $\mathbb{D}^{*}$, as the biholomorphic inverse $\phi^{-1}:\mathbb{D}^{*}\to U$
of the conjugating map would extend through $0$ leading to a Siegel
disk.

Though the classification in the attracting case follows from the
local fixed point theory described in the preceding sections, before
Pérez-Marco's Hedgehog theory was available, it had classically been
derived from the Snail lemma (\cite{Sullivan1983ConformalDynamicalSystems},
\cite{DouadyHubbard1984EtudeDynamiqueDesPolynomesComplexesPartieI,DouadyHubbard1985OntheDynamicsofPolynomiallikeMappings},
\cite{Lyubich1986DynamicsofRationalTransformationsTopologicalPicture}):
\begin{lem}[Snail lemma]
\label{lem:SnailLemma}Let $f\in\End(\mathbb{C},0)$ and $\gamma:[0,+\infty)\to\mathbb{P}\backslash\{0\}$
a continuous ray satisfying 
\[
f(\gamma(t))=\gamma(t+1)\quad\text{for all }t\in\mathbb{R}^{+}
\]
and $\gamma(t)\xrightarrow[t\to+\infty]{}0$. Then the multiplier
of $f$ at $0$ is either attracting or equal to $1$.
\end{lem}

Now, if $U$ is attracting to $p\in\overline{U}$, a path $\gamma_{0}:[0,1]\to U$
from any point $z\in U$ to its image $f(z)\in U$ extends to a continuous
ray $\gamma$ by patching together the paths $\gamma_{0},f\circ\gamma_{0},f^{\circ2}\circ\gamma_{0},\ldots$
and $\gamma(t)\xrightarrow[t\to+\infty]{}p$. So the Snail lemma~\ref{lem:SnailLemma}
implies, that $U$ falls into one of the first two cases.
\end{proof}
The possible topologies of basins of attraction or parabolic basins
in the rational case fall into two categories:
\begin{thm}[{\cite[Thm.~8.9 \& Prob.~10-f]{Milnor2006Dynamicsinonecomplexvariable}}]
\label{thm:basinConnectivity}Let $f\in\End(\mathbb{P})$ be a rational
function. Then the immediate basin of attraction of an attracting
fixed point of $f$ and any immediate parabolic basin at a parabolic
fixed point of $f$ are each either simply connected or infinitely
connected, i.e.\ their complement in $\mathbb{P}$ has infinitely
many connected components.
\end{thm}

The classification was completed by Sullivan's celebrated non-wandering
theorem, that rules out the existence of non-wandering Fatou components
in the following sense:
\begin{defn}
\label{def:WanderingFatComp}A Fatou component $U$ of a map $F\in\End(X)$
of a complex manifold $X$ is:
\begin{enumerate}
\item \emph{\index{periodic Fatou component@\emph{periodic Fatou component}}periodic},
if $F^{\circ n}(U)=U$ for some $n\in\mathbb{N}$,
\item \emph{\index{pre-periodic Fatou component@\emph{pre-periodic Fatou component}}pre-periodic}
or \emph{non-wandering}, if $F^{\circ n}(U)$ is periodic for some
$n\in\mathbb{N}$,
\item \emph{wandering}\index{wandering Fatou component@\emph{wandering Fatou component}},
if it is not pre-periodic.
\end{enumerate}
\end{defn}

\begin{thm}[Sullivan \cite{Sullivan1985QuasiconformalHomeomorphismsandDynamicsISolutionoftheFatouJuliaProblemonWanderingDomains}]
\label{thm:NonWanderingThm}Every Fatou component of a rational map
$f\in\End(\mathbb{P})$ is pre-periodic.
\end{thm}

Sullivan applies Beltrami forms and the straightening of almost complex
structures in his proof. A different approach due to Epstein employs
a density theorem of Bers for quadratic differentials and is presented
in \cite{Buff2018WanderingFatouComponentforPolynomials}. 

It follows that every Fatou component of a rational map $f$ is either
periodic, and hence, up to replacing $f$ by an iterate, of one of
the four types in Theorem~\ref{thm:FatouCompRational} (with topologies
given by Theorem~\ref{thm:basinConnectivity}), or some iterate $f^{\circ k}$
maps it to such a periodic component, either biholomorphically or
as a branched covering.

\subsection{Transcendental functions}

Transcendental functions do not extend to $\mathbb{P}$, so we have
to consider both parts of Definition~\ref{def:normality}: locally
uniform convergence and divergence. This leads to one more type of
invariant Fatou components in addition to those from Theorem~\ref{thm:FatouCompRational}:
\begin{defn}
Let $f\in\End(X)$ for a complex manifold $X$. An invariant Fatou
component $U$ of $f$ such that $f$ diverges locally uniformly from
$X$ on $U$ is called a \emph{Baker domain}.
\end{defn}

Fatou showed in \cite{Fatou1919SurLesEquationsFonctionnellesI}, that
for the map $f\in\End(\mathbb{C})$ 
\[
f(z)=z+1+e^{-z},
\]
the right half plane is contained in a Baker domain. 

Even though all orbits in a Baker domain converge to the point $\infty$
in $\mathbb{P}$, the dynamics are very different from an attracting
or parabolic basin, as, instead of a fixed point, the orbits converge
to an essential singularity at $\infty$. 

The topology of invariant Fatou components is simpler than in the
rational case, thanks to the absence of poles:
\begin{thm}[\cite{Baker1984WanderingDomainsintheIterationofEntireFunctions}]
Let $f\in\End(\mathbb{C})$ be an entire function. Then any periodic
Fatou component of $f$ is simply connected.
\end{thm}

In conclusion, the classification of invariant Fatou components for
entire functions is that of Theorem~\ref{thm:FatouCompRational},
replacing Herman rings by Baker domains.

This is far from the end of the story in the transcendental case,
as the non-wandering theorem~\ref{thm:NonWanderingThm} does not
hold in this case.  A variety of wandering Fatou components have
been constructed, but the general picture is not completely understood.
For example, we can differentiate wandering components by their orbit
behaviour:
\begin{defn}
\label{def:WanderingTypes}Let $F\in\End(X)$ for a complex manifold
$X$. A wandering Fatou component $U$ of $F$ is 
\begin{enumerate}
\item \emph{escaping}, if all orbits diverge from $X$,
\item \emph{oscillating}, if $\bigcup_{n\in\mathbb{N}}F^{\circ n}(U)$ contains
an unbounded orbit and an orbit with a bounded subsequence.
\item \emph{orbitally bounded}, if every $z\in U$ has a bounded orbit.
\end{enumerate}
\end{defn}

Bergweiler constructs transcendental functions with escaping wandering
Fatou components in \cite{Bergweiler1993IterationofMeromorphicFunctions}.
Eremenko and Lyubich \cite{EreemenkoLjubich1987ExamplesofEntireFunctionswithPathologicalDynamics},
and Bishop \cite{Bishop2015ConstructingEntireFunctionsbyQuasiconformalFolding}
construct oscillating examples. The existence of orbitally bounded
wandering components for functions in $\End(\mathbb{C})$ is open.

More details and recent results on wandering Fatou components can
be found in Schleicher's paper \cite{Schleicher2010DynamicsofEntireFunctions}.

\cleardoublepage{}

\chapter{\label{chap:2DlocDyn}Several complex variables}

In this chapter we present a systematic overview of results in local
dynamics in several complex variables near a fixed point and the implications
for global objects that can be deduced from local considerations.
This includes some of our own contributions. After sections on formal
and holomorphic normal forms, the remaining sections are organised
according to the linear part of the map at the fixed point. Compare
and contrast the survey articles \cite{Bracci2004Localdynamicsofholomorphicdiffeomorphisms},
\cite{Abate2010Discreteholomorphiclocaldynamicalsystems}, \cite{Bracci2011Parabolicattitude},
and \cite{Rong2015Abriefsurveyonlocalholomorphicdynamicsinhigherdimensions},
that provide more details and proofs in many cases, but do not systematically
cover all the mixed cases.

\section{Multipliers, resonances, and normal forms}

\selectlanguage{british}
\global\long\def\Transp{\mathsf{T}}%
\global\long\def\Pow{\m{Pow}}%
\global\long\def\ord{\m{ord}}%

As before, the first object of interest is the linear part and, more
specifically, its eigenvalues.
\begin{defn}
Let $F\in\Pow(\mathbb{C}^{d},0)$. Then the \emph{homogeneous expansion\index{homogeneous expansion}}
of $F$ is the unique representation 
\begin{equation}
F(z)=P_{1}(z)+P_{2}(z)+P_{3}(z)+\ldots,\label{eq:homogenExp}
\end{equation}
where $P_{k}$ is a homogeneous polynomial of degree $k$ (with coefficients
in $\mathbb{C}^{d}$) for each $k\ge1$. The eigenvalues of the linear
part $dF_{0}=P_{1}$ are called \emph{\index{multiplier}multipliers}
of $F$. 
\end{defn}

We may assume $dF_{0}$ to be in Jordan normal form. Like in one dimension,
the dynamics of a germ $F$ are described at the first order by its
linear part. Depending on the multipliers and the Jordan block structure
of $dF_{0}$, in some cases both dynamics are the same.
\begin{note*}
The multipliers and Jordan block structure of $dF_{0}$ are invariants
under holomorphic and formal conjugation, but not under topological
conjugation.
\end{note*}
As in the one-dimensional situation, multipliers fall in one of three
categories:
\begin{defn}
A multiplier $\lambda$ of a (formal) germ $F\in\Pow(\mathbb{C}^{d},0)$
is called
\begin{enumerate}[noitemsep]
\item \emph{hyperbolic}\index{hyperbolic multiplier/germ}, if $|\lambda|\neq1$,
\begin{enumerate}[noitemsep]
\item \emph{attracting}\index{attracting multiplier/germ}, if $|\lambda|<1$.
\begin{enumerate}[noitemsep]
\item \emph{geometrically attracting}, if $0<|\lambda|<1$.
\item \emph{superattracting}, if $\lambda=0$.
\end{enumerate}
\item \emph{repelling}\index{repelling multiplier/germ}, if $|\lambda|>1$.
\end{enumerate}
\item \emph{neutral}\index{neutral fixed point/germ}, if $|\lambda|=1$,
\begin{enumerate}[noitemsep]
\item \emph{parabolic}\index{parabolic multiplier}, if $\lambda$ is a
root of unity.
\item \emph{elliptic}\index{elliptic multiplier}, if $|\lambda|=1$, but
$\lambda$ is not a root of unity.
\end{enumerate}
\end{enumerate}
The germ $F$ (or the fixed point $0$ of $F$) is called \emph{hyperbolic\slash parabolic\slash elliptic\slash}etc.\index{hyperbolic germ}\index{parabolic germ}\index{elliptic germ},
if all its multipliers are.
\end{defn}

In several variables, a germ has more than one multiplier, so the
trichotomy of hyperbolic, parabolic and elliptic germs no longer applies.
As it turns out, even purely elliptic germs can exhibit ``parabolic
behaviour'', so it may not even be the best way to generalise the
one variable definitions.

\selectlanguage{british}

In one complex variable, obstructions to formal linearisation of a
Germ $F\in\End_{\lambda}(\mathbb{C})$ are given by vanishing denominators
of the form $\lambda-\lambda^{k}$, which only occur in the parabolic
and superattracting case. In several variables, formal linearisations
involve additional ``mixed'' denominators, which can vanish even
in the invertible hyperbolic and elliptic cases.

For concise formulas, we use multi-index notation\index{multi-index notation}:
\begin{notation}
For tuples $\lambda=(\lambda_{1},\ldots,\lambda_{d})\in\mathbb{C}^{d}$
and multi-indices $\alpha=(\alpha_{1},\ldots,\alpha_{d})\in\mathbb{N}^{d}$,
set 
\[
\lambda^{\alpha}:=\prod_{j=1}^{d}\lambda_{j}^{\alpha_{j}}\quad\text{and}\quad|\alpha|=\sum_{j=1}^{d}\alpha_{j}.
\]
Moreover, denote by $e_{j}$ the multi-index with $1$ in the $j$-th
component and $0$ in all others.
\end{notation}

The construction of formal conjugations works analogously to the one
variable case: Elimination of a monomial $z^{\alpha}$ in the $j$-th
component $F_{j}$ of $F$ requires a denominator $\lambda_{j}-\lambda^{\alpha}$.
However, in this case, even for well-behaved (e.g. geometrically attracting)
germs, these denominators can vanish due to arithmetic relations between
multipliers. This phenomenon is known as \emph{resonance}.
\begin{defn}
\label{def:resonance}Let $F\in\Pow(\mathbb{C}^{d},0)$ be a (formal)
germ such that $dF_{0}$ is in Jordan normal form with multipliers
$\lambda_{1},\ldots,\lambda_{d}$ occurring in order on the diagonal.
\begin{enumerate}
\item A \emph{resonance\index{resonance}} \emph{for the tuple $\lambda=(\lambda_{1},\ldots,\lambda_{d})$}
or a \emph{resonance for $F$} is a relation of the form 
\begin{equation}
\lambda_{j}=\lambda^{\alpha}\quad\text{with }\alpha\neq e_{j},1\le j\le d.\label{eq:resonance}
\end{equation}
The \emph{order\index{order!of a resonance}} of the resonance (\ref{eq:resonance})
is $|\alpha|$.
\item A monomial $z^{\beta}=z_{1}^{\beta_{1}}\cdots z_{d}^{\beta_{d}}$
in the $j$-th component of $F$ is called \emph{resonant}\index{resonant monomial}
(w.r.t. $(\lambda_{1},\ldots,\lambda_{d})$), if $\lambda_{j}=\lambda^{\beta}$.
\end{enumerate}
\end{defn}

\begin{note*}
By the above definition, the trivial identity $\lambda_{j}=\lambda^{e_{j}}$
is not considered a resonance for $F$, whereas the always present
term $z_{j}$ in the $j$-th component of $F$ is considered a resonant
monomial. This is so that the statements ``$F$ has no resonances''
and ``$F$ has no resonant monomials'' both exclude the obvious
exceptions.
\end{note*}
The best result on formal linearisation from Poincaré-Dulac theory
(see \cite[\textsection 25.B]{Arnold1988GeometricalmethodsinthetheoryofordinarydifferntialequationsTranslfromtheRussianbyJosephSzucsEnglTransledbyMarkLevi2nded}
and \cite[Thm.\ 4.21]{IlyashenkoYakovenko2008Lecturesonanalyticdifferentialequations}
for proofs) is 
\begin{thm}[Poincaré \cite{Poincare1879SurLesProprietesDesFonctionsDefiniesParLesEquationsAuxDifferencesPartielles},
Dulac \cite{Dulac1912SolutionsDunSystemeDequationsDifferentiellesDansLeVoisinageDeValeursSingulieres}]
\label{thm:PoincareDulac}For every $F\in\Pow(\mathbb{C}^{d},0)$
there exists a conjugating series $H\in\Pow(\mathbb{C}^{d},0)$ to
a series 
\[
G=H^{-1}\circ F\circ H
\]
without resonant monomials and $dG_{0}$ in Jordan normal Form. For
$dF_{0}=dG_{0}$ in Jordan normal form, $H$ is unique up to arbitrary
choices of the monomials resonant w.r.t $F$. 
\end{thm}

\begin{defn}
A germ $G$ as in Theorem~\ref{thm:PoincareDulac} is called a \emph{\index{Poincaré-Dulac normal form}Poincaré-Dulac
normal form} of $F$.
\end{defn}

It immediately follows:
\begin{cor}[Poincaré \cite{Poincare1879SurLesProprietesDesFonctionsDefiniesParLesEquationsAuxDifferencesPartielles}]
\label{cor:PoincNoResLin}If $F\in\Pow(\mathbb{C}^{d},0)$ has no
resonances of order at least $2$, then it is formally linearisable.
If $F$ has no resonances at all, then the linear part $dF_{0}$ is
moreover diagonalisable.
\end{cor}

\begin{proof}[Proof of Theorem~\ref{thm:PoincareDulac}]
Assume $F(z)=\Lambda z+f(z)=\sum_{|\alpha|\ge1}\sum_{j=1}^{d}f_{\alpha}^{j}z^{\alpha}e_{j}$
with $\Lambda=dF_{0}$ in Jordan normal form and eigenvalues $\lambda_{1},\ldots,\lambda_{d}$.
Consider formal series 
\begin{align*}
G(z) & =\Lambda z+g(z)=\sum_{|\alpha|\ge1}g_{\alpha}z^{\alpha}\quad\text{and}\quad H(z)=z+h(z)=\sum_{|\alpha|\ge1}h_{\alpha}z^{\alpha}.
\end{align*}
Then $H$ conjugates $F$ to $G$, if and only if $H$ and $G$ solve
the homological equation $F\circ H=H\circ G$, explicitly
\[
\sum_{|\alpha|\ge1}f_{\alpha}\paren[\Big]{\sum_{|\beta|\ge1}h_{\beta}z^{\beta}}^{\alpha}=\sum_{|\alpha|\ge1}h_{\alpha}\paren[\Big]{\sum_{|\beta|\ge1}g_{\beta}z^{\beta}}^{\alpha},
\]
or, comparing coefficients:
\begin{equation}
(\lambda^{\alpha}\id-\Lambda)h_{\alpha}=f_{\alpha}-g_{\alpha}+\sum_{2\le k<|\alpha|}\sum_{j_{1}\le\cdots\le j_{k}}\sum_{\beta_{1}+\cdots+\beta_{k}=\alpha}(f_{e_{J}}h_{\beta_{1}}^{j_{1}}\cdots h_{\beta_{k}}^{j_{k}}-h_{e_{J}}g_{\beta_{1}}^{j_{1}}\cdots g_{\beta_{k}}^{j_{k}})\label{eq:HomolExplicit-full}
\end{equation}
for $\alpha\in\mathbb{N}^{d}\backslash\{0\}$, where $e_{J}:=e_{j_{1}}+\cdots+e_{j_{k}}$.
The sum on the right hand side contains only coefficients with index
of order less than $|\alpha|$. and the $j$-th component of the left
hand side is 
\[
(\lambda^{\alpha}-\lambda_{j})h_{\alpha}^{j}+\varepsilon_{j}h_{\alpha}^{j-1},
\]
with $\varepsilon_{j}\in\{0,1\}$. Hence, proceeding by induction
in lexicographic order over $(|\alpha|,j)$, whenever $\lambda^{\alpha}-\lambda_{j}\neq0$,
all coefficients of order less than $|\alpha|$ and $h_{\alpha}^{j'}$
with $j'<j$ are given and we can choose $h_{\alpha}^{j}$ (uniquely)
to ensure $g_{\alpha}^{j}=0$. Whenever $\lambda^{\alpha}-\lambda_{j}=0$,
we can freely choose $h_{\alpha}^{j}$. The resulting $G$ has Poincaré-Dulac
normal form by construction.
\end{proof}
Poincaré-Dulac normal forms are not unique in general. In particular,
a divergent normal form does not rule out the existence of a convergent
or even a polynomial normal form. In \cite{AbateRaissy2013FormalPoincareDulacRenormalizationforHolomorphicGerms},
Abate and Raissy describe an approach to unique formal normal forms.
On the other hand, there is a straightforward description of conjugations
preserving normal form:
\begin{prop}[{Raissy \cite[Prop.~1.3.26]{Raissy2010GeometricalMethodsintheNormalizationofGermsofBiholomorphismsDottoratoinMatematica22CicloPhDThesis}}]
\label{prop:ConjPreservingPoiDul}Let $F\in\Pow(\mathbb{C}^{d},0)$
be in Poincaré-Dulac normal form and $H\in\Pow_{1}(\mathbb{C}^{d},0)$.
Then $G=H^{-1}\circ F\circ H$ is in Poincaré-Dulac normal form if
and only if $H$ contains only resonant monomials w.r.t. $F$.
\end{prop}

In the formally linearisable case, we have uniqueness (up to linear
conjugation):

\begin{lem}[Rüssmann \cite{Ruessmann2002StabilityofEllipticFixedPointsofAnalyticAreaPreservingMappingsundertheBrunoCondition}]
\label{lem:LinPDNFunique}If $F\in\End(\mathbb{C}^{d},0)$ is formally
linearisable, then every Poincaré-Dulac normal form of $F$ is linear.
\end{lem}

This applies in particular when $F$ has no resonances (of order at
least $2$) by Corollary~\ref{cor:PoincNoResLin}.

\section{Hyperbolic fixed points}

\selectlanguage{british}

A hyperbolic germ has only attracting or repelling multipliers. This
is true for generic germs and it is the first case to be studied in
detail. Traditionally, a hyperbolic germ is said to be in the \emph{Poincaré
domain}\index{Poincaré domain}, if it is invertible and either attracting
or repelling. Otherwise, it is said to be in the \emph{\index{Siegel domain}Siegel
domain}. The Poincaré and Siegel domains are named after the authors
of the relevant linearisation theorems \ref{thm:PoincareLinIfNoRes}
and \ref{thm:SiegelSternberg}.

\subsection{The Poincaré domain}

Germs in the Poincaré domain were first studied by the eponymous Henri
Poincaré \cite{Poincare1879SurLesProprietesDesFonctionsDefiniesParLesEquationsAuxDifferencesPartielles,Poincare1886SurLesCourbesDefiniesParUneEquationDifferentielleIV,Poincare1890SurLeProblemeDesTroisCorpsEtLesEquationsDeLaDynamique}
(republished in \cite{Poincare1928vresTome1EquationsDifferentielles}).
As in one variable, the basic dynamics of a purely attracting or repelling
germ correspond to the topological definitions and the topological
dynamics characterise each case. The proof of Lemma~\ref{lem:1DhypDynamics}
applies almost verbatim, with the contraction principle (\cite[Thm.~6.3.5]{MorosawaNishimuraTaniguchiUeda2000HolomorphicDynamicsTranslfromtheJapanese})
taking the place of Schwarz's lemma to show:
\begin{lem}
\label{lem:AttrIffTopAttr}A germ $F\in\End(\mathbb{C}^{d},0)$ is
attracting\slash repelling if and only if it is topologically attracting\slash repelling.
\end{lem}

In fact, we can always linearise these germs (see Theorem~\ref{thm:GrobmanHartman})
in the topological category. In the holomorphic category, this is
no longer true due to resonances, but those are the only obstruction:
\begin{thm}[Poincaré \cite{Poincare1879SurLesProprietesDesFonctionsDefiniesParLesEquationsAuxDifferencesPartielles}]
\label{thm:PoincareLinIfNoRes}If $F\in\End(\mathbb{C}^{d},0)$ is
a hyperbolic germ in the Poincaré domain and has no resonances, then
$F$ is holomorphically linearisable.
\end{thm}

Poincaré's proof relies on majorant series and is a special case of
Theorem~\ref{thm:Brjuno}, but the estimates involving small divisors
are much simpler in the Poincaré domain.

If we allow resonances, a simple but crucial observation is the following:
Assume $F\in\End(\mathbb{C}^{d},0)$ is 
\begin{lem}
\label{lem:PoincDomFiniteRes}A hyperbolic germ $F\in\End(\mathbb{C}^{d},0)$
in the Poincaré domain can have only finitely many resonances. 
\end{lem}

\begin{proof}
Assume, without loss of generality, that $F$ is attracting with multipliers
$\lambda_{1},\ldots,\lambda_{d}$. Then there exists $k\in\mathbb{N}$
such that 
\[
(\max_{j}|\lambda_{j}|)^{k}<\min_{j}|\lambda_{j}|,
\]
and hence $|\lambda^{\alpha}|<|\lambda_{j}|$, whenever $\alpha\in\mathbb{N}^{d}$
with $|\alpha|\ge k$ and $j\in\{1,\ldots,d\}$. 
\end{proof}
This shows in particular, that any Poincaré-Dulac normal form in the
Poincaré domain is finite, i.e.\ polynomial, and it turns out, that
one can always find at least one convergent normalisation:
\begin{thm}[Poincaré \cite{Poincare1886SurLesCourbesDefiniesParUneEquationDifferentielleIV,Poincare1890SurLeProblemeDesTroisCorpsEtLesEquationsDeLaDynamique},
Dulac \cite{Dulac1912SolutionsDunSystemeDequationsDifferentiellesDansLeVoisinageDeValeursSingulieres}]
\label{thm:PoincareNormalisationPolynom}Every hyperbolic germ $F\in\End(\mathbb{C}^{d},0)$
in the Poincaré domain is holomorphically conjugated to one of its
(polynomial) Poincaré-Dulac normal forms.
\end{thm}

\begin{rem}
Both finiteness of all Poincaré-Dulac normal forms and the numerical
properties of the multipliers are crucial parts of the proof, so it
does not easily extend easily to a wider class of germs.
\end{rem}

By Lemma~\ref{lem:LinPDNFunique}, the above implies in particular:
\begin{cor}
A hyperbolic germ $F\in\End(\mathbb{C}^{d},0)$ in the Poincaré domain
is holomorphically linearisable, if and only if it is formally linearisable.
\end{cor}

\begin{rem}
Lemma~\ref{lem:PoincDomFiniteRes} implies that a germ in the Poincaré
domain only has a finite number of Poincaré-Dulac normal forms. In
\cite{Reich1969DasTypenproblemBeiFormalBiholomorphenAbbildungenMitAnziehendemFixpunkt},
Reich describes unique formal normal forms in the Poincaré domain. 
\end{rem}

In \cite{Reich1969NormalformenBiholomorpherAbbildungenMitAnziehendemFixpunkt},
he concludes that these are also holomorphic normal forms, by demonstrating
that in the Poincaré domain every formally normalising series converges.
However, Rosay and Rudin give a counter example to an important part
of his argument in \cite{RosayRudin1988HolomorphicmapsfrombfCntobfCn}.
In the appendix of the same paper, they show that every germ in the
Poincaré domain is holomorphically conjugate to a lower-triangular
normal form. Studying the dynamics of these special normal forms they
conclude in particular:
\begin{thm}[Rosay, Rudin \cite{RosayRudin1988HolomorphicmapsfrombfCntobfCn}]
\label{thm:AttrBasinFatBieb}If $F\in\Aut(\mathbb{C}^{d},0)$ is
attracting, then $F$ is conjugate to an upper-triangular polynomial
normal form. If $F$ extends to a global automorphism of $\mathbb{C}^{d}$,
then the realm of attraction $A_{F}(0)$ is biholomorphic to $\mathbb{C}^{d}$
and $F$ is conjugate to this normal form on all of $A_{F}(0)$.
\end{thm}

A detailed exposition of Rosay and Rudin's method was given by Berteloot
in \cite{Berteloot2006MethodesDeChangementDechellesEnAnalyseComplexe}
(see also Raissy's PhD thesis \cite[§1.3.4]{Raissy2010GeometricalMethodsintheNormalizationofGermsofBiholomorphismsDottoratoinMatematica22CicloPhDThesis}).

Special cases of this result were proved by Fatou \cite{Fatou1922SurCertainesFonctionsUniformesDeDeuxVariables}
and Bieberbach \cite{Bieberbach1933BeispielZweierGanzerFunktionenZweierKomplexerVariablenWelcheEineSchlichteVolumtreueAbbildungDesR_4AufEinenTeilSeinerSelbstVermitteln}
to construct the first examples of proper subsets of $\mathbb{C}^{d},d\ge2$
biholomorphic to $\mathbb{C}^{d}$. Such domains cannot exist in $\mathbb{C}$
and are today commonly known as \emph{\index{Fatou-Bieberbach domain@\emph{Fatou-Bieberbach domain}}Fatou-Bieberbach
domains}.

\subsection{The Siegel Domain}

Germs in the Siegel domain may have an infinite number of resonances
and an infinite number of Poincaré-Dulac normal forms. The formal
classification is still open even in the invertible case. Jenkins
\cite{Jenkins2008FurtherReductionsofPoincareDulacNormalFormsinCn1}
proves some partial results. 

For a germ outside the Poincaré domain, holomorphic normalisation
is no longer guaranteed. If the germ is formally linearisable, we
discuss sufficient conditions in section~\ref{sec:BrjunoConditions},
otherwise we don't know of any holomorphic normalisation results.
On the other hand, smooth normalisation is always possible:
\begin{thm}[Sternberg \cite{Sternberg1957LocalContractionsandaTheoremofPoincare,Sternberg1958OntheStructureofLocalHomeomorphismsofEuclideannSpaceII},
Chaperon \cite{Chaperon1986GeometrieDifferentielleEtSingularitesDeSystemesDynamiques}]
\label{thm:hypSmoothNormalisation}Let $F,G\in\End(\mathbb{C}^{d},0)$
be hyperbolic and invertible. Then $F$ and $G$ are formally conjugate,
if and only if they are smoothly conjugate.
\end{thm}

For a hyperbolic (formal) germ $F\in\Pow(\mathbb{C}^{d},0)$ we have
a canonical splitting 
\[
\mathbb{C}^{d}=E^{\mathrm{s}}\oplus E^{\mathrm{u}},
\]
where $E^{\mathrm{s}}$ ($E^{\mathrm{u}}$) is the direct sum of the
generalised eigenspaces of $dF_{0}$ corresponding to its attracting
(repelling) eigenvalues. 

The extension of this splitting to the non-linear case is the stable
manifold theorem. It was first proved in $\mathcal{C}^{1}$ by \cite{Perron1929UberStabilitatUndAsymptotischesVerhaltenDerIntegraleVonDifferentialgleichungssystemen}
and \cite{Hadamard1901SurLiterationEtLesSolutionsAsymptotiquesDesEquationsDifferentielles}
(translated to English in \cite{Hasselblatt2017OnIterationandAsymptoticSolutionsofDifferentialEquationsbyJacquesHadamard}).
Proofs in $\mathcal{C}^{\infty}$ can be found in \cite{FathiHermanYoccoz1983AProofofPesinsStableManifoldTheorem},
\cite{HasselblattKatok1995IntroductiontotheModernTheoryofDynamicalSystems},
\cite{HirschPughShub1977InvariantManifolds}, \cite{Pesin1976FamiliesofInvariantManifoldsThatCorrespondtoNonzeroCharacteristicExponents},
\cite{Shub1987GlobalStabilityofDynamicalSystems}, and \cite{AbbondandoloMajer2006OntheGlobalStableManifold}.
The holomorphic case was covered in \cite{Wu1993ComplexStableManifoldsofHolomorphicDiffeomorphisms}
and a proof in the non-invertible case can be found in \cite{Abate2001AnIntroductiontoHyperbolicDynamicalSystems}.
\begin{thm}[Stable manifold theorem]
\label{thm:StableMnfThm}Let $F\in\End(\mathbb{C}^{d},0)$ be hyperbolic.
Then the realm of attraction $A_{F}$ and the realm of repulsion $A_{F}^{-}$
are completely $F$-invariant complex manifolds containing $0$ such
that $T_{0}A_{F}=E^{\mathrm{s}}$ and $T_{0}A_{F}^{-}=E^{\mathrm{u}}$.

Moreover there exists a neighbourhood $U\subseteq\mathbb{C}^{d}$
of $0$ such that for every $z\in U\backslash(A_{F}\cup A_{F}^{-})$
both the forward orbit and all of the backward orbits of $z$ escape
from $U$. In other words $A_{F}=\Sigma_{F}(U)$ and $A_{F}^{-}=\Sigma_{F}^{-}(U)$.
\end{thm}

\begin{rem}
In this case $W^{\mathrm{s}}=A_{F}=\Sigma_{F}$ and $W^{\mathrm{u}}=A_{F}^{-}=\Sigma_{F}^{-}$
are respectively called \emph{stable\index{stable manifold@\emph{stable manifold}}}
and \emph{unstable manifold}\index{unstable manifold@\emph{unstable manifold}}
of $F$.

By Theorem~\ref{thm:PoincareNormalisationPolynom}, the restriction
$F|_{W^{\mathrm{u}}}$ to the unstable manifold is conjugate to a
polynomial normal form and, if $F$ is invertible, so is the restriction
$F|_{W^{\mathrm{s}}}$ to the stable manifold.
\end{rem}

One can think of the stable and unstable manifolds as deformed coordinate
hyperplanes. For invertible germs, these can be topologically ``straightened
out'' yielding one possible proof of the classical Grobman-Hartman
theorem (see e.g.\ \cite{Shub1987GlobalStabilityofDynamicalSystems},
\cite{HasselblattKatok1995IntroductiontotheModernTheoryofDynamicalSystems},
or \cite{Abate2001AnIntroductiontoHyperbolicDynamicalSystems}). The
theorem was first proved by Grobman \cite{Grobman1959HomeomorphismofSystemsofDifferentialEquations,Grobman1962TopologicalClassificationofNeighborhoodsofaSingularityinnSpace}
and Hartman \cite{Hartman1960ALemmaintheTheoryofStructuralStabilityofDifferentialEquations}.
Local differentiability was proved by Guysinsky, Hasselblatt, and
Rayskin \cite{GuysinskyHasselblattRayskin2003DifferentiabilityoftheHartmanGrobmanLinearization}.
The non-invertible attracting case was examined by Hartman \cite{Hartman1960OnLocalHomeomorphismsofEuclideanSpaces}
and the general hyperbolic case by Aulbach and Garay \cite{AulbachGaray1994PartialLinearizationforNoninvertibleMappings}.

\begin{thm}[Grobman-Hartman]
\label{thm:GrobmanHartman}Let $F\in\End(\mathbb{C}^{d},0)$ be a
hyperbolic germ. 
\begin{enumerate}
\item If $F$ is invertible, then $F$ is topologically linearisable by
a local homeomorphism $\Phi$, that is (real) differentiable at the
origin with $\Phi(z)=z+o(\norm z)$ for $z\in\mathbb{C}^{d}$ near
$0$.
\item Let $m$ be the multiplicity of the multiplier $0$ of $F$. Then
$F$ is topologically conjugate to a continuous map
\begin{equation}
(z,w)\mapsto\begin{pmatrix}Lz\\
Nw+o(\norm w)
\end{pmatrix},\label{eq:GrobmanHartmanNilpot}
\end{equation}
where $(z,w)\in\mathbb{C}^{d-m}\times\mathbb{C}^{m}$, $L\in\mathbb{C}^{(d-m)\times(d-m)}$
is invertible, and $N\in\mathbb{C}^{m\times m}$ is nilpotent and
both matrices are in Jordan normal form.
\end{enumerate}
If $d=2$ or $F$ is attracting or repelling, then the local conjugations
(and hence the map (\ref{eq:GrobmanHartmanNilpot})) can be chosen
in $\mathcal{C}^{1}$.

\end{thm}

\begin{rem}
The conditions given for the last assertion are special cases of a
sufficient condition requiring that all attracting multipliers stay
close to each other and the repelling multipliers do the same (see
\cite{Hartman1960OnLocalHomeomorphismsofEuclideanSpaces}). Otherwise,
the regularity drops to $\alpha$-Hölder continuity with $0<\alpha<1$
depending on the arrangement of the multipliers (see \cite{GuysinskyHasselblattRayskin2003DifferentiabilityoftheHartmanGrobmanLinearization}).
\end{rem}

The precise topological classification of the superattracting part
is still open. In the purely superattracting case, Hubbard and Papadopol
\cite{HubbardPapadopol1994SuperattractiveFixedPointsinCn} ruled out
a result like Theorem~\ref{thm:Boettcher} that allows for general
topological conjugation to the lowest-order non-vanishing term. In
\cite{Favre2000Classificationof2DimensionalContractingRigidGermsandKatoSurfacesI}
and \cite{FavreJonsson2004TheValuativeTree,FavreJonsson2007Eigenvaluations},
Favre and Jonsson classified the so-called \emph{rigid} superattracting
germs in $\mathbb{C}^{2}$ and showed that the bigger class of \emph{dominant}
superattracting germs can be transformed into rigid germs via a sequence
of blow-ups. Ruggiero extended some of their results to non-invertible
germs in $\mathbb{C}^{2}$ in \cite{Ruggiero2012RigidificationofHolomorphicGermswithNoninvertibleDifferential},
and to superattracting germs in higher dimensions in \cite{Ruggiero2013ContractingRigidGermsinHigherDimensions}.
Ueda and Ushiki present Böttcher theorems in special cases of superattracting
germs in \cite{Ueda1992ComplexDynamicalSystemsonProjectiveSpaces}
and \cite{Ushiki1992BottchersTheoremandSuperStableManifoldsforMultidimensionalComplexDynamicalSystems}.
In \cite{BuffEpsteinKoch2012BottcherCoordinates}, Buff, Epstein,
and Koch establish necessary and sufficient conditions for a so-called
\emph{adapted} superattracting germ to admit a holomorphic conjugation
to its lowest order terms in each component. The definition of an
adapted germ implies the existence of a splitting of \noun{$\mathbb{C}^{d}$
}into a direct sum of locally invariant subspaces on which the lowest
order term of the action of the germ is a non-degenerate homogeneous
map. The condition for holomorphic normalisation encodes a ``compatibility''
of the postcritical loci of the germ and its normal form.

\section{Stable manifolds at non-hyperbolic fixed points}

Theorem~\ref{thm:StableMnfThm} is, in fact, a special case of a
more general theorem (see e.g.\ \cite{Abate2001AnIntroductiontoHyperbolicDynamicalSystems}),
that is a powerful tool even for non-hyperbolic germs:
\begin{thm}[Stable, unstable, and centre manifold]
\label{thm:StableCentreMnf}Let $F\in\End(\mathbb{C}^{d},0)$ and
$E^{\mathrm{s}}$, $E^{\mathrm{c}}$, and $E^{\mathrm{u}}$ the direct
sums of the generalised eigenspaces of $dF_{0}$ corresponding to
its attracting, neutral, and repelling eigenvalues, respectively.
Then for each $r\in\mathbb{N}$, there exists a neighbourhood $U\subseteq\mathbb{C}^{d}$
of $0$ and
\begin{enumerate}
\item a unique completely $F$-invariant complex submanifold
\[
W^{\mathrm{ss}}=\{z\in\Sigma_{F}(U)\mid\exists c\in(0,1)\text{ s.t. }z_{n}=O(c^{n})\}\subseteq U
\]
such that $0\in W^{\mathrm{ss}}\subseteq\Sigma_{F}(U)$ and $T_{0}W^{\mathrm{ss}}=E^{\mathrm{s}}$.
\item a real $\mathcal{C}^{r}$-smooth $F$-invariant submanifold $W^{\mathrm{cs}}\subseteq U$
such that $0\in\Sigma_{F}(U)\subseteq W^{\mathrm{cs}}$ and $T_{0}^{\mathbb{R}}W^{\mathrm{cs}}=E^{\mathrm{c}}\oplus E^{\mathrm{s}}$.
\item a real $\mathcal{C}^{r}$-smooth $F$-invariant submanifold $W^{\mathrm{c}}\subseteq U$
such that $0\in\Sigma_{F}(U)\cap\Sigma_{F}^{-}(U)\subseteq W^{\mathrm{c}}$
and $T_{0}^{\mathbb{R}}W^{\mathrm{c}}=E^{\mathrm{c}}$.
\item a real $\mathcal{C}^{r}$-smooth $F$-invariant submanifold $W^{\mathrm{cu}}\subseteq U$
such that $0\in\Sigma_{F}^{-}(U)\subseteq W^{\mathrm{cu}}$ and $T_{0}^{\mathbb{R}}W^{\mathrm{cu}}=E^{\mathrm{c}}\oplus E^{\mathrm{u}}$.
\item a unique completely $F$-invariant complex submanifold
\[
W^{\mathrm{su}}=\{z\in\Sigma_{F}^{-}(U)\mid\exists c\in(0,1),\{z_{-n}\}_{n}\subseteq U\text{ backward orbit of }z\text{ s.t. }z_{-n}=O(c^{n})\}\subseteq U
\]
such that $0\in W^{\mathrm{su}}\subseteq\Sigma_{F}^{-}(U)$ and $T_{0}W^{\mathrm{su}}=E^{\mathrm{u}}$.
\end{enumerate}
\end{thm}

\begin{rem}
\begin{enumerate}
\item $W^{\mathrm{ss}}$ and $W^{\mathrm{su}}$, $W^{\mathrm{c}}$, and
$W^{\mathrm{cs}}$ and $W^{\mathrm{cu}}$ are respectively called
\emph{strong stable} and \emph{strong unstable manifold}, \emph{centre
manifold}, and \emph{centre stable} and \emph{centre unstable manifold}.
\item The failure of the three centre manifolds to be complex manifolds
or even $\mathcal{C}^{\infty}$ is closely related to their non-uniqueness.
\item This is still not the most general form of the theorem. It can be
further refined to a filtration of the strong manifolds $W^{\mathrm{ss}}$
and $W^{\mathrm{su}}$ according to the absolute values of the corresponding
eigenvalues, by choosing suitable fixed constants $c\in(0,1)$ in
their definition (see \cite{HirschPughShub1977InvariantManifolds}
and \cite{Abate2001AnIntroductiontoHyperbolicDynamicalSystems}).
\end{enumerate}
\end{rem}

\section{\label{sec:BrjunoConditions}Brjuno conditions and holomorphic normal
forms}

\selectlanguage{british}

As in one variable, even formally linearisable germs may not admit
a holomorphic linearisation. The convergence of formal linearisation
depends on the denominators $\lambda^{\alpha}-\lambda_{j}$ for $\alpha\in\mathbb{N}^{d}$
and $j\le d$. The Brjuno condition (\ref{eq:1DBrjunoCond}) generalises
to several variables as follows:
\begin{defn}
For a tuple $\lambda=(\lambda_{1},\ldots,\lambda_{d})\in\mathbb{C}^{d}$
and an integer $m\ge2$, let 
\[
\omega(\lambda,m):=\min\{|\lambda^{\alpha}-\lambda_{j}|\mid2\le|\alpha|\le m,1\le j\le d\}.
\]
 We say the numbers $\lambda_{1},\ldots,\lambda_{d}$ satisfy the
\emph{Brjuno condition\index{Brjuno condition}}, if 
\begin{equation}
\sum_{i=0}^{\infty}\frac{1}{p_{i}}\log\frac{1}{\omega(\lambda,p_{i+1})}<\infty\label{eq:BrjunoCond}
\end{equation}
for a strictly increasing sequence of positive integers $\{p_{i}\}_{i\in\mathbb{N}}$.
\end{defn}

\begin{rem}
We can, without loss of generality, use $p_{k}=2^{k}$ in the above
definition (and its generalisations below), but unlike in dimension
one, there is no known number theoretic version of the Brjuno condition
in terms of continued fractions. The Brjuno condition implies in particular,
that there are no resonances.
\end{rem}

\begin{thm}[Brjuno \cite{Brjuno1971AnalyticFormofDifferentialEquationsIII}]
\label{thm:Brjuno}Let $F\in\End(\mathbb{C}^{d},0)$ with multipliers
$\lambda_{1},\ldots,\lambda_{d}$. If $(\lambda_{1},\ldots,\lambda_{d})$
satisfies the Brjuno condition (\ref{eq:BrjunoCond}), then $F$ is
holomorphically linearisable.
\end{thm}

In more than one variable, it is not known whether the Brjuno condition
(\ref{eq:BrjunoCond}) is strict, in that we do not know if for any
tuple $(\lambda_{1},\ldots,\lambda_{d})$ not satisfying the Brjuno
condition, there exists a non-linearisable germ with those multipliers.
We do, however still have a negative result under a stronger condition
analogous to Cremer's theorem~\ref{thm:Cremer}:
\begin{thm}[\cite{Brjuno1971AnalyticFormofDifferentialEquationsIII,Brjuno1973Analyticalformofdifferentialequations}]
Let $\lambda=(\lambda_{1},\ldots,\lambda_{d})\in\mathbb{C}^{d}$
be without resonances and such that 
\[
\limsup_{m\to+\infty}\frac{1}{m}\log\paren[{\bigg}]{\frac{1}{\omega(\lambda,m)}}=+\infty.
\]
Then there exists a germ $F\in\End(\mathbb{C}^{d},0)$ with linear
part $\diag(\lambda)$ that is not holomorphically linearisable.
\end{thm}

Before Brjuno's general result, Siegel \cite{Siegel1952UberDieNormalformAnalytischerDifferentialgleichungeninDerNaheEinerGleichgewichtslosung}
extended his Theorem~\ref{thm:Siegel} to several variables in the
setting of vector fields. Sternberg proved the analogue for holomorphic
maps, that shows in particular, that the Brjuno condition is almost
surely satisfied by random multipliers (see \cite{Arnold1988GeometricalmethodsinthetheoryofordinarydifferntialequationsTranslfromtheRussianbyJosephSzucsEnglTransledbyMarkLevi2nded}):
\begin{thm}[Sternberg \cite{Sternberg1961InfiniteLiegroupsandtheformalaspectsofdynamicalsystems}]
\label{thm:SiegelSternberg}Let $(\lambda_{1},\ldots,\lambda_{d})\in\mathbb{C}^{d}$.
If there exist constants $C>0$, $\gamma>1$, such that 
\begin{equation}
\omega(\lambda,m)\ge Cm^{-\gamma}\label{eq:SiegelSternberg}
\end{equation}
 for every $m\ge2$, then every germ $F\in\End(\mathbb{C},0)$ with
multipliers $\lambda_{1},\ldots,\lambda_{d}$ is holomorphically linearisable.
Furthermore, given any $C$ and $\gamma>(d-1)/2$, almost every tuple
of neutral multipliers $(\lambda_{1},\ldots,\lambda_{d})\in(S^{1})^{d}$
satisfies (\ref{eq:SiegelSternberg}).
\end{thm}

Clearly, a resonance (\ref{eq:resonance}) implies $\omega(\lambda,m)=0$
for all $m\ge|\alpha|$, so this is consistent with Theorem~\ref{thm:PoincareDulac}.
In \cite{Poeschel1986Oninvariantmanifoldsofcomplexanalyticmappingsnearfixedpoints},
Pöschel generalised Brjuno's condition and linearisation result to
cases with resonances in the following way:
\begin{defn}
\label{def:admissible}A subsequence $S=\{\lambda_{i_{1}},\ldots,\lambda_{i_{k}}\}\subseteq\{\lambda_{1},\ldots,\lambda_{d}\}\subseteq\mathbb{C}$,
$k\le d$ satisfies the \emph{\index{partial Brjuno condition@\emph{partial Brjuno condition}}partial
Brjuno condition} with respect to $(\lambda_{1},\ldots,\lambda_{d})$,
if 
\begin{equation}
-\sum_{\nu\ge1}2^{-\nu}\log\omega_{S}(2^{\nu})<\infty,\label{eq:PartialBrjunoPoeschel}
\end{equation}
where 
\[
\omega_{S}(m):=\min\{\abs{\mu_{1}\cdots\mu_{r}-\lambda_{i}}\mid\mu_{1},\ldots,\mu_{r}\in S,2\le r\le m,1\le i\le d\}
\]
for $m\in\mathbb{N}$, $m\ge2$.
\end{defn}

\begin{thm}[Pöschel \cite{Poeschel1986Oninvariantmanifoldsofcomplexanalyticmappingsnearfixedpoints}]
\label{thm:PoeschelBrjuno}Let $F\in\End(\mathbb{C}^{d},0)$ with
multipliers $\lambda_{1},\ldots,\lambda_{d}$ such that $S\subseteq\{\lambda_{1},\ldots,\lambda_{d}\}$
satisfies the partial Brjuno condition (\ref{eq:PartialBrjunoPoeschel}).
Then there exists a complex $F$-invariant manifold through $0$,
of dimension $|S|$, tangent to the sum of the eigenspaces of $dF_{0}$
corresponding to the eigenvalues in $S$, on which $F$ is holomorphically
linearisable.
\end{thm}

Other generalisations of Brjuno's condition allowing for resonances
were made by Rüssmann in 1977, published in \cite{Ruessmann2002StabilityofEllipticFixedPointsofAnalyticAreaPreservingMappingsundertheBrunoCondition},
and Raissy in \cite{Raissy2011Brjunoconditionsforlinearizationinpresenceofresonances},
where she defined the reduced Brjuno condition, and showed it is implied
by Rüssmann's condition. In \cite[Thm.~4.1]{Raissy2013Holomorphiclinearizationofcommutinggermsofholomorphicmaps},
Raissy showed both conditions to be indeed equivalent.
\begin{defn}
We say that the numbers $\lambda_{1},\ldots,\lambda_{d}\in\mathbb{C}$
satisfy the \emph{reduced Brjuno condition\index{reduced Brjuno condition}},
if 
\begin{equation}
-\sum_{\nu\ge1}2^{-\nu}\log\tilde{\omega}(\lambda,2^{\nu})<\infty,\label{eq:ReducedBjurno}
\end{equation}
where 
\[
\tilde{\omega}(\lambda,m):=\{|\lambda^{\alpha}-\lambda_{j}|\mid2\le|\alpha|\le m,1\le j\le d,\lambda^{\alpha}\neq\lambda_{j}\}.
\]
\end{defn}

\begin{thm}[Raissy \cite{Raissy2011Brjunoconditionsforlinearizationinpresenceofresonances},
Rüssmann \cite{Ruessmann2002StabilityofEllipticFixedPointsofAnalyticAreaPreservingMappingsundertheBrunoCondition}]
\label{thm:ReducedBrjunoLin}Let $F\in\End(\mathbb{C}^{d},0)$ with
diagonalisable linear part $dF_{0}$, and whose multipliers satisfy
the reduced Brjuno condition (\ref{eq:ReducedBjurno}). If $F$ is
formally linearisable, then $F$ is holomorphically linearisable.
\end{thm}

An additional obstruction to holomorphic linearisation in more than
one variable is given by Jordan blocks:
\begin{thm}[Yoccoz \cite{Yoccoz1995TheoremeDeSiegelNombresDeBrunoEtPolynomesQuadratiques}]
\label{thm:JordanBlockNoHolLin}Let $A\in\mathbb{C}^{d\times d}$
be an invertible matrix in Jordan normal form with a non-trivial Jordan
block associated to an eigenvalue of modulus $1$. Then there exists
a germ $F\in\End(\mathbb{C}^{d},0)$ with $dF_{0}=A$ that is not
holomorphically linearisable.
\end{thm}

\begin{note}
If there are no resonances except the multiple eigenvalue of the Jordan
block, then such a germ is formally linearisable, but not holomorphically
linearisable.
\end{note}

Écalle and Vallet constructed, without restrictions, holomorphic normal
forms ``nearby'' the original germ in \cite{EcalleVallet1998CorrectionandLinearizationofResonantVectorFieldsandDiffeomorphisms}.
A deformation approach to holomorphic normalisation by Pérez-Marco
was published in \cite{PerezMarco2001TotalConvergenceorGeneralDivergenceinSmallDivisors}.

\subsection{\label{subsec:Brjuno-sets}Brjuno sets}

In this section, we will prove a generalisation of Theorem~\ref{thm:PoeschelBrjuno},
aimed at the holomorphic elimination of infinite families of monomials
even in the absence of formal linearisability. 

We first introduce the notion of a Brjuno set of exponents:
\begin{defn}
\label{def:BrjunoSet}Let $F$ be a germ of endomorphisms of $\mathbb{C}^{d}$
with 
\[
dF_{0}=\Lambda=\diag(\lambda_{1},\ldots,\lambda_{d}).
\]
A set $A\subseteq\mathbb{N}^{d}$ is a \emph{Brjuno set} (\emph{of
exponents}) for $(\lambda_{1},\ldots,\lambda_{d})$ (or for $F$),
if 
\begin{equation}
\sum_{k\ge1}2^{-k}\log\omega_{A}^{-1}(2^{k})<\infty,\label{eq:BrjunoSet}
\end{equation}
where 
\begin{equation}
\omega_{A}(k):=\min\{|\lambda^{\alpha}-\lambda_{i}|\mid\alpha\in A,2\le|\alpha|\le k,1\le i\le d\}\cup\{1\}\label{eq:BrjunoFunction}
\end{equation}
for $k\ge2$.
\end{defn}

\begin{rem}
Subsets and finite unions of Brjuno sets are Brjuno sets.
\end{rem}

This definition includes the classical Brjuno condition (\ref{eq:BrjunoCond})
and Pöschel's partial Brjuno condition (\ref{eq:PartialBrjunoPoeschel})
as follows:
\begin{lem}
\label{lem:BrjunoAndPart}Let $\lambda_{1},\ldots,\lambda_{d}\in\mathbb{C}$.
\begin{enumerate}
\item $\{\lambda_{1},\ldots,\lambda_{d}\}$ satisfies the Brjuno condition
(\ref{eq:BrjunoCond}), if $A=\mathbb{N}^{d}$ is a Brjuno set for
$(\lambda_{1},\ldots,\lambda_{d})$.
\item $S\subseteq\{\lambda_{1},\ldots,\lambda_{d}\}$ satisfies the partial
Brjuno condition (\ref{eq:PartialBrjunoPoeschel}) (with respect to
$(\lambda_{1},\ldots,\lambda_{d})$), if $A=\{\alpha\in\mathbb{N}^{d}\mid\alpha_{j}=0\text{ for }\lambda_{j}\notin S\}$
is a Brjuno set for $(\lambda_{1},\ldots,\lambda_{d})$.
\end{enumerate}
\end{lem}

The following theorem generalises Theorems~\ref{thm:Brjuno} and
\ref{thm:PoeschelBrjuno} in the context of eliminating infinite families
of monomials.
\begin{thm}
\label{thm:GeneralElimination}Let $F$ be a germ of endomorphisms
of $\mathbb{C}^{d}$ of the form  $F(z)=\Lambda z+\sum_{|\alpha|>1}\sum_{j=1}^{d}f_{\alpha}^{j}z^{\alpha}e_{j}$
with $\Lambda=\diag(\lambda_{1},\ldots,\lambda_{d})$. Let $A_{0}$
and $A$ be disjoint sets of multi-indices in $\mathbb{N}^{d}$ such
that 
\begin{enumerate}
\item \label{enu:Elim1}If $\alpha\in A_{0}$ and $\beta\le\alpha$, then
$\beta\in A_{0}$, and if $\alpha\in A$ and $\beta\le\alpha$, then
$\beta\in A_{0}\cup A$.
\item \label{enu:Elim2}If $\beta_{1},\ldots,\beta_{l}\in A_{0}$ and $\beta_{1}+\cdots+\beta_{l}\in A_{0}\cup A$,
$|\beta_{1}|\ge2$ and $f_{\beta_{1}}^{j_{1}}\cdots f_{\beta_{l}}^{j_{l}}\neq0$,
then $e_{j_{1}}+\cdots+e_{j_{l}}\notin A$.
\item \label{enu:Elim3Brjuno}$A$ is a Brjuno set for $F$.
\end{enumerate}
Then there exists a local biholomorphism $H\in\Aut(\mathbb{C}^{d},0)$
conjugating $F$ to $G=H^{-1}\circ F\circ H$ such that $G(z)=\Lambda z+\sum_{|\alpha|>1}g_{\alpha}z^{\alpha}$
with $g_{\alpha}=f_{\alpha}$ for $\alpha\in A_{0}$, $|\alpha|\ge2$
and $g_{\alpha}=0$ for $\alpha\in A$, $|\alpha|\ge2$.
\end{thm}

For $A_{0}=\{|\alpha|\le1\}$ and $A=\mathbb{N}^{d}\backslash A_{0}$
or $A=(\mathbb{N}^{m}\times\{0\})\backslash A_{0}$, we recover the
results from \cite{Brjuno1973Analyticalformofdifferentialequations}
and \cite{Poeschel1986Oninvariantmanifoldsofcomplexanalyticmappingsnearfixedpoints}.
A novelty of phrasing the result in this way is, that it can be iterated
to obtain the following:
\begin{thm}
\label{thm:IteratedEliminationSimple}Let $F$ be a germ of endomorphisms
of $\mathbb{C}^{d}$ of the form $F(z)=\Lambda z+\sum_{|\alpha|>1}\sum_{j=1}^{d}f_{\alpha}^{j}z^{\alpha}e_{j}$
with $\Lambda=\diag(\lambda_{1},\ldots,\lambda_{d})$. Let $A_{0}$
and $A$ be disjoint sets of multi-indices in $\mathbb{N}^{d}$ such
that $A$ admits a partition $A=A_{1}\cup\cdots\cup A_{k_{0}}$ such
that
\begin{enumerate}
\item \label{enu:itElim1}For $0\le k\le k_{0}$, if $\alpha\in A_{k}$
and $\beta\le\alpha$, then $\beta\in A_{\overline{k}}:=A_{0}\cup\cdots\cup A_{k}$
(where $\le$ is taken component-wise).
\item \label{enu:itElim2}For $1\le k\le k_{0}$, if $\beta_{1},\ldots,\beta_{l}\in A_{0}$
such that $\beta_{1}+\cdots+\beta_{l}\in A_{\overline{k}}$, $|\beta_{1}|\ge2$,
and $f_{\beta_{1}}^{j_{1}}\cdots f_{\beta_{l}}^{j_{l}}\neq0$, then
$e_{j_{1}}+\cdots+e_{j_{l}}\notin A_{k}$.
\item \label{enu:itElim3}$A$ is a Brjuno set for $F$.
\end{enumerate}
Then there exists a local biholomorphism $H\in\Aut(\mathbb{C}^{d},0)$
conjugating $F$ to $G=H^{-1}\circ F\circ H$ where $G(z)=\sum_{|\alpha|\ge1}g_{\alpha}z^{\alpha}$
with $g_{\alpha}=f_{\alpha}$ for $\alpha\in A_{0}$ and $g_{\alpha}=0$
for $\alpha\in A$.
\end{thm}

\begin{rem}
If we assume $\sum_{\alpha\in A_{0}}f_{\alpha}z^{\alpha}$ to be in
Poincaré-Dulac normal form, the condition $f_{\beta_{1}}^{j_{1}}\cdots f_{\beta_{l}}^{j_{l}}\neq0$
can be replaced by $\lambda^{\beta_{m}}\neq\lambda_{j_{m}}$ for $1\le m\le l$
to avoid dependence of Condition~(\ref{enu:Elim2}) on the specific
germ $F$.
\end{rem}

The proof of Theorem~\ref{thm:GeneralElimination} emerges largely
by careful examination of that in \cite{Poeschel1986Oninvariantmanifoldsofcomplexanalyticmappingsnearfixedpoints}
with some adjustments to avoid the assumption $\min_{1\le j\le d}|\lambda_{j}|\le1$
in the proofs of Lemmas~\ref{lem:Siegel} and \ref{lem:Brjuno}.
In \cite{Poeschel1986Oninvariantmanifoldsofcomplexanalyticmappingsnearfixedpoints}
this is ensured by considering $F^{-1}$ if necessary, but Condition~(\ref{enu:Elim2})
in our theorem is not invariant under taking inverses.
\begin{proof}[Proof of Theorem~\ref{thm:GeneralElimination}]
Formal series 
\begin{align*}
G(z) & =\Lambda z+g(z)=\sum_{|\alpha|\ge1}g_{\alpha}z^{\alpha}\quad\text{and}\quad H(z)=z+h(z)=\sum_{|\alpha|\ge1}h_{\alpha}z^{\alpha}
\end{align*}
of the required form emerge as solutions to the homological equation
$F\circ H=H\circ G$. Comparing coefficients for $\alpha\in\mathbb{N}^{d}\backslash\{0\}$,
this means
\begin{equation}
(\lambda^{\alpha}\id-\Lambda)h_{\alpha}=f_{\alpha}-g_{\alpha}+\sum_{2\le k<|\alpha|}\sum_{j_{1}\le\cdots\le j_{k}}\sum_{\beta_{1}+\cdots+\beta_{k}=\alpha}(f_{e_{J}}h_{\beta_{1}}^{j_{1}}\cdots h_{\beta_{k}}^{j_{k}}-h_{e_{J}}g_{\beta_{1}}^{j_{1}}\cdots g_{\beta_{k}}^{j_{k}}),\label{eq:HomolExplicit}
\end{equation}
where $e_{J}:=e_{j_{1}}+\cdots+e_{j_{k}}$. Take $h_{\alpha}=0$ for
$\alpha\notin A$. Then for $\alpha\in A_{0}$, the first term in
the sum vanishes by Condition~(\ref{enu:Elim1}) and the second term
vanishes by Condition~(\ref{enu:Elim2}), so $g_{\alpha}=f_{\alpha}$.
For $\alpha\in A$, $\lambda^{\alpha}\id-\Lambda$ is invertible by
Condition~(\ref{enu:itElim3}) and the right hand side depends only
on $h$-terms with index of order less than $|\alpha|$. Hence (\ref{eq:HomolExplicit})
determines $h_{\alpha}$ uniquely by recursion and we obtain a formal
solution $H$ and hence $G=H^{-1}\circ F\circ H$.

To show that $H$ (and hence $G$) converges in some neighbourhood
of the origin, we have to show
\begin{equation}
\sup_{\alpha\in\mathbb{N}^{d}}\frac{1}{|\alpha|}\log\norm{h_{\alpha}}_{1}<\infty.\label{eq:convOfConjugation}
\end{equation}
We apply the majorant method first used by C.~L.~Siegel in \cite{Siegel1942Iterationofanalyticfunctions}
and improved in \cite{Brjuno1973Analyticalformofdifferentialequations}.
We may assume (up to scaling of variables) that $\norm{f_{\alpha}}_{1}\le1$
for all $|\alpha|\ge2$. Now for $\alpha\in A$ again by Condition~(\ref{enu:Elim2}),
the second term in the sum in (\ref{eq:HomolExplicit}) vanishes and
it follows
\begin{equation}
\norm{h_{\alpha}}_{1}\le d\cdot\norm{h_{\alpha}}_{\infty}\le d\cdot\varepsilon_{\alpha}^{-1}\sum_{2\le k\le|\alpha|}\sum_{\beta_{1}+\cdots+\beta_{k}=\alpha}\norm{h_{\beta_{1}}}_{1}\cdots\norm{h_{\beta_{k}}}_{1},\label{eq:HomolTermEstimate}
\end{equation}
where $\varepsilon_{\alpha}:=\min_{1\le i\le d}|\lambda^{\alpha}-\lambda_{i}|$.
We estimate (\ref{eq:HomolTermEstimate}) in two parts, one on the
number of summands, the other on their size. We define recursively
$\sigma_{1}=1$ and 
\begin{equation}
\sigma_{r}:=d\sum_{k=2}^{r}\sum_{r_{1}+\cdots+r_{k}=r}\sigma_{r_{1}}\cdots\sigma_{r_{k}}\quad\text{for }r\ge2,\label{eq:SigmaForTermNumber}
\end{equation}
and $\delta_{e_{1}}=\cdots=\delta_{e_{d}}=1$, $\delta_{\alpha}=0$
for $\alpha\notin A\cup\{e_{1},\ldots,e_{d}\}$, and 
\begin{align}
\delta_{\alpha} & \coloneqq\varepsilon_{\alpha}^{-1}\max_{\substack{\beta_{1}+\cdots+\beta_{k}=\alpha\\
k\ge2
}
}\delta_{\beta_{1}}\cdots\delta_{\beta_{k}}\quad\text{for }\alpha\in A.\label{eq:deltaForDivisors}
\end{align}
Then, by induction on $|\alpha|$, (\ref{eq:HomolTermEstimate}) implies
\begin{equation}
\norm{h_{\alpha}}_{1}\le\sigma_{|\alpha|}\delta_{\alpha}\label{eq:BoundHbySigmaDelta}
\end{equation}
for $\alpha\in A$. Hence to prove (\ref{eq:convOfConjugation}) it
suffices to prove estimates of the same type for $\sigma_{|\alpha|}$
and $\delta_{\alpha}$. 

The estimates on $\sigma_{r}$ go back to \cite{Siegel1942Iterationofanalyticfunctions}
and \cite{Sternberg1961InfiniteLiegroupsandtheformalaspectsofdynamicalsystems}:
Let $\sigma(t):=\sum_{r=1}^{\infty}\sigma_{r}t^{r}$ and observe 
\begin{align*}
\sigma(t) & =t+\sum_{r=2}^{\infty}\paren[\Big]{d\sum_{k=2}^{r}\sum_{r_{1}+\cdots+r_{k}=r}\sigma_{r_{1}}\cdots\sigma_{r_{k}}}t^{r}\\
 & =t+d\sum_{k=2}^{\infty}\paren[\Big]{\sum_{r=1}^{\infty}\sigma_{r}t^{r}}^{k}\\
 & =t+d\frac{\sigma(t)^{2}}{1-\sigma(t)}.
\end{align*}
Solving for $t$ and requiring $\sigma(0)=0$ yields a unique holomorphic
solution 
\[
\sigma(t)=\frac{1+t-\sqrt{(1+t)^{2}-4(d+1)t}}{2(d+1)}
\]
for small $t$, so $\sigma$ converges near $0$ and we have 
\begin{equation}
\sup_{r\ge1}\frac{1}{r}\log\sigma_{r}<\infty.\label{eq:BoundSigmaR}
\end{equation}

The estimates on $\delta_{\alpha}$ take care of the small divisors
$\varepsilon_{\alpha}$ and proceed essentially like \cite{Brjuno1973Analyticalformofdifferentialequations}.
For every $|\alpha|\ge2$, we choose a maximising decomposition $\beta_{1}+\cdots+\beta_{k}=\alpha$
in (\ref{eq:deltaForDivisors}) such that 
\begin{equation}
\delta_{\alpha}=\varepsilon_{\alpha}^{-1}\delta_{\beta_{1}}\cdots\delta_{\beta_{k}}\label{eq:ChosenDecompDelta}
\end{equation}
and $|\alpha|>|\beta_{1}|\ge\cdots\ge|\beta_{k}|\ge1$. In this way,
starting from $\delta_{\alpha}$ we proceed to decompose $\delta_{\beta_{1}},\ldots,\delta_{\beta_{k}}$
in the same way and continue the process until we arrive at a well-defined
decomposition of the form 
\begin{equation}
\delta_{\alpha}=\varepsilon_{\alpha_{0}}^{-1}\varepsilon_{\alpha_{1}}^{-1}\cdots\varepsilon_{\alpha_{s}}^{-1},\label{eq:DeltaAsAllEps}
\end{equation}
where $2\le|\alpha_{s}|,\ldots,|\alpha_{1}|<|\alpha_{0}|$ and $\alpha_{0}=\alpha$.
We may further choose an index $i_{\alpha}\in\{1,\ldots,d\}$ for
each $|\alpha|\ge2$ such that 
\[
\varepsilon_{\alpha}=|\lambda^{\alpha}-\lambda_{i_{\alpha}}|.
\]
Let $n\in\mathbb{N}$ such that 
\[
n-1\ge2\min_{1\le i\le d}|\lambda_{i}|
\]
and $\theta>0$ such that 
\begin{equation}
n\theta=\min_{1\le i\le d}\{|\lambda_{i}|,1\}\le1.\label{eq:ThetaDef}
\end{equation}
Now for the indices $\alpha_{0},\ldots,\alpha_{s}$ in the decomposition
(\ref{eq:DeltaAsAllEps}), we want to bound
\[
N_{m}^{j}(\alpha):=\#\{l\in\{0,\ldots,s\}\mid i_{\alpha_{l}}=j,\varepsilon_{\alpha_{l}}<\theta\omega_{A}(m)\},
\]
for $j\le d$ and $m\in\mathbb{N}$, where we adopt the convention
$\omega_{A}(1):=+\infty$. First we need the following lemma, that
\cite{Poeschel1986Oninvariantmanifoldsofcomplexanalyticmappingsnearfixedpoints}
attributes to Siegel, showing that indices contributing to $N_{m}^{j}(\alpha)$
cannot be to close to each other:
\begin{lem}[Siegel]
\label{lem:Siegel}Let $m\ge1$. If $\alpha>\beta$ are such that
$\varepsilon_{\alpha}<\theta\omega_{A}(m),$ $\varepsilon_{\beta}<\theta\omega_{A}(m),$
and $i_{\alpha}=i_{\beta}=j,$ then $|\alpha-\beta|\ge m$.
\end{lem}

\begin{subproof}For $m=1$, $\alpha>\beta$ implies $|\alpha-\beta|\ge1$.
For $m\ge2$, by the definition of $\omega_{A}$ in (\ref{eq:BrjunoFunction}),
we have $\omega_{A}(m)\le1$. With that and (\ref{eq:ThetaDef}),
the hypothesis $\varepsilon_{\beta}<\theta\omega_{A}(m)$ implies
\[
|\lambda^{\beta}|>|\lambda_{j}|-\theta\omega_{A}(m)\ge n\theta-\theta=(n-1)\theta
\]
and hence 
\begin{align*}
2\theta\omega(m) & >\varepsilon_{\alpha}+\varepsilon_{\beta}\\
 & =|\lambda^{\alpha}-\lambda_{j}|+|\lambda^{\beta}-\lambda_{j}|\\
 & \ge|\lambda^{\alpha}-\lambda^{\beta}|\\
 & =|\lambda^{\beta}||\lambda^{\alpha-\beta}-1|\\
 & >(n-1)\theta\cdot(\min_{1\le i\le d}|\lambda_{i}|)^{-1}\omega(|\alpha-\beta|+1)\\
 & \ge2\theta\omega(|\alpha-\beta|+1),
\end{align*}
i.e.\ $\omega(m)>\omega(|\alpha-\beta|+1)$. But $\omega$ is decreasing,
so we must have $|\alpha-\beta|\ge m$.\end{subproof}

We can now show Brjuno's estimate on $N_{m}^{j}(\alpha)$:
\begin{lem}[Brjuno, \cite{Brjuno1973Analyticalformofdifferentialequations}]
\label{lem:Brjuno}For $|\alpha|\ge2$, $m\ge1$, and $1\le j\le d$,
we have 
\[
N_{m}^{j}(\alpha)\le\begin{cases}
0, & \text{for }|\alpha|\le m\\
2|\alpha|/m-1, & \text{for }|\alpha|>m.
\end{cases}
\]
\end{lem}

\begin{subproof}We fix $m$ and $j$ and proceed by induction on
$|\alpha|$.

If $2\le|\alpha|\le m$, we have 
\[
\varepsilon_{\alpha_{l}}\ge\omega_{A}(|\alpha|)\ge\omega_{A}(m)\ge\theta\omega_{A}(m)
\]
for all $0\le l\le s$, so $N_{m}^{j}(\alpha)=0.$ 

If $|\alpha|>m$, we take the chosen decomposition (\ref{eq:ChosenDecompDelta})
and note that only $|\beta_{1}|$ may be greater than $K=\max\{|\alpha|-m,m\}$.
If $|\beta_{1}|>K$, we decompose $\delta_{\beta_{1}}$ in the same
way and repeat this at most $m-1$ times to obtain a decomposition
\begin{equation}
\delta_{\alpha}=\varepsilon_{\alpha}^{-1}\varepsilon_{\alpha_{1}}^{-1}\cdots\varepsilon_{\alpha_{k}}^{-1}\cdot\delta_{\beta_{1}}\cdots\delta_{\beta_{l}}\label{eq:DeltaDecompForInd}
\end{equation}
with $0\le k\le m-1$, $l\ge2$ and 
\begin{align}
\alpha & >\alpha_{1}>\cdots>\alpha_{k}\nonumber \\
\alpha & =\beta_{1}+\cdots+\beta_{l}\nonumber \\
|\alpha_{k}| & >K\ge|\beta_{1}|\ge\cdots\ge|\beta_{l}|.\label{eq:deltaDecompIndProp}
\end{align}
In particular, (\ref{eq:deltaDecompIndProp}) implies $|\alpha-\alpha_{k}|<m$.
Hence Lemma~\ref{lem:Siegel} shows that at most one of the $\varepsilon$-
factors in (\ref{eq:DeltaDecompForInd}) can contribute to $N_{m}^{j}(\alpha)$
and we have 
\[
N_{m}^{j}(\alpha)\le1+N_{m}^{j}(\beta_{1})+\cdots+N_{m}^{j}(\beta_{l}).
\]
Now let $0\le h\le l$ such that be the such that $|\beta_{1}|,\ldots,|\beta_{h}|>m\ge|\beta_{h+1}|,\ldots,|\beta_{l}|$.
Then by (\ref{eq:deltaDecompIndProp}), we have $|\beta_{1}|,\ldots,|\beta_{h}|\le|\alpha|-m$
and, by induction, the terms with $|\beta|\le m$ vanish and we have
\begin{align*}
N_{m}^{j}(\alpha) & \le1+N_{m}^{j}(\beta_{1})+\cdots+N_{m}^{j}(\beta_{h})\\
 & \le1+2|\beta_{1}+\cdots+\beta_{h}|/m-h\\
 & \le\begin{cases}
1, & \text{for }h=0\\
2\frac{|\alpha|-m}{m}, & \text{for }h=1\\
2|\alpha|/m-(h-1), & \text{for }h\ge2
\end{cases}\\
 & \le2|\alpha|/m-1.\qedhere
\end{align*}
\end{subproof}

To estimate the product (\ref{eq:DeltaAsAllEps}) we partition the
indices into sets 
\[
I_{l}:=\{0\le k\le s\mid\theta\omega_{A}(2^{l+1})\le\varepsilon_{\alpha_{k}}<\theta\omega_{A}(2^{l})\}\quad\text{for }l\ge0
\]
(recall for $I_{0}$ the convention $\omega_{A}(1)=+\infty$). By
Lemma~\ref{lem:Brjuno}, we have 
\[
\#I_{l}\le N_{2^{l}}^{1}(\alpha)+\cdots+N_{2^{l}}^{d}(\alpha)\le2d|\alpha|2^{-l}
\]
and we can estimate
\begin{align*}
\frac{1}{|\alpha|}\log\delta_{\alpha} & =\sum_{k=0}^{s}\frac{1}{|\alpha|}\log\varepsilon_{\alpha_{k}}^{-1}\\
 & \le\sum_{l\ge0}\sum_{k\in I_{l}}\frac{1}{|\alpha|}\log(\theta^{-1}\omega_{A}^{-1}(2^{l+1}))\\
 & \le2d\sum_{l\ge0}2^{-l}\log(\theta^{-1}\omega_{A}^{-1}(2^{l+1}))\\
 & =4d\log(\theta^{-1})+4d\sum_{l\ge1}2^{-l}\log(\omega_{A}^{-1}(2^{l})).
\end{align*}
This bound is independent of $\alpha$ and, since $A$ is a Brjuno
set, it is finite. Hence with (\ref{eq:BoundHbySigmaDelta}) and (\ref{eq:BoundSigmaR})
it follows that 
\[
\sup_{\alpha\in\mathbb{N}^{d}}\frac{1}{|\alpha|}\log\norm{h_{\alpha}}\le\sup_{\alpha\in\mathbb{N}^{d}}\frac{1}{|\alpha|}\log\delta_{\alpha}+\sup_{r\ge1}\frac{1}{r}\log\sigma_{r}<\infty
\]
and thus $H$ and $G$ converge.
\end{proof}
\begin{proof}[Proof of Theorem~\ref{thm:IteratedEliminationSimple}]
Assume by induction on $k_{0}$, that $f_{\alpha}=0$ for $\alpha\in A\backslash A_{k_{0}}$.
We show that $A_{0}'=A_{\overline{k_{0}-1}}$ and $A'=A_{k_{0}}$
satisfy the prerequisites of Theorem~\ref{thm:GeneralElimination}.

Conditions~(\ref{enu:Elim1}) and (\ref{enu:Elim3Brjuno}) follow
directly from their counterparts. 

Let $\beta_{1},\ldots,\beta_{l}\in A_{0}'=A_{\overline{k_{0}-1}}$
as in Condition~(\ref{enu:Elim2}). By induction $f_{\beta_{1}}^{j_{1}}\cdots f_{\beta_{l}}^{j_{l}}\neq0$
implies $\beta_{1},\ldots,\beta_{l}\in A_{0}$, so Assumption~(\ref{enu:itElim2})
of Theorem~\ref{thm:IteratedEliminationSimple} implies $e_{J}\notin A_{k}=A'$,
and Condition~(\ref{enu:Elim2}) is satisfied.

Therefore Theorem~\ref{thm:GeneralElimination} shows that $F$ is
conjugate to $G$ with $g_{\alpha}=f_{\alpha}$ for $\alpha\in A_{0}$
and $g_{\alpha}=0$ for $\alpha\in A$.
\end{proof}

\section{Blow-ups}

\selectlanguage{british}
\global\long\def\Transp{\mathsf{T}}%
\global\long\def\Pow{\m{Pow}}%
\global\long\def\ord{\m{ord}}%

A geometric technique in several complex variables to deal with the
presence of a continuum of complex directions is blowing up a point
or submanifold. A blow-up of a manifold along a submanifold replaces
each point of the submanifold with the space of all complex directions
approaching from outside the submanifold (normal directions). Blow-ups
were developed an important tool in algebraic geometry for the resolution
of singularities of algebraic curves. For our purposes it will allow
us to separate the local dynamics by complex directions from the origin.

We give a short introduction of blow-ups following \cite[\textsection 2.5]{Huybrechts2005ComplexgeometryAnintroduction}.
For our dynamical studies it is important to ``lift'' our germs
to the blown-up manifold, so we give a sufficient condition that ensures
the existence of such a lift, making blow-ups a type of generalised
non-invertible change of coordinates. Section~\ref{subsec:Diagonalisation}
then presents Abate's diagonalisation of non-diagonalisable (linear
parts of) germs via blow-ups \cite{Abate2000Diagonalizationofnondiagonalizablediscreteholomorphicdynamicalsystems}.

In the following sections, we apply blow-ups to ``zoom in'' to a
characteristic direction of a germ to examine the dynamics along that
direction. In particular this can eliminate some resonant terms, that
could not be removed by changing coordinates. 

\subsection{Blow-ups of $\mathbb{C}^{d}$}

\subsubsection{Blow-up of a point}

When blowing up a point, it is replaced with the space of all possible
directions going through it, that is, with a complex projective space
of dimension $d-1$. 

To make the precise definition, we fix some notation for projective
spaces.
\begin{notation}
For a complex vector space $V$, let $\mathbb{P}(V)$ denote the
projective space of one-dimensional subspaces of $V$. If $V=\mathbb{C}^{d}$,
we write $\mathbb{P}^{d-1}:=\mathbb{P}(\mathbb{C}^{d})$. For $v\in V\backslash\{0\}$
let further $[v]\in\mathbb{P}(V)$ denote the linear span of $v$.
\end{notation}

\begin{defn}
The \emph{blow-up\index{blow-up!at a point}} of $\mathbb{C}^{d}$
at the origin $0\in\mathbb{C}^{d}$ is 
\[
\tilde{\mathbb{C}}_{0}^{d}:=\{(z,[v])\in\mathbb{C}^{d}\times\mathbb{P}^{d-1}\mid z\in[v]\}\subseteq\mathbb{C}^{d}\times\mathbb{P}^{d-1}
\]
 with canonical projections $\sigma:\tilde{\mathbb{C}}_{0}^{d}\to\mathbb{C}^{d}$
and $\pi:\tilde{\mathbb{C}}_{0}^{d}\to\mathbb{P}^{d-1}$ to the first
and second component respectively. The hypersurface 
\[
\sigma^{-1}(\{0\})=\{0\}\times\mathbb{P}^{d-1}
\]
is called the \emph{\index{exceptional divisor}exceptional divisor}
of the blow-up.
\end{defn}

These two projections each have their merits in understanding blow-ups:
On one hand, we obtain a canonical complex manifold structure on $\tilde{\mathbb{C}}_{0}^{d}$
by viewing it as the tautological line bundle $\pi:\tilde{\mathbb{C}}_{0}^{d}\to\mathbb{P}^{d-1}$
over $\mathbb{P}^{d-1}$. On the other hand, the projection $\sigma$
represents the idea of the modification: away from $0\in\mathbb{C}^{d}$
and $\sigma^{-1}(0)\in\tilde{\mathbb{C}}_{0}^{d}$, $\sigma$ is just
a biholomorphism, but at $z=0$ it has a non-trivial fibre $\sigma^{-1}(0)=\{(0,[v])\mid v\in\mathbb{C}^{d}\backslash\{0\}\}\cong\mathbb{P}^{d-1}$,
so we can think of $\tilde{\mathbb{C}}_{0}^{d}$ as $\mathbb{C}^{d}\backslash\{0\}\cup\mathbb{P}^{d-1}$.

\subsubsection{Blow-up along a subspace}

Similarly, when blowing up along a subspace, we want to replace it
with the space of all normal directions at each point.
\begin{defn}
A subset $I\subseteq\{1,\ldots,d\}$ (together with $\overline{I}:=\{1,\ldots,d\}\backslash I$)
is called a \emph{splitting} \index{splitting of a positive integer}of
$d$ of \emph{weight\index{weight of a splitting}} $m=|I|$. The
\emph{standard splitting} of weight $m$ is $I=\{1,\ldots,m\}$.
\end{defn}

\begin{notation}
\label{nota:PrimeNotation}For $z=(z_{1},\ldots,z_{d})\in\mathbb{C}^{d}$
and a splitting $I$ of $d$ with $I=\{i_{1},\ldots,i_{m}\}$ and
$\overline{I}=\{i_{m+1},\ldots,i_{d}\}$, where $i_{1}<\cdots<i_{m}$
and $i_{m+1}<\cdots<i_{d}$, we write $z'=(z_{i_{1}},\ldots,z_{i_{m}})$
and $z''=(z_{i_{m+1}},\ldots,z_{d})$ as well as $F'=(F_{i_{1}},\ldots,F_{i_{m}})$
etc.
\end{notation}

\begin{defn}
The \emph{blow-up} of $\mathbb{C}^{d}$ \emph{along the subspace}\index{blow-up!along a subspace}
\[
X=X_{I}:=\{z=(z_{1},\ldots,z_{d})\in\mathbb{C}^{d}\mid z_{j}=0\text{ for all }j\in\overline{I}\}=\{z\in\mathbb{C}^{d}\mid z''=0\}
\]
 for a splitting $I\subseteq\{1,\ldots,d\}$ of $d$ is defined as
\[
\tilde{\mathbb{C}}_{X}^{d}:=\{(z,[v])\in\mathbb{C}^{d}\times\mathbb{P}(\mathbb{C}^{d}/X)\mid z\in X+[v]\}\subseteq\mathbb{C}^{d}\times\mathbb{P}(\mathbb{C}^{d}/X)\cong\mathbb{C}^{d}\times\mathbb{P}^{d-m-1},
\]
where $m=|I|$, with canonical projections $\sigma:\tilde{\mathbb{C}}_{X}^{d}\to\mathbb{C}^{d}$
and $\pi:\tilde{\mathbb{C}}_{X}^{d}\to\mathbb{P}^{d-m-1}$ to the
first and second component respectively.

The blow-up of a domain $U\subseteq\mathbb{C}^{d}$ along $X\cap U$
is then defined as the restriction $\tilde{U}_{X}:=\sigma^{-1}(U)$
with the respective restricted projections.
\end{defn}

Again, the projection $\pi:\tilde{\mathbb{C}}_{X}^{d}\to\mathbb{P}^{d-m-1}$
makes $\tilde{\mathbb{C}}_{X}^{d}$ a $\mathbb{C}^{m+1}$-bundle over
$\mathbb{P}^{d-m-1}$ with fibre $X+[v]$ over $[v]$ and the projection
$\sigma:\tilde{\mathbb{C}}_{X}^{d}\to\mathbb{C}^{d}$ is an isomorphism
on $\mathbb{C}^{d}\backslash X$. The set $\sigma^{-1}(X)$ is now
canonically isomorphic to the projectivised normal bundle $\mathbb{P}(T\mathbb{C}^{d}/TX)$
of $X$. In particular, $\sigma^{-1}(X)$ is a codimension one submanifold
of $\tilde{\mathbb{C}}_{X}^{d}$, i.e.\ a divisor. 
\begin{defn}
We call $E_{X}:=\sigma^{-1}(X)$ the \emph{\index{exceptional divisor}exceptional
divisor} of the blow-up $\sigma$.
\end{defn}

\subsubsection{\label{subsec:BlowupCharts}Charts}

Let now $I=\{i_{1},\ldots,i_{m}\}$ and $\overline{I}=\{i_{m+1},\ldots,i_{d}\}$,
with $i_{1}<\cdots<i_{m}$ and $i_{m+1}<\cdots<i_{d}$. We canonically
obtain charts for the bundle $\pi:\tilde{\mathbb{C}}_{X}^{d}\to\mathbb{P}^{d-m-1}$
from the standard fibre charts of $\mathbb{P}(\mathbb{C}^{d}/X)$
given by 
\[
\varphi_{j}:\{[v]\in\mathbb{P}(\mathbb{C}^{d}/X)\mid v_{j}\neq0\}\to\mathbb{C}^{d-m-1},\quad[v]\mapsto\frac{1}{v_{j}}(v_{i_{m+1}},\ldots\hat{j}\ldots,v_{i_{d}})
\]
for $j\in\overline{I}$ and the trivialisations 
\[
\theta_{j}:\underbrace{\pi^{-1}(\{[v]\in\mathbb{P}(\mathbb{C}^{d}/X)\mid v_{j}\neq0\})}_{=:V_{j}}\to\{[v]\in\mathbb{P}(X^{\perp})\mid v_{j}\neq0\}\times\mathbb{C}^{m+1},\quad(z,[v])\mapsto([v],z',z_{j})
\]
for $j\in\overline{I}$. Reordering the components by their indices
these combine to charts 
\begin{equation}
\chi_{j}:\underbrace{\{(z,[v])\in\mathbb{C}^{d}\times\mathbb{P}(\mathbb{C}^{d}/X)\mid v_{j}\neq0,z\in[v]\}}_{=V_{j}}\to\mathbb{C}^{d},\quad\chi_{j}(z,[v])_{i}=\begin{cases}
z_{i}, & i\in I\cup\{j\}\\
\frac{v_{i}}{v_{j}}, & \text{otherwise}
\end{cases}\label{eq:BlowUpCharts}
\end{equation}
centred at $(0,[e_{j}])\in V_{j}$ for each $j\in\overline{I}$. Note
that outside of $X$, we have $\frac{v_{i}}{v_{j}}=\frac{z_{i}}{z_{j}}$
 and in $X$, we have $z_{j}=0$, so for all $w\in V_{j}$, we have
\[
(\sigma\circ\chi_{j}^{-1}(w))_{i}=\begin{cases}
w_{i}, & i\in I\cup\{j\}\\
w_{i}w_{j}, & \text{otherwise}
\end{cases}
\]
 making the projection $\sigma:\tilde{\mathbb{C}}_{X}^{d}\to\mathbb{C}^{d}$
a holomorphic map. In other words, we obtain $z$-coordinates on $\mathbb{C}^{d}$
from $w$-coordinates on $V_{j}$ via
\begin{equation}
z_{i}=\begin{cases}
w_{i}, & i\in I\cup\{j\}\\
w_{i}w_{j}, & \text{otherwise}.
\end{cases}\label{eq:coordinatesSimpleBlowup}
\end{equation}

\subsection{General Blow-ups and Modifications}

Let $M$ be a complex manifold of dimension $d$ and $X\subseteq M$
be a submanifold of dimension $m$
\begin{defn}
A chart $\varphi:V\to\mathbb{C}^{d}$ of $M$ is \emph{adapted\index{adapted chart (to a submanifold)}}
to $X$, if there is a splitting $I$ of $d$ such that $X\cap V=\{p\in V\mid\varphi(p)_{j}=0\text{ for all }j\in\overline{I}\}=\varphi^{-1}(X_{I})$.
\end{defn}

Let $\varphi:V\to\mathbb{C}^{d}$ be a chart adapted to $X$ and let
$\sigma:\tilde{U}_{X_{I}}\to\mathbb{C}^{d}$ be the blow-up of $U=\varphi(V)$
at $X_{I}=\varphi(X\cap V)$. For an atlas of charts adapted to $X$,
these local blow-ups glue naturally to yield a manifold structure
on $M\backslash X\cup\mathbb{P}(\mathcal{N}_{M/X})$, where $\mathcal{N}_{M/X}:=TM/TX$
is the normal bundle and $\mathbb{P}(\mathcal{N}_{M/X})$ its projectivisation
(see \cite[Prop.~2.5.3]{Huybrechts2005ComplexgeometryAnintroduction}).
\begin{defn}
The \emph{blow-up} of a complex manifold\index{blow-up!of a complex manifold}
$M$ along a \uline{closed} submanifold $X$ is the set $M\backslash X\cup\mathbb{P}(\mathcal{N}_{M/X})$
equipped with the manifold structure explained above and the projection
$\sigma:M\backslash X\cup\mathbb{P}(\mathcal{N}_{M/X})\to M$. The
manifold $X$ is also called the \emph{centre} \index{centre of a blow-up}of
the blow-up $\sigma$.
\end{defn}

\begin{defn}
The hypersurface $\sigma^{-1}(M)$ is called the \emph{exceptional
divisor\index{exceptional divisor!of a blow-up}} of the blow-up $\sigma$.
\end{defn}

Submanifolds of can be lifted to the blow-up as follows.
\begin{defn}
The \emph{proper transform\index{proper transform}} $\tilde{Y}$
of a submanifold $Y\subseteq M$ in $\tilde{M}_{X}$ is the closure
$\overline{\sigma^{-1}(Y\backslash X)}$.
\end{defn}

The general definition allows us to repeatedly blow up along new
centres.
\begin{defn}
A map $\pi:\tilde{M}\to M$ is called a \emph{modification}, if it
is a composition of blow-ups, that is, there exist blow-ups $\sigma_{j}:M^{j}\to M^{j-1},j\in\{1,\ldots,k\}$
such that the diagram \begin{equation*}
    \begin{tikzpicture}[baseline = (current bounding  box.center)]
        \matrix (m) [matrix of math nodes,
        row sep=3em, column sep=3em,
		text height=1.5ex, text depth=0.25ex]
        {
		  M=M^0 & M^1 & \cdots & M^{k-1} & M^k=\tilde{M}\\
        };
        \path[font=\scriptsize]
		(m-1-1) edge[bend right,<-] node[auto] {$\pi$} (m-1-5)
        (m-1-1) edge[<-] node[auto] {$\sigma_1$} (m-1-2)
        (m-1-2) edge[<-] node[auto] {$\sigma_2$} (m-1-3)
        (m-1-3) edge[<-] node[auto] {$\sigma_{k-1}$} (m-1-4)
        (m-1-4) edge[<-] node[auto] {$\sigma_k$} (m-1-5)
		;
    \end{tikzpicture}
\end{equation*} commutes.
\end{defn}

\begin{defn}
Let $M$ be a manifold and $X\subseteq M$ a closed submanifold. For
a modification $\pi=\sigma_{1}\circ\ldots\sigma_{k}$ as above, we
define the \emph{iterated proper transform\index{proper transform!iterated}
}of $X^{j}$ in $M^{j}$ recursively as the proper transform of the
iterated proper transform $X^{j-1}$ of $X$ in $M^{j-1}$.
\end{defn}

\begin{defn}
The \emph{exceptional divisor\index{exceptional divisor!of a modification}}
$E_{\pi}$ of a modification $\pi=\sigma_{1}\circ\ldots\sigma_{k}:\tilde{M}\to M$
is the union of the iterated proper transforms in $\tilde{M}$ of
the exceptional divisors of $\sigma_{1},\ldots,\sigma_{k}$.
\end{defn}

\begin{note*}
The exceptional divisor of a modification is no longer necessarily
a submanifold, but a union of hypersurfaces. It may also be considered
a formal sum of hypersurfaces, leading to the classical definition
of divisors. For more on divisors see \cite[\textsection 2.3]{Huybrechts2005ComplexgeometryAnintroduction}.
\end{note*}

\subsection{Lifting Maps to Blow-Ups}

\subsubsection{Generalised Germs}

Returning to local dynamics, we wish to employ blow-ups to differentiate
the directions approaching our fixed point $p$ at $M$. Blowing-up
replaces a point by a closed submanifold, so we  need a generalised
definition of germs.
\begin{defn}
Let $M$ and $N$ be complex manifolds (or topological spaces) and
$X\subseteq M$ be a closed subset. Let $U$ and $V$ be open neighbourhoods
of $X$ in $M$. Then we call continuous maps $f:U\to N$ and $g:V\to N$
equivalent if they coincide on some (possibly small) open neighbourhood
$W\subseteq U\cap V$ of $X$. The equivalence classes of this relation
are again called \emph{germs\index{germ of a map!at a closed subset}
(of maps)} or \emph{local maps\index{local map!at a closed subset}
}from $M$ to $N$ at $X$.
\end{defn}

\begin{notation}
As before, we use the same symbol for a germ $F$ and a representative
$F:U\to N$ and if $F(X)\subseteq Y$ for a closed subset $Y\subseteq N$,
we further write $F:(M,X)\to(N,Y)$.
\end{notation}

\begin{defn}
Let $\End(M,X)$ denote the set of holomorphic germs $F:(M,X)\to(M,X)$
mapping $X$ into itself.
\end{defn}

\subsubsection{\label{subsec:UniversalPropertyOfBlowup}Universal Property}

Blow-ups satisfy the following universal property (see \cite[\textsection 4.1]{Fischer1976Complexanalyticgeometry}
for a general proof).
\begin{thm}
\label{thm:UnivPropBlowUp}Let $M$ be a complex manifold and $X\subseteq M$
a closed submanifold. The blow-up $\sigma:\tilde{M}_{X}\to M$ of
$M$ along $X$ satisfies the following universal property: Any holomorphic
map $\tau:N\to M$ such that $\tau^{-1}(X)$ is a hypersurface, can
be lifted to $\tilde{M}_{X}$, i.e.\  for any such $\tau$, there
exists a unique holomorphic map $\tilde{\tau}:N\to\tilde{M}_{X}$
such that the diagram \begin{equation*}
    \begin{tikzpicture}[baseline = (current bounding  box.center)]
        \matrix (m) [matrix of math nodes,
        row sep=3em, column sep=5em,
		text height=1.5ex, text depth=0.25ex]
        {
		  & \tilde M_X \\
		  N & M\\
        };
        \path[font=\scriptsize]
		(m-1-2) edge[->] node[auto] {$\sigma$} (m-2-2)
		(m-2-1) edge[dashed,->] node[auto] {$\exists!\tilde \tau$} (m-1-2)
        (m-2-1) edge[->] node[auto] {$\tau$} (m-2-2)
		;
    \end{tikzpicture}
\end{equation*}commutes.
\end{thm}

\begin{note}
As usual, this universal property provides an equivalent definition
of the blow-up (up to isomorphisms).
\end{note}

\begin{singlespace}

\end{singlespace}

\begin{defn}
A germ $F\in\End(M,X)$ with $F^{-1}(X)=X$ is called \emph{non-degenerate
along $X$\index{non-degenerate germ!along a submanifold}.}
\end{defn}

\begin{cor}
\label{cor:LiftingMapstoBlowUP}For a non-degenerate germ $F\in\End(M,X)$,
there exists a unique lift under the blow-up $\sigma:\tilde{M}_{X}\to M$
along $X$, that is, a germ $\tilde{F}\in\End(\tilde{M}_{X},E_{X})$
such that \begin{equation}
    \begin{tikzpicture}[baseline = (current bounding  box.center)]
        \matrix (m) [matrix of math nodes,
        row sep=3em, column sep=5em,
		text height=1.5ex, text depth=0.25ex]
        {
		  (\tilde{M}_{X},E_X) & (\tilde{M}_{X},E_X) \\
		  \left(M,X\right) & \left(M,X\right) \\
        };
        \path[font=\scriptsize]
		(m-1-1) edge[->] node[auto] {$\tilde F$} (m-1-2)
		(m-1-1) edge[->] node[auto] {$\sigma$} (m-2-1)
		(m-1-1) edge[->] node[auto] {$\tau$} (m-2-2)
		(m-1-2) edge[->] node[auto] {$\sigma$} (m-2-2)
        (m-2-1) edge[->] node[auto] {$F$} (m-2-2)
		;
    \end{tikzpicture}
	\label{eq:LiftToBlowUp}
\end{equation}commutes.
\end{cor}

\begin{proof}
The map $\tau:=F\circ\sigma:(\tilde{M}_{X},E_{X})\to(M,X)$ satisfies
the conditions of the universal property \ref{thm:UnivPropBlowUp},
since $\tau^{-1}(X)=\sigma^{-1}\circ F^{-1}(X)=\sigma^{-1}(X)\subseteq\tilde{M}_{X}$
is just the exceptional divisor $E_{X}$, which is a hypersurface.
Thus we can lift $\tau$ to a unique germ $\tilde{F}:(\tilde{M}_{X},E_{X})\to(\tilde{M}_{X},E_{X})$
such that \ref{eq:LiftToBlowUp} commutes. 
\end{proof}
This allows us to use blow-ups as a new type of change of coordinates,
which has the drawback of not being invertible along the exceptional
divisor. We can however still establish a useful formula for the lift
on the exceptional divisor. 
\begin{prop}
\label{prop:BlowUpOnExcepDiv}Let $F\in\End(M,X)$ be invertible and
non-degenerate along $X$. Then the lift $\tilde{F}$ to $\tilde{M}_{X}$
on the exceptional divisor $E_{X}\cong\mathbb{P}(TM/TX)$ is given
by 
\begin{equation}
\tilde{F}([v+T_{p}X])=[dF_{p}(v)+T_{p}X]\quad\text{for all }v\in T_{p}M,p\in X.\label{eq:LiftOnEBydF-1}
\end{equation}
\end{prop}

\begin{proof}
By non-degeneracy, $\tilde{F}(w)=\sigma^{-1}\circ F\circ\sigma(w)$
is well defined for $w\in\tilde{M}_{X}\backslash E_{X}$ (since $F\circ\sigma(w)\notin X$).
Taking the limit $w\to E_{X}$ reduces to the (non-degenerate) linear
part of $F$.
\end{proof}
\begin{note}
A more general version of this Proposition can be used to directly
prove Corollary~\ref{cor:LiftingMapstoBlowUP} (see \cite[Prop.~2.1]{Abate2000Diagonalizationofnondiagonalizablediscreteholomorphicdynamicalsystems}).
\end{note}

\subsubsection{\label{subsec:LiftCoords}The Lift in Coordinates}
\begin{lem}
\label{lem:NoProblematicTerms}Let $F\in\End(\mathbb{C}^{d},0)$ be
invertible with $dF_{0}$ in Jordan canonical form given by 
\[
F(z)=\sum_{\alpha\in\mathbb{N}_{0}^{d}}p_{\alpha}z^{\alpha},
\]
and non-degenerate along $X=X_{I}=\langle e_{i}\mid i\in I\rangle$
for a splitting $I\subseteq\{1,\ldots,d\}$. Then, using Notation~\ref{nota:PrimeNotation},
there are no terms $(z')^{\beta}$ in $F''$.
\end{lem}

\begin{proof}
Assume there is a term $(z')^{\beta}$ in $F''$ of minimal order.
Then there is a point $z\in X$ such that $(z')^{\beta}\neq0$ and
hence $F_{i}(\varepsilon z)\neq0$ for small $\varepsilon>0$, and
thus $F(\varepsilon z)\notin X$, a contradiction.
\end{proof}
Let $\tilde{F}$ be the lift of $F$ under the blow-up $\sigma$ along
$X$. Now for a proper eigendirection $[e_{j}]$of $dF_{0}$ with
$j\in\overline{I}$ let $(V_{j},\chi_{j})$ be the chart centred at
$[e_{j}]$ defined in (\ref{eq:BlowUpCharts}). Then $\tilde{F}$
fixes $[e_{i}]$ by Proposition~\ref{prop:BlowUpOnExcepDiv}. For
a multi-index $\alpha\in\mathbb{N}_{0}^{d}$ set $k_{\alpha}:=\sum_{i\in\overline{I}\backslash\{j\}}\alpha_{i}$.
Then a monomial $z^{\alpha}$ in $F_{i}$ occurs in $\tilde{F}_{i}$
as 
\begin{equation}
\begin{alignedat}{1}w^{\alpha}w_{j}^{k_{\alpha}} & \quad\text{if }i\in I\cup\{j\},\\
w^{\alpha}w_{j}^{k_{\alpha}}/F_{j}(z) & =w^{\alpha}w_{j}^{k_{\alpha}}\paren[{\Big}]{\lambda_{j}w_{j}+\sum_{|\beta|\ge2}p_{\beta,j}w^{\beta}w_{j}^{k_{\beta}}}^{-1}\\
 & =\lambda_{j}^{-1}w^{\alpha}w_{j}^{k_{\alpha}-1}\paren[{\Big}]{1+\sum_{|\beta|\ge2}\frac{p_{\beta,j}}{\lambda_{j}}w^{\beta}w_{j}^{k_{\beta}-1}}^{-1}\\
 & =\lambda_{j}^{-1}w^{\alpha}w_{j}^{k_{\alpha}-1}\sum_{k=0}^{\infty}\paren[{\Big}]{-\sum_{|\beta|\ge2}\frac{p_{\beta,j}}{\lambda_{j}}w^{\beta}w_{j}^{k_{\beta}-1}}^{k}\\
 & =\lambda_{j}^{-1}w^{\alpha}w_{j}^{k_{\alpha}-1}+HOT\quad\text{if }i\in\overline{I}\backslash\{j\}.
\end{alignedat}
\label{eq:monomialTransfBlowUp}
\end{equation}

\begin{rem}
\label{rem:ProblematicTerms}These are well-defined power series by
Lemma~\ref{lem:NoProblematicTerms}, since $\beta_{j}+k_{\beta}=0$
occurs if and only if $\beta''=0$, i.e.\ $z^{\beta}=(z')^{\beta'}$.
\end{rem}

\begin{rem}
\label{rem:BlowUpsMultipliers}In particular, the eigenvalues of $d\tilde{F}_{[e_{j}]}$
are 
\[
\tilde{\lambda}_{i}=\begin{cases}
\lambda_{i}, & \text{for }i\in I\cup\{j\},\\
\frac{\lambda_{i}}{\lambda_{j}}, & \text{otherwise},
\end{cases}
\]
so any resonance $\lambda_{i}=\lambda^{\alpha}$ for $F$ induces
a resonance for $\tilde{F}$ at $[e_{j}]$ of the form 
\[
\tilde{\lambda}_{i}=\begin{cases}
\tilde{\lambda}^{\alpha}\tilde{\lambda}_{j}^{k_{\alpha}}, & \text{for }i\in I\cup\{j\}\\
\tilde{\lambda}^{\alpha}\tilde{\lambda}_{j}^{k_{\alpha}-1}, & \text{otherwise}.
\end{cases}
\]
Comparing this with (\ref{eq:monomialTransfBlowUp}), we observe that
resonant monomials stay resonant at the lowest order and even at all
higher orders that come from resonant terms $z^{\beta}$ in $F_{j}$.
In conclusion, elimination of non-resonant terms up to high orders
is preserved under these blow-ups.
\end{rem}

\begin{example}
\label{exa:NewResonancesAfterBlowUp}There may still occur new resonances
in $\tilde{\lambda}$ not generated by those of $\lambda$. For example,
if $\lambda_{j}^{3}=\lambda_{k}^{2}$ for some $k\notin I\cup\{j\}$,
this is not a resonance for $F$, but for $\tilde{F}$ at $[e_{j}]$,
we get a resonance 
\[
\tilde{\lambda}_{j}=\lambda_{j}=\paren[{\Big}]{\frac{\lambda_{k}}{\lambda_{j}}}^{2}=\tilde{\lambda}_{k}^{2}.
\]
\end{example}

\begin{rem}
Even if there are new resonances, the blown-up germ $\tilde{F}$ will
not contain any resonant terms corresponding to these resonances,
as those transformed like the multipliers, and hence do not stem from
a polynomial in $F$. In our example above, a term $w_{k}^{2}$ in
$\tilde{F}_{i}$ would have been a term $z_{k}^{2}z_{j}^{-2}$ in
$F_{j}$.
\end{rem}

\begin{singlespace}

\end{singlespace}

\subsection{\label{subsec:Diagonalisation}Diagonalisation of Jordan blocks}

\selectlanguage{british}
\global\long\def\Transp{\mathsf{T}}%
\global\long\def\Pow{\m{Pow}}%
\global\long\def\ord{\m{ord}}%

If the germ $F\in\End(\mathbb{C}^{d},0)$ is not hyperbolic with linear
part $L=dF_{0}$, then the presence of a Jordan block corresponding
to a neutral eigenvalue in $L$, significantly alters the dynamics
of the linear part and can be an obstruction to finding holomorphic
normal forms (see Theorem~\ref{thm:JordanBlockNoHolLin}). Many result
in the following sections hence assume the linear part $L$ to be
diagonalisable. To tackle germs with non-diagonalisable linear part,
Abate describes in \cite{Abate2000Diagonalizationofnondiagonalizablediscreteholomorphicdynamicalsystems}
a general procedure to diagonalise any germ using a sequence of blow-ups:
\begin{thm}[{Diagonalisation Theorem \cite[Thm. 2.4]{Abate2000Diagonalizationofnondiagonalizablediscreteholomorphicdynamicalsystems}}]
\label{thm:AbateDiagonalisation}Let $F\in\End(\mathbb{C}^{d},0)$
be invertible such that 
\[
dF_{0}=\diag(J_{\mu_{1}}(\lambda_{1}),\ldots,J_{\mu_{\rho}}(\lambda_{\rho})),
\]
where $\mu_{1}+\cdots+\mu_{\rho}=d$ and $J_{m}(\lambda)$ denotes
the Jordan block of size $m\in\mathbb{N}$ for the eigenvalue $\lambda\in\mathbb{C}$.

Then there exists a complex $d$-dimensional manifold $M$, a modification
$\pi:M\to\mathbb{C}^{d}$, a point $p\in M$ and a holomorphic lift
$\tilde{F}\in\End(M,\pi^{-1}(0))$ such that 
\begin{enumerate}
\item $\pi:M\backslash\pi^{-1}(0)\to\mathbb{C}^{d}\backslash\{0\}$ is biholomorphic,
\item $\pi\circ\tilde{F}=F\circ\pi$ where defined,
\item $p$ is a fixed point of $\tilde{F}$ and $d\tilde{F}_{p}$ is diagonalizable
as 
\[
\diag(\tilde{\lambda}_{1},D_{\mu_{1}-1}(1),D_{\mu_{2}}(\lambda_{2}/\lambda_{1}),\ldots,D_{\mu_{\rho}}(\lambda_{\rho}/\lambda_{1})),\quad\text{with }\tilde{\lambda}_{1}=\begin{cases}
\lambda_{1}, & \text{if }\mu_{1}>\mu_{2}\\
\frac{\lambda_{1}^{2}}{\lambda_{2}}, & \text{if }\mu_{1}=\mu_{2},
\end{cases}
\]
where $D_{m}(\lambda):=\diag(\lambda,\ldots,\lambda)\in\mathbb{C}^{m\times m}$. 
\end{enumerate}
\end{thm}

\begin{rem}
\label{rem:ExceptionalDivisorCollapsingDynamics}An important caveat
of the diagonalisation theorem is, that the local dynamics of $\tilde{F}$
near $p$ contain only local dynamics of $F$ tangent to $e_{1}$
and the dynamics of $\tilde{F}$ contained in $\pi^{-1}(0)$ will
collapse to $0$ for $F$.
\end{rem}

\begin{singlespace}
\end{singlespace}

\section{\label{sec:Parabolic-cases}Parabolic fixed points}

\selectlanguage{british}

The results in this section are on germs whose only multiplier is
$1$. As in dimension one, any parabolic germ $F\in\End(\mathbb{C}^{d},0)$
has an iterate $F^{\circ n}$ with only multiplier $1$. For most
of this section, we will assume the linear part of $F$ to be diagonalisable,
or\ $F$ to be tangent to the identity:
\begin{defn}
Let $F\in\Pow(\mathbb{C}^{d},0)$. Then  $F$ is \emph{tangent to
the identity\index{tangent germ!to the identity}},  of order $k+1$,
if it has homogeneous expansion 
\begin{equation}
F=\id+P_{k+1}+P_{k+2}+\cdots,\label{eq:TangIdHomExp}
\end{equation}
where $P_{k+1}\neq0$ and $P_{n}$ is a homogeneous polynomial of
order $n$ for all $n\ge k+1$.  Since the linear part is the identity,
the dominant term for the dynamics is the first non-linear term $P_{k+1}$. 
\end{defn}

To find an analogue to one-dimensional parabolic flowers in several
variables, we would like to focus on individual complex directions.
To do so, we need to identify complex directions that are close to
being preserved. More precisely, in \cite{Hakim1998AnalytictransformationsofmathbbCp0tangenttotheidentity},
Hakim calls directions that are preserved by the dominant term \emph{characteristic}:
\begin{defn}
Let $F\in\End(\mathbb{C}^{d},0)$ be tangent to the identity of order
$k+1\ge2$ of the form (\ref{eq:TangIdHomExp}). A \emph{\index{characteristic direction!parabolic}characteristic
direction} of $F$ is the complex direction $[v]\in\mathbb{P}^{d-1}$
of a vector $v\in\mathbb{C}^{d}\backslash\{0\}$ such that 
\[
P_{k+1}(v)=\gamma v
\]
for some $\gamma\in\mathbb{C}$. If $\gamma=0$, we call $[v]$ \emph{degenerate}\index{degenerate characteristic direction},
otherwise \emph{non-degenerate}\index{non-degenerate characteristic direction}. 
\end{defn}

\begin{rem}
\label{rem:CharDirAsFP}The map $P_{k+1}$ in (\ref{eq:TangIdHomExp})
is a homogeneous polynomial of order $k+1$, so if its components
have no common zeroes, $P_{k+1}$ induces a holomorphic self-map $P_{k+1}\in\End(\mathbb{P}^{d-1})$
of $\mathbb{P}^{d-1}$ of order $k+1$, whose fixed points $[v]\in\mathbb{P}^{d-1}$
are precisely the (non-degenerate) characteristic directions of $F$.
 Hence, generically, $F$ has $((k+1)^{d}-1)/k$ non-degenerate characteristic
directions, counted with multiplicity as fixed points of $P_{k+1}$
(see e.g.\ \cite{AbateTovena2003ParaboliccurvesinmathbbC3}).

In the other extreme, if all complex directions are characteristic,
the germ $F$ is called \emph{\index{dicritical@\emph{dicritical}}dicritical}
and this is the case if and only if $P_{k+1}(z)=p_{0}\cdot z$, where
$p_{0}\in\mathbb{C}[z]$ is homogeneous of degree $k$.
\end{rem}

\begin{defn}
Let $F\in\End(\mathbb{C}^{d},0)$. An orbit $\{F^{\circ n}(z)\}_{n\in\mathbb{N}}$
is said to converge to $0$ \emph{tangentially\index{tangential convergence}}
to a complex direction $[v]\in\mathbb{P}^{d-1}$, if $F^{\circ n}(z)\xrightarrow[n\to\infty]{}0$
and $[F^{\circ n}(z)]\xrightarrow[n\to\infty]{}[v]$ in $\mathbb{P}^{d-1}$.
More generally, the convergence is called tangential to a submanifold
$M\subseteq\mathbb{C}^{d}$ containing $0$ or the linear subspace
$V=T_{0}M\subseteq\mathbb{C}^{d}$, if $F^{\circ n}(z)\xrightarrow[n\to\infty]{}0$
and $[F^{\circ n}(z)+V]\xrightarrow[n\to\infty]{}[V]$ in $\mathbb{P}(\mathbb{C}^{d}/V)$.
\end{defn}

The first basic observation is that an orbit converging to $0$ tangentially
to a complex direction forces that direction to be characteristic:
\begin{prop}[Hakim \cite{Hakim1998AnalytictransformationsofmathbbCp0tangenttotheidentity}]
Let $F\in\End(\mathbb{C}^{d},0)$ be tangent to the identity. If
there is an orbit of $F$ converging tangentially to a direction $[v]\in\mathbb{P}^{d-1}$,
then $[v]$ is necessarily a characteristic direction. 
\end{prop}

\begin{rem}
There exist examples of germs tangent to identity that admit orbits
that converge to $0$, but not tangent to any complex direction (see
\cite{Rivi1998LocalBehaviourofDiscreteDynamicalSystems} and \cite{AbateTovena2011PoincareBendixsonTheoremsforMeromorphicConnectionsandHolomorphicHomogeneousVectorFields}).
\end{rem}

\begin{defn}
A \emph{parabolic manifold\index{parabolic manifold}} of dimension
$m\le d$ for a germ $F\in\End(\mathbb{C}^{d},0)$ is (the image $P=\varphi(\Delta)$)
of a holomorphic map $\varphi:\Delta\to\mathbb{C}^{d}$, where 
\begin{enumerate}
\item $\Delta\subseteq\mathbb{C}^{m}$ is a simply connected open subset
with $0\in\partial\Delta$,
\item $\varphi$ extends continuously to $\partial\Delta$ with $\varphi(0)=0$,
\item $P=\varphi(\Delta)$ is $F$-invariant and all orbits in $P$ converge
to $0$.
\end{enumerate}
If all orbits in $P$ converge to $0$ along a complex direction $[v]\in\mathbb{P}^{d-1}$,
then we say $P$ is a parabolic manifold \emph{along $[v]$}.

A parabolic manifold of dimension $1$ is called a \emph{parabolic
curve}\index{parabolic curve}, one of dimension $d$ a \emph{\index{parabolic domain}parabolic
domain}.

If $F$ is parabolic of order $k+1\ge2$, a set of $k$ distinct parabolic
manifolds of dimension $m$ along a fixed direction $[v]\in\mathbb{P}^{d-1}$
is called a \emph{parabolic flower\index{parabolic flower}} of dimension
$m$ for $F$ along $[v]$.

\end{defn}

The first generalisation of the Leau-Fatou theorem~\ref{thm:LeauFatou}
was given by Écalle \cite{Ecalle1985LesFonctionsResurgentes3LequationDuPontEtLaClassificationAnalytiqueDesObjectsLocaux}
using his theory of resurgent series and Hakim \cite{Hakim1998AnalytictransformationsofmathbbCp0tangenttotheidentity}
using more classical techniques (see also \cite{ArizziRaissy2014OnEcalleHakimsTheoremsinHolomorphicDynamics}
for a detailed exposition of Hakim's approach).
\begin{thm}[Écalle \cite{Ecalle1985LesFonctionsResurgentes3LequationDuPontEtLaClassificationAnalytiqueDesObjectsLocaux},
Hakim \cite{Hakim1998AnalytictransformationsofmathbbCp0tangenttotheidentity}]
\label{thm:HakimEcale}Let $F\in\End(\mathbb{C}^{d},0)$ be tangent
to the identity of order $k+1\ge2$. Then for any non-degenerate characteristic
direction $[v]\in\mathbb{P}^{d-1}$, there exist (at least) $k$ parabolic
curves for $F$ tangent to $[v]$.
\end{thm}

\begin{rem}
Characteristic directions are complex directions unlike the real attracting
and repelling directions in one variable. The characteristic directions
of the inverse $F^{-1}$ are the same as the ones of $F$. The generalisation
of the Leau-Fatou theorem~\ref{thm:LeauFatou} lies in the fact,
that a non-degenerate characteristic direction contains the analogue
of a set of $k$ real attracting (and repelling) directions, that
each have one of the $k$ parabolic curves tangent to them and all
orbits in that petal converging to $0$ along that real attracting
direction.
\end{rem}

To obtain higher dimensional parabolic manifolds, we need to examine
the structure of the dominant term. To this end Hakim performs a series
of changes of coordinates summarised as follows:
\begin{prop}[\cite{Hakim1998AnalytictransformationsofmathbbCp0tangenttotheidentity}]
\label{prop:HakimCoords}Let $F\in\End(\mathbb{C}^{d},0)$ be tangent
to the identity of order $k+1\ge2$, with a non-degenerate characteristic
direction $[v]\in\mathbb{P}^{d-1}$. Then there exists local coordinates
$(x,y)\in\mathbb{C}\times\mathbb{C}^{d-1}$ in which $[v]=[e_{1}]$
and such that for $(x_{1},y_{1})=F(x,y)$, we have 
\begin{equation}
\begin{aligned}x_{1} & =x-\frac{1}{k}x^{k+1}+O(\norm yx^{k},x^{2k+1})\\
y_{1} & =y-x^{k}\paren[\Big]{A+\frac{1}{k}I}y+O(\norm y^{2}x^{k-1},yx^{2k},x^{k+2})
\end{aligned}
\label{eq:HakimCoordsDownstairs}
\end{equation}
and the lift of $F$ to the blow-up $y=ux$ is given by
\begin{align*}
x_{1} & =x-\frac{1}{k}x^{k+1}+O(\norm ux^{k+1},x^{2k+1})\\
u_{1} & =(I-x^{k}A)u+O(\norm u^{2}x^{k},x^{k+1}),
\end{align*}
where $y_{1}=u_{1}x_{1}$. The similarity class of the matrix $A$
is preserved under changes of coordinates preserving the form above.
\end{prop}

\begin{defn}
\label{def:DirectorsPara}Let $F\in\End(\mathbb{C}^{d},0)$ be tangent
to the identity of order $k+1\ge2$ with non-degenerate characteristic
direction $[v]\in\mathbb{P}^{d-1}$. Then the eigenvalues of the matrix
$A$ from Proposition~\ref{prop:HakimCoords} are called the \emph{\index{director@\emph{director}}directors}
of $[v]$. If all directors of $[v]$ have positive real part, we
call the non-degenerate characteristic direction $[v]$ \emph{\index{attracting non-degenerate characteristic direction@\emph{attracting non-degenerate characteristic direction}}attracting}.
\end{defn}

\begin{rem}
In the setting of Definition~\ref{def:DirectorsPara}, $P_{k+1}$
induces a local holomorphic self-map $P_{k+1}\in\End(\mathbb{P}^{d-1},[v])$
near the fixed point $[v]$ (compare Remark~\ref{rem:CharDirAsFP}).
The directors of $[v]$ are then precisely the eigenvalues of the
linear operator 
\[
\frac{1}{k}((dP_{k+1})_{[v]}-I):T_{[v]}\mathbb{P}^{d-1}\to T_{[v]}\mathbb{P}^{d-1}.
\]
\end{rem}

Now we can state the main local results of \cite{HakimTransformationstangenttotheidentityStablepiecesofmanifolds}
generalising Theorem~\ref{thm:LeauFatou} (in the formulation of
\cite{ArizziRaissy2014OnEcalleHakimsTheoremsinHolomorphicDynamics}):
\begin{thm}[Hakim \cite{HakimTransformationstangenttotheidentityStablepiecesofmanifolds}]
\label{thm:Hakim2}Let $F\in\End(\mathbb{C}^{d},0)$ be tangent to
the identity of order $k\ge2$ with non-degenerate characteristic
direction $[v]\in\mathbb{P}^{d-1}$ with directors $\alpha_{1},\ldots,\alpha_{d-1}\in\mathbb{C}$.
If there exist $c>0$ and $m\in\mathbb{N}$ such that $\Re\alpha_{1},\ldots,\Re\alpha_{m}>c$
and $\Re\alpha_{m+1},\ldots,\Re\alpha_{d-1}<c$. Then:
\begin{enumerate}
\item There exist a parabolic flower of $k$ parabolic manifolds $P_{1},\ldots,P_{k}$
of dimension $m+1$ along $[v]$, each tangent to the direct sum of
$[v]$ and the generalised eigenspaces of the directors $\alpha_{1},\ldots,\alpha_{m}$.
\item There exists a conjugation on each $P_{j},j=1,\ldots,k$ of $F$ to
the translation $(\zeta,\xi)\mapsto(\zeta+1,\xi)$ with $(\zeta,\xi)\in\mathbb{C}\times\mathbb{C}^{m}$
and $\zeta\sim1/x$ as $(x,y)\to0$.
\item $P_{1},\ldots,P_{k}$ can be taken as the $k$ distinct $F$-invariant
connected components of
\begin{equation}
\{(x,y)\in\mathbb{C}^{d}\mid(x_{n},y_{n})\to0\text{ and }x_{n}^{-kc-1}y_{n}\to0\},\label{eq:reducedHakimBasins}
\end{equation}
where $(x,y)$ are the coordinates giving $F$ the form (\ref{eq:HakimCoordsDownstairs})
and $(x_{n},y_{n}):=F^{\circ n}(x,y)$.
\item If, moreover, $\Re\alpha_{m+1},\ldots,\Re\alpha_{m-1}<0$, then the
set (\ref{eq:reducedHakimBasins}) is equal to 
\[
\Omega_{[v]}:=\{z\in\mathbb{C}^{d}\mid F^{\circ n}(z)\to0\text{ tangentially to }[v]\}.
\]
\end{enumerate}
\end{thm}

\begin{rem}
\label{rem:ParabolicSingularExamples}For $m=d-1$, this ensures the
existence of parabolic domains. This is, however, not a necessary
condition for the existence of parabolic domains at $0$. There are
examples of parabolic domains with logarithmic convergence along a
degenerate characteristic direction (\cite{Ushiki1996Parabolicfixedpointsoftwodimensionalcomplexdynamicalsystems})
and convergence not tangent to any complex direction (\cite{Rivi1998LocalBehaviourofDiscreteDynamicalSystems}).
Other singular examples of parabolic domains are discussed in \cite{Vivas2012FatouBieberbachDomainsAsBasinsofAttractionofAutomorphismsTangenttotheIdentity}
and \cite{AbateTovena2011PoincareBendixsonTheoremsforMeromorphicConnectionsandHolomorphicHomogeneousVectorFields}.
\end{rem}

Though the existence of non-degenerate characteristic directions
is generic, we are still looking for general criteria for the existence
of parabolic manifolds. In dimension $2$, López-Hernanz and Rosas
in \cite{LopezHernanzRosas2020CharacteristicDirectionsofTwoDimensionalBiholomorphisms}
proved the existence of parabolic manifolds tangent to any characteristic
direction, building on results from \cite{Vivas2012DegenerateCharacteristicDirectionsforMapsTangenttotheIdentity}
and \cite{LopezHernanzRaissyRibonSanzSanchez2019StableManifoldsofTwoDimensionalBiholomorphismsAsymptotictoFormalCurves}
(see also \cite{Molino2009TheDynamicsofMapsTangenttotheIdentityandwithNonvanishingIndex}
and \cite{LopezHernanzSanzSanchez2018ParabolicCurvesofDiffeomorphismsAsymptotictoFormalInvariantCurves}).
\begin{thm}[\cite{LopezHernanzRosas2020CharacteristicDirectionsofTwoDimensionalBiholomorphisms}]
Let $F\in\End(\mathbb{C}^{2},0)$ be tangent to the identity of order
$k+1$ with a characteristic direction $[v]\in\mathbb{P}^{d-1}$ and
no curve of fixed points through $0$ tangent to $[v]$. Then there
exist $k$ disjoint parabolic manifolds $P_{1},\ldots,P_{k}$ along
$[v]$, such that:
\begin{enumerate}
\item There exists an infinite sequence of blow-ups of points starting in
$0$, such that at each finite step, the liftings of all orbits in
$P_{1}\cup\cdots\cup P_{k}$ converge to a unique point in the exceptional
divisor (tangentially to a unique direction that is the centre of
the next blow-up in the sequence).
\item If all $P_{1},\ldots,P_{k}$ are parabolic domains, then each is foliated
by parabolic curves.
\end{enumerate}
\end{thm}

In particular, this implies the general existence of parabolic curves
first proved by Abate in \cite{Abate2001TheResidualIndexandtheDynamicsofHolomorphicMapsTangenttotheIdentity}
(see also \cite{Bracci2003TheDynamicsofHolomorphicMapsnearCurvesofFixedPoints},
\cite{AbateBracciTovena2004IndexTheoremsforHolomorphicSelfMaps},
and \cite{BrocheroMartinezCanoLopezHernanz2008ParabolicCurvesforDiffeomorphismsinC2}):
\begin{cor}[Abate \cite{Abate2001TheResidualIndexandtheDynamicsofHolomorphicMapsTangenttotheIdentity}]
If $F\in\End(\mathbb{C}^{2},0)$ is tangent to the identity of with
an isolated fixed point at $0$, then $F$ admits a parabolic flower
of parabolic curves tangent to a characteristic direction.
\end{cor}

Even if we allow Jordan blocks, Abate's blow-up procedure of Theorem~\ref{thm:AbateDiagonalisation}
leads to a germ tangent to the identity and the above theorem can
be applied to show:
\begin{thm}[Abate \cite{Abate2001TheResidualIndexandtheDynamicsofHolomorphicMapsTangenttotheIdentity}]
If $F\in\End(\mathbb{C}^{2},0)$ is a parabolic germ, whose linear
part is a single Jordan block of dimension $2$ with eigenvalue $1$
, with an isolated fixed point at $0$, then $f$ admits at least
one parabolic curve tangent to $[1:0]$.
\end{thm}

\section{Parabolic-hyperbolic fixed points}

\selectlanguage{british}

In the parabolic-hyperbolic case, parabolic manifolds have been constructed
by combining attracting and parabolic dynamics. 
\begin{defn}
A (formal) germ $F\in\Pow(\mathbb{C}^{d},0)$ is called \emph{parabolic-attracting\index{parabolic-attracting germ}}
if $F$ has both parabolic and attracting and no other multipliers.
\end{defn}

The parabolic-attracting case is traditionally known as \emph{semi-attractive}\index{semi-attractive germ@\emph{semi-attractive germ}},
but we use the prefix ``semi'' to include all other types of multipliers.
 Heuristically, the components corresponding to attracting directions
in an orbit shrink so fast that the parabolic dynamics dominate the
behaviour. 

\subsection{Curves of fixed points}

In the simplest case, the parabolic part is the identity, i.e.~there
is a manifold of fixed points through $0$.
\begin{thm}
Let $F\in\End(\mathbb{C}^{d},0)$ with multipliers $\lambda_{1},\ldots,\lambda_{m},1,\ldots,1$
such that $(\lambda_{1},\ldots,\lambda_{m})$ has no resonances. Let
$S\subseteq\mathbb{C}^{d}$ a complex manifold through $0$ of codimension
$m$, pointwise fixed by $F$. 

the space $T_{p}S=\Eig(dF_{p},1)$ is the eigenspace associated to
the eigenvalue $1$ of $dF_{p}$ and all other eigenvalues of $dF_{p}$
are attracting and without resonances among them. Then there exists
a unique holomorphic map $\varphi:A_{F}(S)\to\mathcal{N}_{M/S}$ from
the realm of attraction 
\[
A_{F}(S):=\{z\in M\mid d(F^{\circ n}(z),S)\to0\}
\]
 to the normal bundle $\mathcal{N}_{M/S}=TM/TS$ pointwise fixing
$S$ and such that 
\begin{equation}
dF\circ\varphi=\varphi\circ F,\label{eq:NishimuraConj-1}
\end{equation}
where the action of $dF$ on the normal bundle descends from $TM$.
Moreover, there exists a neighbourhood $U\subseteq A_{F}(S)$ of $S$
such that $\varphi|_{U}$ is a biholomorphism onto its image and hence
(\ref{eq:NishimuraConj-1}) is a conjugation of $F$ to $dF$ on $U$.
If $F$ is an automorphism of $M$, then this holds for $U=A_{F}(S)$
and $\varphi$ maps $A_{F}(S)$ biholomorphically onto $\mathcal{N}_{M/S}$.
\end{thm}

\begin{thm}[Nishimura \cite{Nishimura1983AutomorphismesAnalytiquesAdmettantDesSousVarietesDePointsFixesAttractivesDansLaDirectionTransversale}]
\label{thm:NishimuraSemiAttrFixedCurve}Let $M$ be a complex manifold
and $F\in\Aut(M,S)$ a germ at a holomorphic submanifold $S\subseteq M$
that  pointwise fix a holomorphic submanifold $S\subseteq M$ such
that for every $p\in S$, the space $T_{p}S=\Eig(dF_{p},1)$ is the
eigenspace associated to the eigenvalue $1$ of $dF_{p}$ and all
other eigenvalues of $dF_{p}$ are attracting and without resonances
among them. Then there exists a unique holomorphic map $\varphi:A_{F}(S)\to\mathcal{N}_{M/S}$
from the realm of attraction 
\[
A_{F}(S):=\{z\in M\mid d(F^{\circ n}(z),S)\to0\}
\]
 to the normal bundle $\mathcal{N}_{M/S}=TM/TS$ pointwise fixing
$S$ and such that 
\begin{equation}
dF\circ\varphi=\varphi\circ F,\label{eq:NishimuraConj}
\end{equation}
where the action of $dF$ on the normal bundle descends from $TM$.
Moreover, there exists a neighbourhood $U\subseteq A_{F}(S)$ of $S$
such that $\varphi|_{U}$ is a biholomorphism onto its image and hence
(\ref{eq:NishimuraConj}) is a conjugation of $F$ to $dF$ on $U$.
If $F$ is an automorphism of $M$, then this holds for $U=A_{F}(S)$
and $\varphi$ maps $A_{F}(S)$ biholomorphically onto $\mathcal{N}_{M/S}$.
\end{thm}

In other words, centred at each point $p\in S$, we can find regularising
coordinates $(x,y)\in\mathbb{C}^{m}\times\mathbb{C}^{d-m}$ such that
$S=\{x=0\}$ and 
\[
F(x,y)=(A(y)x,y)
\]
with $A(y)\in\mathbb{C}^{m\times m}$.

\subsection{One parabolic multiplier}

The next harder case is if the parabolic dynamics are one-dimensional
and all hyperbolic multipliers are of the same type. In dimension
$2$, this is the only parabolic-hyperbolic case and has been studied
first by Fatou \cite{Fatou1924SubstitutionsAnalytiquesEtEquationsFonctionnellesaDeuxVariables}
and then in more detail by Ueda \cite{Ueda1986LocalstructureofanalytictransformationsoftwocomplexvariablesI,Ueda1991LocalstructureofanalytictransformationsoftwocomplexvariablesII}
who gave a complete description of the local dynamics in the case
of semi-parabolic order $2$. Hakim extended the result to multiple
attracting multipliers in \cite{Hakim1994AttractingDomainsforSemiAttractiveTransformationsofCp}.
Finally, Rivi in \cite{Rivi1998LocalBehaviourofDiscreteDynamicalSystems}
and \cite{Rivi2001ParabolicManifoldsforSemiAttractiveHolomorphicGerms}
and Rong in \cite{Rong2011ParabolicmanifoldsforsemiattractiveanalytictransformationsofmathbfCn}
established parabolic manifolds for general parabolic-attracting germs
with non-degenerate characteristic directions in generalisation of
Theorem~(\ref{thm:Hakim2}).
\begin{lem}[{\cite[Prop.~2.1]{Rivi2001ParabolicManifoldsforSemiAttractiveHolomorphicGerms}}]
\label{lem:SemAttrSetupCoord}Let $F\in\End(\mathbb{C}^{d},0)$ with
multipliers $1$, of matching geometric and algebraic multiplicity
$m$, and $\lambda_{m+1},\ldots,\lambda_{d}$ attracting. Then we
can write $F$ in local coordinates $(x,y)\in\mathbb{C}^{m}\times\mathbb{C}^{d-m}$
such that for $(x_{1},y_{1})=F(x,y)$, we have
\begin{equation}
\begin{aligned}x_{1} & =x+P_{2,y}(x)+P_{3,y}(x)+\cdots\\
y_{1} & =\Lambda y+G(y)+O(\norm x\cdot\norm{(x,y)}),
\end{aligned}
\label{eq:OrdForParAttr}
\end{equation}
where $G:(\mathbb{C}^{d-m},0)\to\mathbb{C}^{d-m}$ is a holomorphic
germ such that $G(y)=O(\norm y^{2})$, and, for each $j\ge2$, $P_{j,y}$
is a polynomial of order $j$ with holomorphic coefficients in $y$.
\end{lem}

\begin{proof}[Proof idea]
By the stable/centre manifold theorem~\ref{thm:StableCentreMnf}
there is an invariant strong stable manifold $W^{\mathrm{ss}}$ tangent
to the attracting eigenspaces of $dF_{0}$. In local coordinates,
we may assume that $W^{\mathrm{ss}}=\{x=0\}$. After diagonalising
$dF_{0}$, the formal elimination of terms $x^{e_{j}}y^{\beta}$ in
$x_{1}$ and $x^{\alpha}$ in $y_{1}$ is ensured by Poincaré-Dulac
theory. The attracting multipliers then ensure convergence of the
formal series via majorant series similar to the attracting case due
to Poincaré.
\end{proof}

\begin{defn}
\label{def:SemiParOrderRiv}Let $F\in\End(\mathbb{C}^{d},0)$ be parabolic-attracting
and in the form (\ref{eq:OrdForParAttr}). Then
\[
\inf\{j\ge2\mid P_{j,0}\equiv0\}
\]
is the \emph{\index{semi-parabolic order@\emph{semi-parabolic order}}semi-parabolic
order} of $F$.
\end{defn}

\begin{rem}
The semi-parabolic order is a holomorphic invariant for $F$. $F$
is of infinite semi-parabolic order, if and only if $F$ has a manifold
of fixed points through $0$ of dimension $m$.
\end{rem}

\begin{thm}[Ueda \cite{Ueda1986LocalstructureofanalytictransformationsoftwocomplexvariablesI,Ueda1991LocalstructureofanalytictransformationsoftwocomplexvariablesII}]
\label{thm:SemiAttrUeda}Let $F\in\Aut(\mathbb{C}^{2},0)$ be parabolic-attracting
with multipliers $(1,\lambda)$, $0<|\lambda|<1$ and of semi-parabolic
order $2$. Then:
\begin{enumerate}
\item We can write $F$ in local coordinates $(x,y)\in\mathbb{C}^{2}$ such
that for $(x_{1},y_{1})=F(x,y)$, we have 
\begin{align*}
x_{1} & =x+x^{2}+O(x^{3})\\
y_{1} & =\lambda y+xO(\norm{(x,y)}).
\end{align*}
\item There exists $\rho>0$ such that, for every $\theta\in(0,\pi)$, there
exists $R(\theta)>0$ such that, for any $R\ge R(\theta)$, the set
\[
K:=K(\rho,\theta,R):=\{(x,y)\in\mathbb{C}^{2}\mid|\arg(1/x)-R|<\theta,|y|<\rho\}
\]
is a parabolic domain for $F$ and all orbits in $K$ converge to
$0$ locally uniformly. 
\item The strong stable manifold $W^{\mathrm{ss}}=\{x=0\}$ is contained
in the boundary $\partial K$ near $0$, the realm of attraction is
\[
A_{F}(0)=W^{\mathrm{ss}}\cup\bigcup_{n\in\mathbb{N}}F^{\circ(-n)}(K).
\]
and realm of repulsion $C=A_{F^{-1}}(0)$ is a parabolic curve for
$F^{-1}$ tangent to $[1:0]$ with non-empty intersection $C\cap(\bigcup_{n\in\mathbb{N}}F^{\circ(-n)}(K))$.
In particular, for all other points in $\mathbb{C}$, no forward or
backward orbit converges to $0$.
\item There exists an injective holomorphic function $\varphi:K\to\mathbb{C}^{2}$
conjugating $F$ to the translation $(z,w)\mapsto(z+1,w)$. $\varphi$
extends holomorphically to $\bigcup_{n\in\mathbb{N}}F^{\circ(-n)}(K)$
still satisfying $\varphi+e_{1}=\varphi\circ F$.
\item If $F\in\Aut(\mathbb{C}^{2},0)$ is an automorphism, then $\varphi:\bigcup_{n\in\mathbb{N}}F^{\circ(-n)}(K)\to\mathbb{C}^{2}$
is a biholomorphism and $W^{\mathrm{ss}}$ and $C=A_{F^{-1}}(0)$
are each biholomorphic to $\mathbb{C}$.
\end{enumerate}
\end{thm}

Canille Martins \cite{CanilleMartins1992HolomorphicFlowsinC30withResonances}
showed, that, topologically, the attracting and parabolic dynamics
can be completely decoupled, leading to a simple classification:
\begin{thm}[Canille Martins \cite{CanilleMartins1992HolomorphicFlowsinC30withResonances}]
\label{thm:TopClassSemAttr}Let $F\in\Aut(\mathbb{C}^{2},0)$ be
parabolic-attracting of semi-parabolic order $k\le+\infty$. Then
$F$ is topologically conjugate to the map $(z,w)\mapsto(z+z^{k},\nicefrac{1}{2}w)$
if $k$ is finite and to the map $(z,w)\mapsto(z,\nicefrac{1}{2}w)$
if $k$ is infinite.
\end{thm}

Hakim's extension of Theorem~\ref{thm:SemiAttrUeda} in \cite{Hakim1994AttractingDomainsforSemiAttractiveTransformationsofCp}
to non-invertible germs in higher dimensions and higher orders gives
less details of the local dynamics (though without obvious obstructions):
\begin{thm}[Hakim \cite{Hakim1994AttractingDomainsforSemiAttractiveTransformationsofCp}]
\label{thm:SemAttrHakim}Let $F\in\End(\mathbb{C}^{d},0)$ be parabolic-attracting
with multiplier $1$ of algebraic multiplicity $1$, the remaining
multipliers attracting, and finite semi-parabolic order $k+1$. Then
$0$ is an isolated fixed point of $F$ and there exist $k$ disjoint
parabolic domains for $F$ at $0$. 

If, moreover, $F\in\Aut(\mathbb{C}^{2})$ is a global automorphism,
then, on each of these parabolic domains $P$, there exists a biholomorphism
$\varphi:P\to\mathbb{C}^{2}$ conjugating $F$ to the translation
$(z,w)\mapsto(z+1,w)$.
\end{thm}

\subsection{Several parabolic multipliers}

To tackle the case of multiple parabolic multipliers, Rivi extended
the language of Hakim's parabolic results to the parabolic-attracting
case:
\begin{defn}
Let $F\in\End(\mathbb{C}^{d},0)$ be parabolic-attracting of finite
semi-parabolic order $k$ in the form (\ref{eq:OrdForParAttr}). A
\emph{\index{characteristic direction!parabolic-attracting}characteristic
direction} for $F$ is a direction $[v:0]\in\mathbb{P}^{d-1}$ with
$v\in\mathbb{C}^{m}$ such that $P_{k}(v)=\gamma v$ for some $\gamma\in\mathbb{C}$.
If $\gamma=0$, we call $[v:0]$ \emph{degenerate}\index{degenerate characteristic direction!parabolic-attracting},
otherwise \emph{non-degenerate}\index{non-degenerate characteristic direction!parabolic-attracting}.
\end{defn}

\begin{defn}
Let $F\in\End(\mathbb{C}^{d},0)$ be parabolic-attracting of finite
semi-parabolic order $k$ in the form (\ref{eq:OrdForParAttr}), with
a non-degenerate characteristic direction $[v:0]\in\mathbb{P}^{d-1}$.
Then $P_{k}$ induces a holomorphic self-map of $\mathbb{P}^{m-1}$
and the eigenvalues of the linear operator 
\[
\frac{1}{k-1}((dP_{k})_{[v]}-\id):T_{[v]}\mathbb{P}^{m-1}\to T_{[v]}\mathbb{P}^{m-1}
\]
are called the \emph{directors\index{director@\emph{director!parabolic-attracting}}}
of $[v]$. If all directors of $[v]$ have positive real part, we
call the non-degenerate characteristic direction $[v]$ \emph{\index{attracting non-degenerate characteristic 
direction@\emph{attracting non-degenerate characteristic direction!parabolic-attracting}}attracting}.

The final result of this section was proved by Rivi \cite{Rivi2001ParabolicManifoldsforSemiAttractiveHolomorphicGerms}
and Rong \cite{Rong2011ParabolicmanifoldsforsemiattractiveanalytictransformationsofmathbfCn}
using the techniques of \cite{Hakim1998AnalytictransformationsofmathbbCp0tangenttotheidentity,HakimTransformationstangenttotheidentityStablepiecesofmanifolds}.
\end{defn}

\begin{thm}[Rivi \cite{Rivi2001ParabolicManifoldsforSemiAttractiveHolomorphicGerms},
Rong \cite{Rong2011ParabolicmanifoldsforsemiattractiveanalytictransformationsofmathbfCn}]
Let $F\in\End(\mathbb{C}^{d},0)$ be parabolic-attractingof finite
semi-parabolic order $k+1\ge2$ of the form (\ref{eq:OrdForParAttr})
and with a non-degenerate characteristic direction $[v]\in\mathbb{P}^{d-1}$
with directors $\alpha_{1},\ldots,\alpha_{m-1}\in\mathbb{C}$. If
there exist $c>0$ and $q\in\mathbb{N}$ such that $\Re\alpha_{1},\ldots,\Re\alpha_{q}>c$
and $\Re\alpha_{q+1},\ldots,\Re\alpha_{m-1}<c$. Then there exist
a parabolic flower of $k$ parabolic manifolds $P_{1},\ldots,P_{k}$
of dimension $(d-m)+q+1$, each tangent to the direct sum of $[v]$
and the generalised eigenspaces of the attracting multipliers and
the directors $\alpha_{1},\ldots,\alpha_{m}$, and such that all contained
orbits converge to $0$ tangentially to $[v]$.
\end{thm}

\begin{rem}
In \cite{Rivi1998LocalBehaviourofDiscreteDynamicalSystems} and \cite{Rivi2001ParabolicManifoldsforSemiAttractiveHolomorphicGerms},
Rivi further adapted Abate's diagonalisation theorem~\ref{thm:AbateDiagonalisation}
to diagonalise Jordan blocks in the parabolic part of a parabolic-attracting
germ and construct examples of parabolic manifolds for parabolic-attracting
germs with general non-diagonalisable linear part.
\end{rem}

In the presence of multiple parabolic multipliers, a description of
the dynamics in a full neighbourhood of the fixed point is strictly
dependent on solving the same problem for parabolic germs.

\subsection{Mixed hyperbolic multipliers}

Recently, Lyubich, Radu, and Tanase gave an alternative proof of Theorem~\ref{thm:SemAttrHakim}
in \cite{LyubichRaduTanase2016HedgehogsinHigherDimensionsandTheirApplications}
via centre manifold reduction, that extends to the case of mixed hyperbolic
multipliers. An important ingredient is the quasiconformal conjugation
of $F$ to a holomorphic map on the (non-analytic) centre manifold:
\begin{thm}[\cite{LyubichRaduTanase2016HedgehogsinHigherDimensionsandTheirApplications}]
\label{thm:CentManRedLyubich}Let $F\in\Aut(\mathbb{C}^{d},0)$ with
one neutral multiplier $\lambda$ and hyperbolic multipliers otherwise.
Let $W^{\mathrm{c}}$ be a $\mathcal{C}^{1}$-smooth centre manifold
as in Theorem~\ref{thm:StableCentreMnf}. Then $F:(W^{\mathrm{c}},0)\to(W^{\mathrm{c}},0)$
is quasiconformally conjugate to a holomorphic germ $h\in\End(\mathbb{C},0)$
of the form $h(z)=\lambda z+O(z^{2})$. Moreover, 
\begin{enumerate}
\item For small enough neighbourhoods $W\subseteq W^{\mathrm{c}}$ of $0$
in $W^{\mathrm{c}}$, the conjugacy is holomorphic on the interior
of the set $\Sigma_{F}(W)\cap\Sigma_{F^{-1}}(W)$ of points whose
forward and backward orbits stay in $W$.
\item If all hyperbolic multipliers of $F$ are attracting, the conjugacy
is holomorphic on the interior of the unstable set $\Sigma_{F^{-1}}(W)$
(as a subset of $W^{\mathrm{c}}$) for small enough neighbourhood
$W\subseteq W^{\mathrm{c}}$ of $0$ in $W^{\mathrm{c}}$.
\end{enumerate}
\end{thm}

This facilitates the examination of the dynamics on the centre manifold
with classical one-dimensional techniques to obtain parabolic manifolds
with local Fatou coordinates. Though, a priori, these are only $\mathcal{C}^{1}$,
the fact that the petals lie in the holomorphic part of $W^{\mathrm{c}}$
(see \cite{Abate2001AnIntroductiontoHyperbolicDynamicalSystems}),
ensures holomorphicity:
\begin{thm}[\cite{LyubichRaduTanase2016HedgehogsinHigherDimensionsandTheirApplications}]
Let $F\in\Aut(\mathbb{C}^{d},0)$ be parabolic-attracting with multiplier
$1$ of algebraic multiplicity $1$, one attracting multiplier $\lambda$,
the remaining multipliers repelling, and finite semi-parabolic order
$k+1$. Then $0$ is an isolated fixed point of $F$ and there exist
$k$ disjoint parabolic manifolds of dimension $2$ for $F$ at $0$. 

If, moreover, $F\in\Aut(\mathbb{C}^{2})$ is a global automorphism,
then, on each of these parabolic manifolds $P$, there exists a biholomorphism
$\varphi:P\to\mathbb{C}^{2}$ conjugating $F$ to the translation
$(z,w)\mapsto(z+1,w)$.
\end{thm}

Using techniques from \cite{PalisTakens1977TopologicalEquivalenceofNormallyHyperbolicDynamicalSystems},
in \cite{Giuseppe2006TopologicalClassificationofHolomorphicSemiHyperbolicGermsinquasiAbsenceofResonances},
Di Giuseppe generalised the proof of Theorem~\ref{thm:TopClassSemAttr}
from \cite{CanilleMartins1992HolomorphicFlowsinC30withResonances}
to achieve a topological classification with mixed hyperbolic multipliers
with finite resonances.
\begin{thm}[\cite{Giuseppe2006TopologicalClassificationofHolomorphicSemiHyperbolicGermsinquasiAbsenceofResonances}]
Let $F\in\Aut(\mathbb{C}^{1+d},0)$ with multipliers $\mu,\lambda_{1},\ldots,\lambda_{d}$,
where $\mu$ is parabolic and $\lambda_{1},\ldots,\lambda_{d}$ are
hyperbolic with only finitely many resonances among the hyperbolic
multipliers $\lambda_{1},\ldots,\lambda_{d}$. Then $F$ is topologically
conjugate to $dF_{0}$ or to the map 
\[
(x,y)\mapsto(\mu x+x^{k+1},\Lambda y),
\]
where $(x,y)\in\mathbb{C}\times\mathbb{C}^{d}$, $\Lambda=\diag(\lambda_{1},\ldots,\lambda_{d})$
and $k+1$ is the semi-parabolic order of $F$.
\end{thm}

\begin{rem}
Here, the semi-parabolic order of $F$ can be defined as in Definition~\ref{def:SemiParOrderRiv}
if $\mu=1$: Since $\lambda_{1},\ldots,\lambda_{d}$ have only finitely
many resonances, we have $\lambda^{\alpha}\neq1$ for all $\alpha\in\mathbb{N}^{d}\backslash\{0\}$.
Hence we can eliminate resonant terms $y^{\alpha}$ in $x_{1}$ until
$F$ has the form (\ref{eq:OrdForParAttr}) and define the semi-parabolic
order of $F$ as $\inf\{j\ge2\mid P_{j,0}\equiv0\}$. If $\mu$ is
a primitive $q$-th root of unity, then we define the semi-parabolic
order\index{semi-parabolic order} of $F$ as that of $F^{\circ q}$.
In this case, $k$ is a multiple of $q$.
\end{rem}

\section{\label{sec:Multi-resonant-cases}$m$-resonant fixed points}

\selectlanguage{british}
\global\long\def\Pow{\m{Pow}}%
\global\long\def\Res{\m{Res}}%
\global\long\def\Jac{\m{Jac}}%
\global\long\def\Transp{\mathsf{T}}%

\subsection{One-resonance}

Until now, resonances have occurred as a formal obstruction to linearisation.
In \cite{BracciZaitsev2013Dynamicsofoneresonantbiholomorphisms},
Bracci and Zaitsev exploit a one-dimensional family of resonances
to find a simple normal form as well as conditions for the existence
of parabolic basins. Their main tool is a non-linear projection to
a lower dimensional germ with parabolic dynamics on certain subsets.
This non-linearity is a fundamental difference to many other results
that use linear projections to lower dimensional dynamics.
\begin{defn}
\label{def:oneRes}Let $F\in\Pow(\mathbb{C}^{d},0)$ be a germ with
multipliers $\lambda_{1},\ldots,\lambda_{d}$. Then $F$ is called
\emph{(partially) one-resonant\index{one-resonant germ}} with respect
to $\lambda_{1},\ldots,\lambda_{r}$, $1\le r\le d$, with \emph{\index{generator@\emph{generator}}generator}
$0\neq\alpha\in\mathbb{N}^{r}\times\{0\}^{d-r}$, if we have $\lambda_{j}=\lambda^{\beta}$
for $1\le j\le r$ and $\beta\in\mathbb{N}^{d}$, if and only if $\beta=k\alpha+e_{j}$
for some $k\in\mathbb{N}_{>0}$, i.e.\ the resonances for $F$ in
the first $r$ components are precisely of the form 
\begin{equation}
\lambda_{j}=\lambda^{k\alpha+e_{j}},\quad k\in\mathbb{N}_{>0},1\le j\le r.\label{eq:oneResonance}
\end{equation}
$F$ is (\emph{fully}) \emph{one-resonant}, if it is one-resonant
with respect to all its multipliers, that is $r=d$.
\end{defn}

\begin{rem}
One-resonance of $F$ with respect to $\lambda_{1},\ldots,\lambda_{r}$
immediately implies that the first $r$ multipliers $\lambda_{1},\ldots,\lambda_{r}$
have multiplicity $1$, that is $\lambda_{j}\neq\lambda_{s}$ for
$j\le r$, $s\le n$ with $j\neq s$.
\end{rem}

If $F\in\Pow(\mathbb{C}^{d},0)$ is one-resonant of index $\alpha$
with respect to $\lambda_{1},\ldots,\lambda_{m}$, then the by the
Poincaré-Dulac Theorem~\ref{thm:PoincareDulac}, $F$ is either
formally linearisable in the first $m$ variables or there exist
local coordinates at $0$ such that 
\begin{equation}
F_{j}(z)=\lambda_{j}z_{j}(1+a_{j}z^{\alpha k})+O(\norm z^{(k+1)|\alpha|+1})\quad\text{for }j=1,\ldots,m,\label{eq:OneResPDNF}
\end{equation}
with $a=(a_{1},\ldots,a_{m})\neq0$. Then by one-resonance, we have
\begin{align*}
(F(z))^{\alpha} & =z^{\alpha}\prod_{j=1}^{m}(1+a_{j}z^{k\alpha})^{\alpha_{j}}+O(\norm z^{(k+2)|\alpha|})
\end{align*}
and with the projection $u=\pi(z)=z^{\alpha}$, $u_{1}=(F(z))^{\alpha}$,
we get 
\begin{equation}
u_{1}=\underbrace{u(1+A\cdot u^{k})+h(u)}_{=:\Phi(u)}+O(\norm z^{(k+2)|\alpha|})\label{eq:ParabShadow}
\end{equation}
with $A:=\sum_{j=1}^{d}\alpha_{j}a_{j}$ and $h(u)=O(u^{k+2})$. In
particular, $\Phi\in\End(\mathbb{C},0)$ is a parabolic germ in one
variable, called the \emph{parabolic shadow\index{parabolic shadow@\emph{parabolic shadow}}}
of $F$. If moreover $A\neq0$, $F$ is called \emph{non-degenerate
one-resonant\index{non-degenerate one-resonant germ@\emph{non-degenerate one-resonant germ}}}
of \emph{weighted order\index{weighted order@\emph{weighted order}}}
$k$. 
\begin{rem}
\label{rem:BZ3.2OneResOrderInvariant}Recall from Proposition~\ref{prop:ConjPreservingPoiDul},
that a conjugating germ preserving the form (\ref{eq:OneResPDNF}),
has only resonant terms commuting with $dF_{0}$, and hence leaves
$k$ invariant and rescales the vector $(a_{1},\ldots,a_{m})$ by
a non-zero constant. Hence, one-resonant order and non-degeneracy
are well-defined invariants for $F$. 
\end{rem}

After multiplying all variables with a constant, we may assume $A=-\frac{1}{k}$,
and Bracci and Zaitsev establish a unique formal normal form modelled
after the one-dimensional parabolic case (Proposition~\ref{prop:1DparaFormalClass}):
\begin{thm}[Bracci, Zaitsev \cite{BracciZaitsev2013Dynamicsofoneresonantbiholomorphisms}]
\label{thm:BracciZaitsevNormalForm}Let $F\in\Pow(\mathbb{C}^{d},0)$
be non-degenerate one-resonant of index $\alpha$ with respect to
$\lambda_{1},\ldots,\lambda_{m}$. Then there exist unique $k\in\mathbb{N}$
and $\mu,a_{1},\ldots,a_{m}\in\mathbb{C}$ such that $A=\sum_{j=1}^{m}\alpha_{j}a_{j}=-\frac{1}{k}$
and $F$ is formally conjugated to a germ $\hat{F}\in\Pow(\mathbb{C}^{d},0)$
with 
\begin{equation}
\hat{F}_{j}(z)=\lambda_{j}z_{j}(1+a_{j}z^{k\alpha}+\mu\alpha_{j}|\lambda_{j}|^{-2}z^{2k\alpha})\quad\text{for }j=1,\ldots,m,\label{eq:BZNF}
\end{equation}
and the components $\hat{F}_{j}$ for $m<j\le d$ contain only resonant
monomials.
\end{thm}

With $A=-\frac{1}{k}$ in (\ref{eq:ParabShadow}), the Leau-Fatou
Theorem~\ref{thm:LeauFatou} for $\Phi$ yields a parabolic flower
of $k$ attracting sectors
\[
S_{h}(R,\theta):=\pbrace[\big]{u\in\mathbb{C}\mid{\abs[\big]{u^{k}-\tfrac{1}{2R}}}<\tfrac{1}{2R},{\abs[\big]{\arg(u)-\tfrac{2\pi h}{k}}}<\theta}
\]
for $h=0,\ldots,k-1$ and suitable $R>0$ and $\theta\in(0,\nicefrac{\pi}{2k})$.
The sequence $\{u_{n}\}_{n}=\{\pi(F^{\circ n}(z))\}_{n}$ will inherit
the parabolic dynamics of $\Phi$, if $z$ can be controlled in terms
of $u=\pi(z)$, say in the set 
\[
W(\beta):=\{z\in\mathbb{C}^{d}\mid|z^{j}|<|\pi(z)|^{\beta}\text{ for }j\le d\}.
\]
for some $\beta\in(0,1/|\alpha|)$. In particular we need $z_{n}^{j}\to0$
for each $j\le d$. This can happen at first order, $|\lambda_{j}|<1$
or at second order:
\begin{defn}
Let $F\in\End(\mathbb{C}^{d},0)$ be non-degenerate one-resonant with
respect to $\lambda_{1},\ldots,\lambda_{m}$ of the form (\ref{eq:OneResPDNF})
with $A=\sum_{j=1}^{m}\alpha_{j}a_{j}=-\frac{1}{k}$. Then we call
$F$ \emph{parabolically attracting} with respect to $\lambda_{1},\ldots,\lambda_{m}$,
if $|\lambda_{j}|=1$ and $\Re(a_{j})<0$ for $j=1,\ldots,m$.
\end{defn}

In this case, since $\{u_{n}\}_{n}$ tends to $0$ along the real
direction $v=1$, the factor $(1+a_{j}u)$ in (\ref{eq:OneResPDNF})
will be contracting (compare Definition~\ref{def:DirectorsPara}
of directors in the parabolic case) and Bracci and Zaitsev \cite{BracciZaitsev2013Dynamicsofoneresonantbiholomorphisms}
showed the existence of attracting parabolic domains. With Raissy
\cite{BracciRaissyZaitsev2013Dynamicsofmultiresonantbiholomorphisms}
they also constructed Fatou coordinates on these domains:
\begin{thm}[Bracci, Raissy, Zaitsev \cite{BracciZaitsev2013Dynamicsofoneresonantbiholomorphisms},
\cite{BracciRaissyZaitsev2013Dynamicsofmultiresonantbiholomorphisms}]
\label{thm:BracciZaitsevDynamics}Let $F\in\End(\mathbb{C}^{d},0)$
be non-degenerate one-resonant with generator $\alpha$ of order $k$
and parabolically attracting with respect to $\lambda_{1},\ldots,\lambda_{m}$.
Let moreover $|\lambda_{j}|<1$ for $j>m$. Then there exist $\beta\in(0,1/|\alpha|)$,
$R>0$ and $\theta\in(0,\nicefrac{\pi}{2k})$ such that 
\[
B_{h}:=B_{h}(R,\theta,\beta):=\{z\in W(\beta)\mid\pi(z)\in S_{h}(R,\theta)\}
\]
for $h=0,\ldots,k-1$ are disjoint open $F$-invariant and attracting
to $0$ on their boundaries. Moreover, there exist holomorphic maps
$\psi_{h}:B_{h}\to\mathbb{C}$ such that $\psi_{h}\circ F=\psi_{h}+1$.

\end{thm}

More precisely, we observed in \cite{Reppekus2019PeriodiccyclesofattractingFatoucomponentsoftypemathbbCtimesmathbbCd1inautomorphismsofmathbbCd}
that each $B_{h}$ is connected, if and only if $\gcd(\alpha)=1$.
Generally, each $B_{h}$ consists of $\ell=\gcd(\alpha)$ components
that are cyclically permuted by $F$ and these components correspond
to the invariant attracting basins of $F^{\circ\ell}$, which is one-resonant
of generator $\alpha_{0}:=\alpha/\ell$ ($\gcd(\alpha_{0})=1$), weighted
order $k\cdot\ell$ and parabolically attracting. This is the analogue
of the permuting action on the petals of the Leau-Fatou flower when
the multiplier is a root of unity in dimension one as described in
Corollary~\ref{cor:ParabFlowerPermutation}.  
\begin{rem}[Shape and arrangement]
\label{rem:BZbasinShape}The shape and topology of the local basins
becomes apparent in polar coordinates. Assume $\gcd(\alpha)=1$ 
and $0<R<1$, so we can ignore the bounds in terms of $R$. For $z\in\mathbb{C}^{d}$
take componentwise argument $\arg(z)\in\mathbb{T}^{d}=(S^{1})^{d}$
and modulus $|z|\in\mathbb{R}_{+}^{d}$, where we identify $S^{1}=\mathbb{R}/(2\pi\mathbb{Z})$.
Then the modulus component of all $B_{h}$ is 
\[
\{x\in\mathbb{R}_{+}^{d}\mid|x^{j}|<|x^{\alpha}|^{\beta}\},
\]
which is simply connected and contains the $x^{j}$-axis precisely
when $\alpha_{j}=0$. The argument component of $B_{h}$ is the $\theta$-neighbourhood
of the hypersurface 
\[
\alpha_{1}\arg(z^{1})+\cdots+\alpha_{m}\arg(z^{m})=2\pi h/k.
\]
 that winds around the torus $\mathbb{T}^{d}$ according to the non-zero
entries of $\alpha$. The basin $B_{h}$ hence winds around the $z_{j}$-axes
for $\alpha_{j}\neq0$ and contains the $z_{j}$-axes (near $0$)
for $\alpha_{j}=0$. In particular, $B_{h}$ has homotopy type $(S_{1})^{\#\{\alpha_{j}\neq0\}-1}$.
Figure~\ref{fig:Argument-components-attracting} shows this decomposition
for $d=2$, $\alpha=(1,1)$, $k=2$, where the argument component
takes the shape of a ``ribbon'' winding around the torus $\mathbb{T}^{2}$
$\alpha_{1}=1$-time in one component while winding $\alpha_{2}=1$-time
in the other.

For general $R>0$, the basins are truncated, but remain the same
near the origin and preserve their homotopy type.
\end{rem}

\begin{figure}[h]
$\begin{array}{c}\includegraphics[width=0.4\columnwidth]{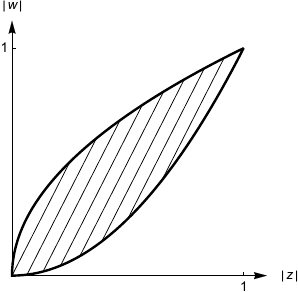}\end{array}\bigtimes\quad{}\begin{array}{c}\includegraphics[bb=0bp 0bp 263bp 191bp,width=0.5\columnwidth]{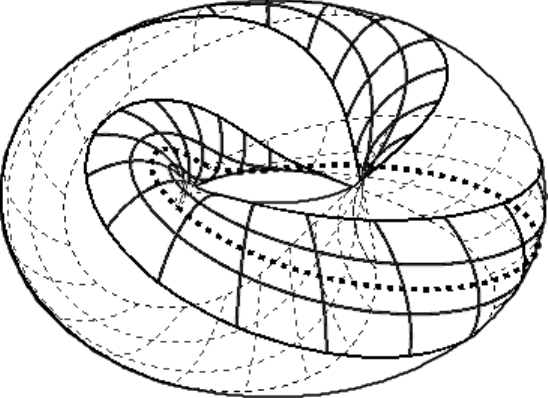}\end{array}$

\caption{\label{fig:Argument-components-attracting}Decomposition in modulus
and argument components for two local basins $B_{0}$ and $B_{1}$
with central curve $\arg z+\arg w\equiv0$ in $B_{0}$}
\end{figure}

\lyxaddress{}

If $F$ is in (or conjugate to) Poincaré-Dulac normal form, then the
tail in (\ref{eq:ParabShadow}) depends only on $u$, and we get parabolic
dynamics for $u$ in a full neighbourhood of the origin $0$ and the
dynamics in $z_{1},\ldots,z_{r}$ are described by multiplication
with a function $1+O(u)$ in $u$ (see \cite{BracciZaitsev2013Dynamicsofoneresonantbiholomorphisms}).
In particular, on the $F$-invariant set $\{u=0\}$, $F$ is a rotation
in the first $r$ components. If $F$ is moreover fully one-resonant
and parabolically attracting, then so is $F^{-1}$ and the dynamics
on a full neighbourhood are described as follows:
\begin{thm}[Bracci, Raissy, Zaitsev \cite{BracciRaissyZaitsev2013Dynamicsofmultiresonantbiholomorphisms}]
Let $F\in\Aut(\mathbb{C}^{d},0)$ be fully one-resonant with generator
$\alpha\in\mathbb{N}^{d}$ and holomorphically conjugate to a Poincaré-Dulac
normal form. If $F$ is neutral ($|\lambda_{j}|=1$ for $j\le d$),
non-degenerate, and parabolically attracting, then there exists a
neighbourhood $U\subseteq\mathbb{C}^{d}$ of the origin that is a
disjoint union of 
\begin{enumerate}
\item for each $j\in\{1,\ldots,d\}$ with $\alpha_{j}\neq0$, a Siegel hypersurface\index{Siegel hypersurface}
$M_{j}\subseteq U$ tangent to $\{z_{j}=0\}$ through $0$, i.e.\ a
hypersurface such that $F|_{M_{j}}:M_{j}\to M_{j}$ acts as (conjugate
of) a rotation on $M_{j}$.
\item for each $h=0,\ldots,k-1$ the attracting parabolic basin $\Omega_{h}:=\bigcup_{n\in\mathbb{N}}F|_{U}^{\circ(-n)}(B_{h})$
in $U$ corresponding to $B_{h}$ from Theorem~\ref{thm:BracciZaitsevDynamics}.
\item the analogous $k$ local basins in $U$ for $F^{-1}$ (repelling basins
for $F$).
\end{enumerate}
\end{thm}

In particular, this shows that in this case the basins $\Omega_{h}$
are (unions of) Fatou components. For general one-resonant germs,
the dynamics outside of the parabolic domains from Theorem~\ref{thm:BracciZaitsevDynamics}
have not been understood. Without Siegel hypersurfaces we can expect
dynamics at least as complicated as one-dimensional Hedgehog dynamics
(see Theroem~\ref{thm:PerezMarcoCremerHedgehog}). It is not known
in general weather the basins $\Omega_{h}$ are Fatou components.

Bracci, Raissy and Stensønes \cite{BracciRaissyStensonesAutomorphismsofmathbbCkwithaninvariantnonrecurrentattractingFatoucomponentbiholomorphictomathbbCtimesmathbbCk1}
study in detail the dynamics of one-resonant germs with generator
$(1,\ldots,1)$ of the form (\ref{eq:ReppGerms}) with $k=1$. Given
the existence of Siegel hypersurfaces tangent to each axis, they apply
a novel Kobayashi distance argument to show that the basins $\Omega_{h}$
are in fact Fatou components. Under the partial Brjuno condition on
$\{\lambda_{1},\ldots,\lambda_{d}\}\backslash\{\lambda_{j}\}$ for
each $j=1,\ldots,d$, they use Pöschel's theorem~\ref{thm:PoeschelBrjuno}
to ensure the existence of these Siegel hypersurfaces. In \cite{Reppekus2019PeriodiccyclesofattractingFatoucomponentsoftypemathbbCtimesmathbbCd1inautomorphismsofmathbbCd},
we extend their considerations to the case $k>1$ and observe that
under the same partial Brjuno condition, for any $\ell\in\mathbb{N}$
Theorem~\ref{thm:IteratedEliminationSimple} yields coordinates such
that the tail on the parabolic shadow in (\ref{eq:ParabShadow}) is
of order $O(u^{\ell})$ on a full neighbourhood of the origin. This
allows us to classify the stable orbits on some neighbourhood of the
origin:
\begin{thm}[Reppekus \cite{Reppekus2019PeriodiccyclesofattractingFatoucomponentsoftypemathbbCtimesmathbbCd1inautomorphismsofmathbbCd}]
\label{thm:FatouCCstarMulti}Let $F\in\Aut(\mathbb{C}^{d},0)$ be
of the form 
\begin{equation}
F(z^{1},\ldots,z^{d})=(\lambda_{1}z^{1},\ldots,\lambda_{d}z^{d})\paren[\Big]{1-\frac{(z^{1}\cdots z^{d})^{k}}{kd}}+O(\norm z^{l}),\label{eq:ReppGerms}
\end{equation}
with $|\lambda_{j}|=1$ for all $1\le j\le d$, one-resonant with
generator $(1,\ldots,1)$ and $l>2kd+1$. If each subset $\{\lambda_{1},\ldots,\lambda_{d}\}\backslash\{\lambda_{j}\}$,
$j=1,\ldots,d$ satisfies the Brjuno condition \ref{eq:BrjunoCond},
then there exists a neighbourhood $U\subseteq\mathbb{C}^{d}$ of the
origin such that all stable orbits in $U$ are contained in one of:
\begin{enumerate}
\item a Siegel hypersurface\index{Siegel hypersurface} $M_{j}\subseteq U$
tangent to $\{z_{j}=0\}$ through $0$ for $j\in\{1,\ldots,d\}$.
\item the attracting basins $\bigcup_{n\in\mathbb{N}}F|_{U}^{\circ(-n)}(B_{h})$,
in $U$ corresponding to $B_{h}$ from Theorem~\ref{thm:BracciZaitsevDynamics}
for $h\in\{0,\ldots,k-1\}$.
\end{enumerate}
\end{thm}

Moreover, for each $h=0,\ldots,k-1$, there exists a holomorphic map
$\phi:\Omega_{h}\to\mathbb{C}\times(\mathbb{C}^{*})^{d-1}$, injective
on $B_{h}$ such that $\phi\circ F=\phi+e_{1}$ and if $F\in\Aut(\mathbb{C}^{d})$,
then $\phi$ is biholomorphic.

Another interesting observation in \cite{Reppekus2019PeriodiccyclesofattractingFatoucomponentsoftypemathbbCtimesmathbbCd1inautomorphismsofmathbbCd}
is, that in this case, the convergence to $0$ of orbits in $B_{h}$
is never tangent to any one complex direction, since the arguments
accumulate on the whole central hyperplane from \ref{rem:BZbasinShape}.
In fact orbits accumulate at each complex direction outside the coordinate
hyperplanes. This is a consequence of the identical contracting factor
$(1-u^{k}/kd)$ causing the same shrinking rate in each component
of (\ref{eq:ReppGerms}). For varying contracting factors, one can
expect accumulation on proper complex subspaces or even a single complex
direction.

\subsection{Multi-Resonance}

In \cite{BracciRaissyZaitsev2013Dynamicsofmultiresonantbiholomorphisms},
Bracci, Raissy, and Zaitsev generalise this idea to the multi-resonant
case, where instead of a single generator a finite number of generators
induce a projection to a multi-dimensional parabolic shadow. 
\begin{defn}
For $m\in\mathbb{N}$, a germ $F\in\End(\mathbb{C}^{d},0)$ with multipliers
$\lambda_{1},\ldots,\lambda_{d}$ is (partially) \emph{$m$-resonant}
or \emph{multi-resonant} with respect to the first $r$ multipliers
$\lambda_{1},\ldots,\lambda_{r}$, where $1\le r\le d$, if there
exist $m$ $\mathbb{Q}$-linearly independent multi-indices $\alpha^{1},\ldots,\alpha^{m}\in\mathbb{N}^{r}\times\{0\}^{d-r}$,
such that the resonances for $F$ in the first $r$ components are
precisely of the form 
\[
\lambda_{j}=\lambda^{k_{1}\alpha^{1}+\cdots+k_{r}\alpha^{r}+e_{j}},\quad k_{1},\ldots,k_{r}\in\mathbb{N},1\le j\le r.
\]
\end{defn}

\begin{rem}
If $F\in\End(\mathbb{C}^{d},0)$ is multi-resonant with respect to
$\lambda_{1},\ldots,\lambda_{r}$, then the multipliers $\lambda_{1},\ldots,\lambda_{r}$
again have multiplicity $1$.
\end{rem}

If $F\in\Pow(\mathbb{C}^{d},0)$ is \emph{$m$-resonant} with respect
to the first $r$ multipliers $\lambda_{1},\ldots,\lambda_{r}$, $1\le r\le d$,
with generators $\alpha^{1},\ldots,\alpha^{m}\in\mathbb{N}^{r}\times\{0\}^{d-r}$,
then its Poincaré-Dulac normal forms (Theorem~\ref{thm:PoincareDulac})
take the form
\begin{equation}
G_{j}(z)=\lambda_{j}z_{j}(1+\sum_{k=k_{0}}^{\infty}g_{kj}(z^{\alpha^{1}},\ldots,z^{\alpha^{m}})),\quad1\le j\le r,\label{eq:PDNFmultires}
\end{equation}
with $g_{kj}$ homogeneous polynomials of order $k$ and $g_{k_{0}j}\not\equiv0$.
We call $k_{0}$ the \emph{weighted order} of $F$. Once again, the
dynamics of the normal form via the the projection $\pi(z)=(z^{\alpha^{1}},\ldots,z^{\alpha^{m}})$
depend on those of the \emph{parabolic shadow}\index{parabolic shadow}
$\Phi\in\End(\mathbb{C}^{m},0)$ such that $\pi\circ G=\Phi\circ\pi$
and $\Phi$ has the form:
\[
\Phi(u)=u+H_{k_{0}+1}(u)+O(\norm u^{k_{0}+2}).
\]
If $\Phi$ has attracting parabolic domains in \noun{$\mathbb{C}^{m}$}
and the behaviour of $u_{n}$ induces attraction through $g_{k_{0}j}$
in (\ref{eq:PDNFmultires}), we can again find attracting parabolic
domains for the original germ $F$ within their $\pi$-preimage. Bracci,
Raissy, and Zaitsev \cite{BracciRaissyZaitsev2013Dynamicsofmultiresonantbiholomorphisms}
apply this to Hakim's parabolic domains from Theorem~\ref{thm:Hakim2}:
\begin{thm}[Bracci, Raissy, Zaitsev \cite{BracciRaissyZaitsev2013Dynamicsofmultiresonantbiholomorphisms}]
\label{thm:BRZmultiRes}Let $F\in\Pow(\mathbb{C}^{d},0)$ be \emph{$m$-resonant}
with respect to the first $r$ multipliers $\lambda_{1},\ldots,\lambda_{r}$,
$|\lambda_{j}|=1$ for $j\le r$ and $|\lambda_{j}|<1$ for $r<j\le d$.
If a parabolic shadow $\Phi$ of $F$ has an attracting non-degenerate
characteristic direction $[v]\in\mathbb{P}^{m-1}$ in the sense of
Definition~\ref{def:DirectorsPara}, $H_{k_{0}+1}(v)=-\frac{1}{k_{0}}v$
and $\Re(g_{k_{0}j}(v))<0$ for $1\le j\le r$ in (\ref{eq:PDNFmultires}).
Then there exist $k_{0}$ open sets $B_{1},\ldots,B_{k_{0}}\subseteq\mathbb{C}^{d}$,
disjoint, open, $F$-invariant, and attracting to $0$ on their boundaries,
and holomorphic maps $\psi_{j}:B_{j}\to\mathbb{C}$ such that $\psi_{j}\circ F=\psi_{j}+1$
for $1\le j\le k_{0}$.
\end{thm}

In \cite{RaissyVivas2013Dynamicsoftworesonantbiholomorphisms}, Raissy
and Vivas perform a similar construction for parabolic shadows with
so-called irregular attracting non-degenerate characteristic directions
(in particular not fully attracting in the sense of Definition~\ref{def:DirectorsPara}).
They moreover show that parabolic attraction ($\Re(g_{k_{0}j}(v))<0$
for $1\le j\le r$) cannot be satisfied for a degenerate characteristic
direction $[v]\in\mathbb{P}^{m-1}$ of a parabolic shadow, so in this
case an analogue of Theorem~\ref{thm:BRZmultiRes} cannot be established
in the same way.

\section{Parabolic-elliptic fixed points}

\selectlanguage{british}

\global\long\def\Pow{\m{Pow}}%
\global\long\def\Res{\m{Res}}%
\global\long\def\Jac{\m{Jac}}%
\global\long\def\Transp{\mathsf{T}}%

The dynamics of parabolic-elliptic germs in dimension $2$ have been
studied by Bracci and Molino in \cite{BracciMolino2004ThedynamicsnearquasiparabolicfixedpointsofholomorphicdiffeomorphismsinmathbbC2}
and \cite{LopezHernanzRaissyRibonSanzSanchez2019StableManifoldsofTwoDimensionalBiholomorphismsAsymptotictoFormalCurves}
and in higher dimensions by Rong in \cite{Rong2008QuasiparabolicanalytictransformationsofmathbbCn,Rong2010QuasiparabolicanalytictransformationsofmathbfCnParabolicmanifolds}
combining ideas from \cite{BracciMolino2004ThedynamicsnearquasiparabolicfixedpointsofholomorphicdiffeomorphismsinmathbbC2}
and \cite{Hakim1998AnalytictransformationsofmathbbCp0tangenttotheidentity,HakimTransformationstangenttotheidentityStablepiecesofmanifolds}.
\begin{defn}
A (formal) germ $F\in\Pow(\mathbb{C}^{d},0)$ is called 
\begin{enumerate}
\item \emph{parabolic-elliptic\index{parabolic-elliptic germ}} if $F$
has both parabolic and elliptic multipliers and no other multipliers.
\item \emph{\index{semi-parabolic germ@\emph{semi-parabolic germ}}semi-parabolic},
if $F$ has at least one parabolic multiplier.
\end{enumerate}
\end{defn}

\begin{note}
The results from these papers were actually shown for the larger class
of \emph{quasi-parabolic\index{quasi-parabolic germ}} germs. However,
the cases not covered below are parabolic and the concluded parabolic
manifolds are the same as Hakim's from Theorems~\ref{thm:HakimEcale}
and \ref{thm:Hakim2}. To get the general results, replace all elliptic
multipliers by neutral multipliers $\lambda\neq1$.
\end{note}

\subsection{Parabolic curves}

To tackle the case of multiple parabolic multipliers, we extend the
language of Hakim's parabolic results to semi-parabolic germs following
\cite{Rong2008QuasiparabolicanalytictransformationsofmathbbCn,Rong2010QuasiparabolicanalytictransformationsofmathbfCnParabolicmanifolds}.

Let $F\in\End(\mathbb{C}^{d},0)$ with multipliers $1$ of multiplicity
$m$ and $\lambda_{m+1},\ldots,\lambda_{d}$ non-parabolic, and with
diagonal linear part $dF_{0}=\diag(1,\ldots,1,\lambda_{m+1},\ldots,\lambda_{d})$.
Then we can write $F$ in local coordinates $(w,z)\in\mathbb{C}^{m}\times\mathbb{C}^{d-m}$
such that for $(w_{1},z_{1})=F(w,z)$, we have 
\begin{equation}
\begin{aligned}w_{1} & =w+\sum_{k\in\mathbb{N}}p_{k}(w)+O(\norm z\cdot\norm{(w,z)})\\
z_{1} & =\Lambda z+\sum_{k\in\mathbb{N}}q_{k}(w)+O(\norm z\cdot\norm{(w,z)})
\end{aligned}
\label{eq:semi-paraForOrder}
\end{equation}
where $\Lambda=\diag(\lambda_{m+1},\ldots,\lambda_{d})$ and $p_{k}$
and $q_{k}$ are tuples of homogeneous polynomials of degree $k$
with respective dimensions $m$ and $d-m$.
\begin{rem}
If in (\ref{eq:semi-paraForOrder}), we have 
\[
\nu:=\inf\{k\ge2\mid p_{k}\not\equiv0\}=+\infty,
\]
then $\{z=0\}$ is a manifold of fixed points. Otherwise, since $q_{k}(w)$
contains only non-resonant terms, the Poincaré-Dulac theorem~\ref{thm:PoincareDulac}
allows us to holomorphically eliminate any finite number of the polynomials
$q_{k}$  in $z_{1}$ until 
\[
\mu:=\inf\{k\ge2\mid q_{k}\not\equiv0\}\ge\nu.
\]
Under a formal change of coordinates, we can even achieve $\mu=+\infty$.
\end{rem}

\begin{defn}
If $\nu\le\mu$ in (\ref{eq:semi-paraForOrder}), then we say that
$F$ is in \emph{ultra-resonant\index{ultra-resonant form@\emph{ultra-resonant form}}}
form and call $\nu$ the \emph{\index{semi-parabolic order@\emph{semi-parabolic order}}semi-parabolic
order} of $F$. If $\mu=+\infty$, then we say $F$ is in \emph{\index{asymptotic ultra-resonant form@\emph{asymptotic ultra-resonant form}}asymptotic
ultra-resonant form}.
\end{defn}

\begin{lem}[{\cite[Lem.~2.3]{Rong2008QuasiparabolicanalytictransformationsofmathbbCn}}]
Unless $\mu<\nu=+\infty$ in (\ref{eq:semi-paraForOrder}) (and hence
we have a manifold of fixed points), the semi-parabolic order is a
well-defined holomorphic invariant for $F$.
\end{lem}

\begin{rem}
In the parabolic-attracting case, we can always achieve $\mu=+\infty$
and the above definition is equivalent to Definition~\ref{def:SemiParOrderRiv}.
\end{rem}

Now we can once more define characteristic directions:
\begin{defn}
Let $F\in\End(\mathbb{C}^{d},0)$ be semi-parabolic of finite semi-parabolic
order $\nu<\infty$ in ultra-resonant form (\ref{eq:semi-paraForOrder}).
A \emph{\index{characteristic direction!semi-parabolic}characteristic
direction} for $F$ is a direction $[v:0]\in\mathbb{P}^{d-1}$ with
$v\in\mathbb{C}^{m}$ such that $p_{\nu}(v)=\gamma v$ for some $\gamma\in\mathbb{C}$
and $q_{\nu}(v)=0$. If $\gamma=0$, we call $[v:0]$ \emph{degenerate}\index{degenerate characteristic direction!semi-parabolic},
otherwise \emph{non-degenerate}\index{non-degenerate characteristic direction!semi-parabolic}.
\end{defn}

To ensure parabolic curves tangent to a non-degenerate characteristic
direction, we need the additional condition of dynamical separation
introduced by Bracci and Molino in \cite{BracciMolino2004ThedynamicsnearquasiparabolicfixedpointsofholomorphicdiffeomorphismsinmathbbC2}.

Let $F\in\End(\mathbb{C}^{d},0)$ be semi-parabolic of finite semi-parabolic
order $\nu<\infty$ with a non-degenerate characteristic direction
$[v]\in\mathbb{P}^{d-1}$. Then, as in the parabolic case, up to a
linear change of coordinates, we may assume that $[v]=[1:0:\cdots:0]$.
Eliminating all non-resonant terms $w^{\alpha}z^{\beta}$ in $z_{1}$
up to order $\nu$ in (\ref{eq:semi-paraForOrder}), and splitting
the coordinates $w=(x,y)\in\mathbb{C}\times\mathbb{C}^{m-1}$, we
can write $F$ as 
\begin{equation}
\begin{alignedat}{2}x_{1} & =x+p_{\nu}(x,y) &  & +P(x,y,z)+O(\nu+1),\\
y_{1} & =y+q_{\nu}(x,y) &  & +Q(x,y,z)+O(\nu+1),\\
z_{1} & =\Lambda z &  & +R(x,y,z)+O(\nu+1),
\end{alignedat}
\label{eq:2.9quasipar3var}
\end{equation}
where $p_{\nu}$ and $q_{\nu}$ are homogeneous of order $\nu$ and
$P,Q$ and $R$ only contain resonant terms of total order between
$2$ and $\nu$ and order in $z$ at least $1$.
\begin{defn}
Let $F\in\End(\mathbb{C}^{d},0)$ be quasi-parabolic of finite order
$\nu<\infty$ with characteristic direction $[v]=[1:0:\cdots:0]$
in the form (\ref{eq:2.9quasipar3var}). 
\begin{enumerate}
\item \emph{\index{ultra-resonant monomial}Ultra-resonant }terms are resonant
terms of the form
\begin{enumerate}
\item $x^{i}$ with $i\ge2$ in $x_{1}$,
\item $x^{i}z^{e_{k}}$ with $i\ge1$ and $1\le k\le d-m$ in $z_{1}$ (via
$\lambda_{k}=\lambda_{j}$ in the $j$-th component of $z_{1}$).
\end{enumerate}
\item $F$ is \emph{\index{dynamically separating}dynamically separating}
in the characteristic direction $[v]$, if $R$ contains no ultra-resonant
terms $x^{i}z^{e_{k}}$, $k\le d-m$ up to order $\nu-1$ (with $1\le i<\nu-1$).
\end{enumerate}
\end{defn}

\begin{lem}[{\cite[Lem.~2.10]{Rong2008QuasiparabolicanalytictransformationsofmathbbCn}}]
Dynamic separation is a holomorphic invariant.
\end{lem}

\begin{rem}
Ultra-resonant terms are those resonant terms, that cannot be eliminated
via blow-ups centred at the characteristic direction $[v]$. Dynamical
separation ensures that large values of the dominant parabolic variable
$x$ do not interfere too much with the behaviour of the $z$ component.
\end{rem}

Now we are ready to state Bracci, Molino, and Rong's first main result
generalising Theorem~\ref{thm:HakimEcale}:
\begin{thm}[Bracci, Molino \cite{BracciMolino2004ThedynamicsnearquasiparabolicfixedpointsofholomorphicdiffeomorphismsinmathbbC2},
Rong \cite{Rong2008QuasiparabolicanalytictransformationsofmathbbCn}]
\label{thm:QuasiParaCurves}Let $F\in\End(\mathbb{C}^{d},0)$ be
quasi-parabolic of finite order $\nu<\infty$, with diagonalisable
linear part $dF_{0}$, and dynamically separating in a non-degenerate
characteristic direction $[\nu]$. Then there exists a parabolic flower
of $\nu-1$ parabolic curves tangent to $[v]$ at $0$.
\end{thm}

A more general sufficient condition in dimension two in terms of formal
invariant curves is given in \cite[Thm.~2]{LopezHernanzRaissyRibonSanzSanchez2019StableManifoldsofTwoDimensionalBiholomorphismsAsymptotictoFormalCurves}.

\subsection{Parabolic manifolds}

In \cite{Rong2010QuasiparabolicanalytictransformationsofmathbfCnParabolicmanifolds},
Rong further extended Theorem~\ref{thm:QuasiParaCurves} to include
higher dimensional parabolic manifolds by adapting Hakim's proof of
Theorem~\ref{thm:Hakim2} in \cite{HakimTransformationstangenttotheidentityStablepiecesofmanifolds}.
The formulation does not look much different from Theorem~\ref{thm:Hakim2}:
\begin{thm}[Rong \cite{Rong2010QuasiparabolicanalytictransformationsofmathbfCnParabolicmanifolds}]
\label{thm:RongQuasiParaParaMnf}Let $F\in\End(\mathbb{C}^{d},0)$
be quasi-parabolic of finite order $\nu<\infty$, with diagonalisable
linear part $dF_{0}$, and dynamically separating in a non-degenerate
characteristic direction $[v]$ with directors $\alpha_{1},\ldots,\alpha_{d-1}$.
If there exist $c>0$ and $k\le d-1$ such that $\Re\gamma_{1},\ldots,\Re\gamma_{k}>c$
and $\Re\gamma_{k+1},\ldots,\Re\gamma_{d-1}<c$, then there exists
a parabolic flower of $\nu-1$ parabolic manifolds of dimension $k+1$
tangent to $[v]$. If $F\in\Aut(\mathbb{C}^{d})$ is a global automorphism,
each such parabolic manifold $M$ admits a biholomorphism $\varphi:M\to\mathbb{C}^{k+1}$
conjugating $F$ to the translation $(z,w)\mapsto(z+1,w)$.
\end{thm}

However, the definition of the directors is a bit more involved in
this case:

Let $F\in\End(\mathbb{C}^{d},0)$ be as in (\ref{eq:2.9quasipar3var})
and $\sigma:(\tilde{\mathbb{C}}_{0}^{d},E_{0})\to(\mathbb{C}^{d},0)$
be the blow-up at $0\in\mathbb{C}^{d}$. Then the coordinates on $\{x\neq0\}$
as in Section~\ref{subsec:BlowupCharts} are just $(x,u,v)=(x,\frac{y}{x},\frac{z}{x})$
and are centred at $(0,[1:0:0])\in\tilde{\mathbb{C}}_{0}^{d}$. Since
$[1:0:0]$ was a characteristic direction, the lift $\tilde{F}$ of
$F$ in these coordinates is of the form 
\begin{equation}
\begin{alignedat}{2}x_{1} & =x+x^{\nu}p_{\nu}(1,u,0) &  & +\tilde{P}(x,u,v)+O(\nu+1),\\
u_{1} & =u+x^{\nu-1}(q_{\nu}(1,u,0)-p_{\nu}(1,u,0)\cdot u) &  & +\tilde{Q}(x,u,v)+O(\nu+1)+O(x^{\nu}),\\
v_{1} & =\Lambda v+x^{\nu-1}(r_{\nu}(1,0,v)-p_{\nu}(1,u,0)\cdot\Lambda v) &  & +\tilde{R}(x,u,v)+O(\nu+1)+O(x^{\nu}),
\end{alignedat}
\label{eq:2.18quasiparBlownUp}
\end{equation}
where $r_{\nu}$ contains all terms of the form $x^{\nu-1}z_{k}$,
$k\in\{1,\ldots,m\}$ in $R$. In particular, $\tilde{F}$ is still
a quasi-parabolic germ at $0$ in these coordinates with characteristic
direction $[1:0:0]$.
\begin{rem}
\label{rem:MonomialsTransformUnderBlowUp}Recall From (\ref{eq:monomialTransfBlowUp})
that terms of the form $x^{i}y^{\alpha}z^{\beta}$ transform into
$x^{i+|\alpha|+|\beta|}u^{\alpha}v^{\beta}$ in $x_{1}$ and to 
\[
x^{i+|\alpha|+|\beta|}u^{\alpha}v^{\beta}/x_{1}=x^{i+|\alpha|+|\beta|-1}u^{\alpha}v^{\beta}+HOT,
\]
 otherwise (assuming total order $i+|\alpha|+|\beta|\ge2$). In particular,
the terms $x^{i}z_{k}$ in $z_{1}$ relevant for dynamic separation
are precisely the terms that occur as $x^{i}v_{k}$ in $v_{1}$. Therefore,
dynamic separation is preserved, and we may repeat the procedure as
many times as we like.
\end{rem}

\begin{rem}
Note that the new characteristic direction $[1:0:0]$ is not tangent
to the exceptional divisor, so all parabolic manifolds tangent to
it will be outside the exceptional divisor, where $\sigma$ is a biholomorphism,
so they will be mapped to parabolic manifolds of the original germ.
\end{rem}

Let now $w=(u,v)$, to get 
\begin{equation}
\begin{alignedat}{2}x_{1} & =x+x^{\nu}p_{\nu}(1,u,0) &  & +\tilde{P}(x,u,v)+O(\nu+1),\\
w_{1} & =Lw+x^{\nu-1}s(u,v) &  & +\tilde{S}(x,u,v)+O(\nu+1)+O(x^{\nu}),
\end{alignedat}
\label{eq:2.18quasiparBlownUpGrouped}
\end{equation}
where $L=\diag(I_{l-1},\Lambda)$. Rescaling $x$ by $(-p_{\nu}(1,0,0))^{-1/(\nu-1)}$
and using Taylor expansion of $s$, this finally becomes 
\begin{equation}
\begin{alignedat}{2}x_{1} & =x-x^{\nu} &  & +\tilde{P}(x,w)+O(x^{\nu-1},x^{\nu}\lVert w\rVert),\\
w_{1} & =(L-x^{\nu-1}\tilde{M})w &  & +\tilde{S}(x,w)+O(x^{\nu-1}\lVert w\rVert^{2},x^{\nu}\lVert w\rVert,x^{\nu}),
\end{alignedat}
\label{eq:2.19quasiparAlmostDoneMine}
\end{equation}
where $p_{\nu}(1,0,0)\cdot\tilde{M}$ is the linear part of $s$ at
$(0,0)$ and $\tilde{P}$ and $\tilde{S}$ only contain resonant terms
of total order between $2$ and $\nu$ and order in $v$ at least
$1$.

\begin{rem}
\label{rem:2.12TranformOfDirectors}Changes of coordinates in $w$
act on $\tilde{M}$ via conjugation with their linear part, so the
class of similarity of $\tilde{M}$ is preserved.
\end{rem}

To eliminate the lower order terms in $\tilde{P}$ and $\tilde{S}$,
we perform a number of blow-ups centred at $[1:0]$. Each blow up
increases the total order of the terms in $\tilde{P}$ and $\tilde{S}$
that are not purely powers of $x$ by at least one (see Remark~\ref{rem:MonomialsTransformUnderBlowUp}),
so we need at least
\[
\delta:=\nu(F)-\min\{|\alpha|\mid\tilde{P}_{[i,\alpha]}\neq0\}\cup\{|\alpha|-1\mid\tilde{S}_{[i,\alpha]}\neq0\}
\]
blow-ups, resulting in 
\begin{equation}
\begin{alignedat}{1}x_{1} & =x-x^{\nu}+O(x^{\nu-1},x^{\nu}\lVert w\rVert)\\
w_{1} & =(L-x^{\nu-1}M)w+\tilde{S}(x)+O(x^{\nu-1}\lVert w\rVert^{2},x^{\nu}\lVert w\rVert,x^{\nu}),
\end{alignedat}
\label{eq:quasiParPenultimate}
\end{equation}
 where $M=\tilde{M}-\delta\cdot I_{d-1}$ and $\tilde{S}$ contains
only pure powers of $x$ (and only in the $u_{1}$-components). Based
on this form we can finally define the last remaining invariant we
need to generalise Hakim's results.
\begin{defn}
The eigenvalues of $N:=L^{-1}M$ are the \emph{directors\index{director of a characteristic direction}}
of the characteristic direction $[1:0]$.
\end{defn}

Thus we have all the definitions necessary to understand the statement
of Theorem~\ref{thm:RongQuasiParaParaMnf}.

\subsection{Other cases}

In \cite[§4]{Rong2008QuasiparabolicanalytictransformationsofmathbbCn},
Rong gives examples of a parabolic-elliptic germs that are not dynamically
separating, but still exhibit parabolic curves and domains. In \cite{BracciRong2014Dynamicsofquasiparaboliconeresonantbiholomorphisms},
Bracci and Rong find conditions for parabolic domains for parabolic-elliptic
germs with simple multiplier $1$ and the remaining multipliers forming
a one-resonant tuple in the sense of Definition~\ref{def:oneRes}.

\section{\label{sec:Elliptic-attracting-fixed-points}Elliptic-attracting
fixed points}

\selectlanguage{british}

In the elliptic-attracting case with only one elliptic multiplier,
we can expect one-dimensional elliptic dynamics, such as Siegel disks
and Hedgehogs.
\begin{defn}
A (formal) germ $F\in\Pow(\mathbb{C}^{d},0)$ is called 
\begin{enumerate}
\item \emph{elliptic-attracting\index{elliptic-attracting germ}} if $F$
has both elliptic and attracting multipliers and no other multipliers.
\item \emph{neutral-attracting\index{neutral-attracting germ}} if $F$
has both neutral and attracting multipliers and no other multipliers.
\end{enumerate}
\end{defn}

A first negative result in the case of a single elliptic multiplier
is a generalisation of the snail lemma~\ref{lem:SnailLemma} by Lyubich
and Peters:

\begin{thm}[Lyubich, Peters \cite{LyubichPeters2014ClassificationofinvariantFatoucomponentsfordissipativeHenonmaps}]
\label{thm:2DsnailLemma}Let $F\in\End(\mathbb{C}^{d},0)$ with multipliers
$\lambda_{1},\ldots,\lambda_{d}$ with $|\lambda_{1}|=1$ and $|\lambda_{j}|<1$
for $j\ge2$. If there exists an open set $W\subseteq\mathbb{C}$
such that $f(W)\cap W\neq\emptyset$ and the sequence $\{F^{\circ n}\}_{n}$
converges uniformly to $0$ on $W$, then $\lambda_{1}=1$.
\end{thm}

Just like in dimension one, this rules out the existence of uniformly
attracting domains at $0$ in this case.

More precise information in dimension $2$ was recently established
by Firsova, Lyubich, Radu, and Tanase, and in \cite{FirsovaLyubichRaduTanase2016HedgehogsforNeutralDissipativeGermsofHolomorphicDiffeomorphismsofC20}
and \cite{LyubichRaduTanase2016HedgehogsinHigherDimensionsandTheirApplications},
who use the centre manifold reduction of Theorem~\ref{thm:CentManRedLyubich}
to recover Perez-Marco's hedgehogs:

\begin{thm}[\cite{FirsovaLyubichRaduTanase2016HedgehogsforNeutralDissipativeGermsofHolomorphicDiffeomorphismsofC20}]
\label{thm:SCVSiegelCompacta}Let $F\in\End(\mathbb{C}^{2},0)$ be
neutral-attracting with neutral multiplier $\lambda\in S^{1}$. For
sufficiently small open balls $B\subseteq\mathbb{C}^{2}$ containing
$0$ (such that $F$ is partially hyperbolic on a neighbourhood $B'\subseteq\mathbb{C}^{2}$
of $\overline{B}$), there exists $K\subseteq\overline{B}$ such that 
\begin{enumerate}
\item $K$ is contained in any centre manifold $W^{\mathrm{c}}$ (relative
to $B'$) of $F$.
\item $K$ is compact, connected, completely $F$-invariant and full.
\item $0\in K$ and $K\cap\partial B\neq\emptyset$.
\end{enumerate}
Moreover: 
\begin{enumerate}
\item $K$ contains a neighbourhood of $0$ inside $W^{\mathrm{c}}$, if
and only if $F$ is holomorphically conjugate to a map $(x,y)\mapsto(\lambda x,O(y))$.
\item If $\lambda$ is elliptic, $K$ is unique and equal to the connected
component containing $0$ of the completely stable set 
\[
\Sigma_{F}^{\pm}(\overline{B})=\{z\in W^{\mathrm{c}}\mid F^{\circ n}(z)\in\overline{B}\text{ for all }z\in\mathbb{Z}\}.
\]
\end{enumerate}
\end{thm}

\begin{defn}
In the setting of Theorem~\ref{thm:SCVSiegelCompacta}, if $\lambda$
is elliptic and $F$ does not admit a Siegel disk at $0$ (without
loss of generality contained in $K$), then the compact $K$ is called
a \emph{\index{hedgehog@\emph{hedgehog}}hedgehog} of $F$ in $B$.
\end{defn}

The analogue of Theorem~\ref{thm:PerezMarcoCremerHedgehog} is:
\begin{thm}[\cite{LyubichRaduTanase2016HedgehogsinHigherDimensionsandTheirApplications}]
\label{thm:EllAttrHedgehog}Let $F\in\End(\mathbb{C}^{2},0)$ be
elliptic-attracting with elliptic multiplier $\lambda\in S^{1}$ and
a hedgehog $K$ in $B\subseteq\mathbb{C}^{2}$. Then 
\begin{enumerate}
\item $K$ has empty interior $K^{\circ}=\emptyset$.
\item $K$ is not locally connected at any point except possibly $0$.
\item $K\backslash\{0\}$ has an uncountable number of connected components.
\item There exists a subsequence $\{n_{k}\}_{k\in\mathbb{N}}\subseteq\mathbb{N}$
such that $\{F^{\circ n_{k}}\}_{k}$ converges uniformly on $K$ to
$\id_{K}$.
\item All orbits outside $K$ accumulating on $K$ do so exponentially fast.
\item In particular, all orbits converging to $0$ do so exponentially fast,
i.e. the realm of attraction $A_{F}(0)$ is equal to the strong stable
manifold $W^{\mathrm{ss}}$.
\end{enumerate}
\end{thm}

\section{Other elliptic fixed points}

\selectlanguage{british}

The understanding of general elliptic behaviour in higher dimensions
is wide open, outside full or partial linearisation in Section~\ref{sec:BrjunoConditions}
or reduction to one variable in Section~\ref{sec:Elliptic-attracting-fixed-points}.
Two major open problems are the structure of resonances and formal
normal forms and the convergence (or divergence) of formal normalisations.
Before we say a few words about both of these problems, let us state
the analogue of Proposition~\ref{prop:linIffStable} relating linearisation
and stability:
\begin{prop}
Let $F\in\End(\mathbb{C}^{d},0)$ be a neutral germ with multipliers
$\lambda_{1},\ldots,\lambda_{n}\in S^{1}$. Then the following are
equivalent:
\begin{enumerate}[noitemsep]
\item $F$ is holomorphically linearisable and $dF_{0}$ diagonalisable.
\item $F$ is topologically linearisable and $dF_{0}$ diagonalisable.
\item $0$ is a stable fixed point for $F$, i.e. $\{F^{n}\}$ is normal
at $0$.
\end{enumerate}
\end{prop}

The proof is the same as that of Proposition~\ref{prop:linIffStable}
once we observe that if $dF_{0}$ contains a Jordan block, then $dF_{0}^{\circ n}$
does not converge and hence $\{F^{\circ n}\}$ cannot be normal at
$0$.

\subsection{Structure of resonances}

We have fundamentally seen two types of resonances: Those that occur
singularly, as is the case in the Poincaré domain and those that are
generated by identities of the form $\lambda^{\alpha}=1$ and hence
belong to an infinite family of resonances. These are two extremes
of all possible resonances. The structure of general resonances has
been studied by Raissy in \cite{Raissy2010Torusactionsinthenormalizationproblem}.
We categorise resonances with the following definition following \cite{Abate2011OpenProblemsinLocalDiscreteHolomorphicDynamics}:
\begin{defn}
Let $\lambda=(\lambda_{1},\ldots,\lambda_{d})\in\mathbb{C}^{d}$. 
\begin{enumerate}[noitemsep]
\item The set of resonances of$\lambda$ is 
\[
\Res(F):=\Res(\lambda):=\pbrace{(\alpha,j)\in\mathbb{N}^{d}\times\{1,\ldots,d\}\mid|\alpha|\ge2,\lambda^{\alpha}=\lambda_{j}}.
\]
\item For a product $\lambda^{\alpha}=1$, with $\alpha\neq0$, the multi-index
$\alpha$ is called a \emph{resonance generator }for $\lambda$. We
let 
\[
\m{ResGen}(\lambda)=\{\alpha\in\mathbb{N}^{d}\backslash\{0\}\mid\lambda^{\alpha}=1\}
\]
 denote the set of resonance generators.
\item A resonance $\lambda_{j}=\lambda^{\beta}\neq0$, $|\beta|\ge2$ is
called \emph{regular}, if $\beta_{j}\ge1$ and $\lambda^{\beta-e_{j}}=1$,
and \emph{irregular} otherwise.
\item $\lambda$ is called \emph{resonance effective}, if $\lambda$ has
only regular resonances.
\item An irregular resonance $\lambda_{j}=\lambda^{\beta}$ is called \emph{singular},
if it does not properly contain a generator, i.e.\ there is no multi-index
$\alpha<\beta$ such that $\lambda^{\alpha}=1$.
\item A resonance $0=\lambda_{j}=\lambda^{\alpha}$ is called \emph{degenerate.}
\end{enumerate}
\end{defn}

\begin{rem}[Properties]
Let $\lambda=(\lambda_{1},\ldots,\lambda_{d})\in\mathbb{C}^{d}$.
\begin{enumerate}[noitemsep]
\item A generator $\lambda^{\alpha}=1$ generates an infinite number of
regular resonances $\lambda_{j'}=\lambda^{k\alpha}\lambda_{j'}$ for
$j'\in\{1,\ldots,d\}$ and $k\in\mathbb{N}$ and all regular resonances
are generated this way. If $\lambda$ is resonance effective, then
all resonances of $\lambda$ are induced by generators, in that $\lambda^{\alpha}=\lambda_{j}$,
if and only if $\alpha-e_{j}\in\Res_{1}(\lambda)$. 
\item The set $\m{ResGen}(\lambda)$ is closed under addition and either
empty or infinite.
\item An irregular resonance $\lambda_{j}=\lambda^{\beta}$ is either singular
or can be decomposed into a singular resonance $\lambda_{j}=\lambda^{\beta'}$
and a generator $\lambda^{\alpha}=1$ such that $\beta=\alpha+\beta'$.
This decomposition is in general not unique.

\begin{singlespace}

\end{singlespace}
\item If $\lambda$ has a finite number of resonances, they are all singular
and non-degenerate. $\lambda$ can only have a finite number of non-degenerate
singular resonances.
\end{enumerate}
\end{rem}

\begin{example}
Let $F\in\End(\mathbb{C}^{d},0)$ with multipliers $\lambda=(\lambda_{1},\ldots,\lambda_{d})$.
\begin{enumerate}[noitemsep]
\item If $F$ is in the Poincaré domain, then $F$ has a finite number
of resonances and they are all singular. 
\item If $F$ is parabolic, $\lambda$ is resonance effective, if and only
if at most one multiplier is equal to $1$.
\item If $F$ is tangent to the identity, $\Res(\lambda)$ contains any
$|\alpha|\ge2$. All resonances are generated by the generators $e_{1},\ldots,e_{d}$
and the singular resonances $\lambda_{j}=\lambda_{k}$ for $j\neq k$.
\item If $F$ is multi-resonant (see Definition~\ref{def:oneRes}), then
$F$ is resonance effective and there exists a basis of generators
$P^{1},\ldots,P^{m}\in\m{ResGen}(\lambda)$.
\end{enumerate}
\end{example}

Even outside the multi-resonant case, \cite[Prop.~7.19]{Raissy2010Torusactionsinthenormalizationproblem}
implies the following representation (see also \cite{Abate2011OpenProblemsinLocalDiscreteHolomorphicDynamics}):
\begin{lem}
\label{lem:ResGenDecomp}Let $\lambda\in(\mathbb{C}^{*})^{d}$. There
exist a finite number of distinct elements $P^{1},\ldots,P^{m},C_{1},\ldots,C_{r}\in\m{ResGen}(\lambda)$,
such that every $\alpha\in\m{ResGen}(\lambda)$ can be written as
\[
\alpha=a_{1}P^{1}+\cdots+a_{m}P^{m}+\varepsilon C_{j}
\]
for suitable $a_{1},\ldots,a_{m}\in\mathbb{N}$, $\varepsilon\in\{0,1\}$,
and $j\in\{1,\ldots,r\}$.
\end{lem}

\begin{defn}
The elements $P^{1},\ldots,P^{m}$ are called \emph{minimal}, the
elements $C_{1},\ldots,C_{r}$ \emph{cominimal}.
\end{defn}

\begin{rem}
In this language, multi-resonance is equivalent to resonance effectiveness
with 
\begin{enumerate}[noitemsep]
\item $\mathbb{Q}$-linearly independent minimal elements and 
\item no cominmial elements. 
\end{enumerate}
By Lemma~\ref{lem:ResGenDecomp} if $F$ is resonance effective and
in Poincaré-Dulac normal form, we may be able to study the dynamics
via a parabolic shadow on $\mathbb{C}^{m+r}$. In \cite{Abate2011OpenProblemsinLocalDiscreteHolomorphicDynamics},
Abate poses the question, if it is possible to obtain similar results
to Theorem~\ref{thm:BracciZaitsevDynamics} if we drop one of above
conditions.
\end{rem}

On the other extreme, germs with only singular resonances have finite
Poincaré-Dulac normal forms. This is sometimes referred to as the
\emph{nodal} case.

\subsection{Normalisation}

There are two important questions about normalisation of general elliptic
germs: 
\begin{enumerate}[noitemsep]
\item What are sufficient conditions for a normalisation to converge?
\item If the normalisation diverges, can we still understand the dynamics
of the germ?
\end{enumerate}
Raissy gives some perspective on possible answers to both questions
in the final section of \cite{Raissy2018ALocalApproachtoHolomorphicDynamicsinHigherDimension}.
We give only a basic idea of nature of both problems:
\begin{enumerate}
\item The proof of holomorphic normalisation in the Poincaré domain from
Theorem~\ref{thm:PoincareNormalisationPolynom} does not extend to
resonant elliptic germs even with finite sets of resonances. A suitable
additional condition for holomorphic normalisation of holomorphic
vector fields has be been given in \cite{Brjuno1971AnalyticFormofDifferentialEquationsIII},
but no analogue for biholomorphisms is known. Raissy investigated
the differences between the two settings in \cite{Raissy2010Torusactionsinthenormalizationproblem}.
\item Even the dynamics of elliptic germs $F\in\End(\mathbb{C}^{d},0)$
that are formally, but not holomorphically, linearisable are still
mysterious. Before a general theory like Pérez-Marco's Siegel compacta
can take form, a complete description of the parabolic case might
be required.

\end{enumerate}

\cleardoublepage{}

\chapter{\label{chap:FatouAut}Fatou components of automorphisms}

\selectlanguage{british}

In dimension $2$ and above, we are far from understanding the Fatou
components of endomorphisms of $\mathbb{C}^{2}$. The class $\End(\mathbb{C}^{2})$
is much larger and more complicated than $\End(\mathbb{C})$. Already
the group of automorphisms $\Aut(\mathbb{C}^{2})$ has infinite dimension
(whereas $\Aut(\mathbb{C})$ or $\Aut(\mathbb{P})$ have finite dimension
and trivial dynamics). For a systematic treatment, see Forstneri\v{c}
\cite[Chap.~4]{Forstneriv2017SteinmanifoldsandholomorphicmappingsThehomotopyprincipleincomplexanalysis}.

Therefore most classification results apply only to certain subclasses.
We start with a quick review of the research on polynomials and projective
endomorphisms, before honing in on the class of automorphisms $\Aut(\mathbb{C}^{2})$
and its invariant components, that we discuss in detail.

First, we introduce some terminology:
\begin{defn}
Let $X$ be a complex manifold and $F\in\End(X)$. 
\begin{enumerate}[noitemsep]
\item The \emph{$\omega$-limit set\index{omega-limit set@\emph{$\omega$-limit set}}}
$\omega_{F}(p)$ of a point $p\in X$ or $\omega_{F}(U)$ of an open
set $U\subseteq\mathbb{C}^{2}$ under $F$ is the set of all accumulation
points of orbits under $F$ starting in $p$ or $U$ respectively.
\item If $V$ is a Fatou component of $F$ and $\{n_{k}\}_{k}$ is a subsequence
such that $F^{\circ n_{k}}\xrightarrow[k\to+\infty]{}F_{\infty}$
locally uniformly on $V$, then $F_{\infty}:V\to\omega_{F}(V)$ is
called a \emph{\index{limit function@\emph{limit function}}limit
function} on $V$.
\end{enumerate}
\end{defn}

\begin{defn}
Let $X$ be a complex manifold, $F\in\End(X)$, and $V\subseteq X$
be a Fatou component of $F$.
\begin{enumerate}[noitemsep]
\item $V$ is \emph{\index{invariant Fatou component@\emph{invariant Fatou component}}invariant},
if $F(V)=V$.
\item $V$ is \emph{recurrent}\index{recurrent Fatou component@\emph{recurrent Fatou component}},
if $V$ contains an accumulation point of an $F$-orbit, and \emph{non-recurrent}
or \emph{transient} otherwise.
\item $V$ is \emph{attracting} \index{attracting Fatou component@\emph{attracting Fatou component}}to
$P\in X$, if $F^{\circ n}(z)\xrightarrow[n\to+\infty]{}P$ for all
$z\in V$.
\item $V$ is a \emph{rotation domain}\index{rotation domain@\emph{rotation domain}},
if there exists a subsequence $\{n_{k}\}_{k}$ such that $F^{\circ n_{k}}|_{V}\xrightarrow[k\to+\infty]{}\id_{V}$,
i.e.\ $\id_{V}$ is a limit function on $V$.
\item $V$ is \emph{$p$-periodic} for some $p\in\mathbb{N}^{*}$, if $F^{\circ p}(V)=V$
and $F^{\circ n}(V)\neq V$ for any $n<p$. In this case, $\{F^{\circ n}(V)\}_{0\le n<p}$
is a \emph{$p$-periodic cycle\index{periodic Fatou component@\emph{periodic Fatou component}}}
of Fatou components.
\item $V$ is \emph{\index{wandering Fatou component@\emph{wandering Fatou component}}wandering},
if no image $F^{\circ n}(V)$ is periodic for any $n$.
\end{enumerate}
\end{defn}

\begin{rem}
Let $V$ be a Fatou component of $F\in\End(X)$.
\begin{enumerate}[noitemsep]
\item $V$ is recurrent, precisely when $z\in\omega_{F}(z)$ for some $z\in V$.
A non-periodic point of this type is called a \emph{recurrent point\index{recurrent point@\emph{recurrent point}}}
for $F$.
\item If $V$ is recurrent, then $V$ is periodic.
\item If $V$ is attracting to $P\in X$, then $P$ is a fixed point of
$F$. 
\item If $V$ is invariant and attracting to $P$, then $P\in\overline{V}$
and $V$ is recurrent if and only if $P\in V$.
\end{enumerate}
\end{rem}

\section{Classification results }

\subsection{Polynomials and projective endomorphisms}

In dimension $2$, we immediately recover products of all the types
of Fatou components from Section~\ref{sec:1DFatouClassification},
by considering direct products $F\in\End(\mathbb{C}^{2})$,
\begin{equation}
F(z,w)=(f(z),g(w)),\label{eq:dirProd}
\end{equation}
where $f$ and $g$ are polynomials (or entire functions). However,
already for polynomial \emph{skew-products\index{skew-product@\emph{skew-product}}}
\begin{equation}
F(z,w)=(f(z,w),g(w)),\label{eq:SkewProd}
\end{equation}
where $f:\mathbb{C}^{2}\to\mathbb{C}$ and $g:\mathbb{C}\to\mathbb{C}$
are polynomials, several new phenomena appear: Boc Thaler, Fornæss
and Peters \cite{BocThalerFornessPeters2015FatouComponentswithPuncturedLimitSets}
constructed examples of Fatou components with punctured limit set,
and wandering components have been found in \cite{AstorgBuffDujardinPetersRaissy2016ATwoDimensionalPolynomialMappingwithaWanderingFatouComponent}
and \cite{AstorgBocThalerPeters2019WanderingDomainsArisingfromLavaursMapswithSiegelDisks}.
Dynamics of polynomial skew products have been studied in \cite{Heinemann1998JuliaSetsofSkewProductsinC2},
\cite{Rivi1998LocalBehaviourofDiscreteDynamicalSystems}, \cite{Jonsson1999DynamicsofPolynomialSkewProductsonC2},
\cite{HruskaRoeder2010TopologyofFatouComponentsforEndomorphismsofCPkLinkingwiththeGreensCurrent},
\cite{Roeder2011ADichotomyforFatouComponentsofPolynomialSkewProducts},
\cite{PetersVivas2016PolynomialSkewProductswithWanderingFatouDisks},
\cite{Raissy2017PolynomialSkewProductsinDimension2BulgingandWanderingFatouComponents},
\cite{PetersSmit2018FatouComponentsofAttractingSkewProducts}, \cite{AstorgBianchi2018HyperbolicityandBifurcationsinHolomorphicFamiliesofPolynomialSkewProducts},
and \cite{PetersRaissy2019FatouComponentsofEllipticPolynomialSkewProducts}
.

A rational map $F:\mathbb{C}^{2}\to\mathbb{P}^{2}$ extends to an
endomorphism $F\in\End(\mathbb{P}^{2})$, if and only if all components
of $F$ have the same degree. Hence, examples of endomorphisms $F\in\End(\mathbb{P}^{2})$\textbf{
}can be readily constructed as direct products (\ref{eq:dirProd})
and skew-products (\ref{eq:SkewProd}) of rational maps $f$ and $g$
of common degree. In particular, the skew-products in \cite{AstorgBuffDujardinPetersRaissy2016ATwoDimensionalPolynomialMappingwithaWanderingFatouComponent},
that admit wandering components, can be chosen such that they extend
to projective endomorphisms.

A classification of recurrent Fatou components of endomorphisms of
$\mathbb{P}^{2}$ similar to Theorem~\ref{thm:PolAutRecurrent} follows
from \cite{FornessSibony1995Classificationofrecurrentdomainsforsomeholomorphicmaps}
and \cite{Ueda2008HolomorphicmapsonprojectivespacesandcontinuationsofFatoumaps},
and has been extended to $\mathbb{P}^{d},d\ge2$ in \cite{FornessRong2014ClassificationofRecurrentDomainsforHolomorphicMapsonComplexProjectiveSpaces}.
Further results on the dynamics of projective endomorphisms can be
found in \cite{FornaessSibony1994ComplexdynamicsinhigherdimensionI},
\cite{Sibony1999DynamiqueDesApplicationsRationnellesDePk}, and \cite{DinhSibony2010DynamicsinSeveralComplexVariablesEndomorphismsofProjectiveSpacesandPolynomiallikeMappings}.

\subsection{Polynomial automorphisms}

An important class of polynomial automorphisms are \emph{Hènon maps\index{Hènon map@\emph{Hènon map}}}
\[
H(z,w)=(p(z)-\delta w,z),\quad(z,w)\in\mathbb{C}^{d-1}\times\mathbb{C}
\]
for a polynomial $p\in\mathbb{C}[z]$ and a constant $\delta\in\mathbb{C}^{*}$.
For $d=2$, they contain all interesting dynamics of Polynomial automorphisms
(Friedland, Milnor \cite{FriedlandMilnor1989DynamicalPropertiesofPlanePolynomialAutomorphisms}).
This is no longer true for $d\ge3$ (Shestakov, Umirbaev \cite{ShestakovUmirbaev2004TheTameandtheWildAutomorphismsofPolynomialRingsinThreeVariables}).
This was the starting point for the study of extensive investigations
of the dynamics of polynomial Automorphisms and Hénon maps on $\mathbb{C}^{2}$
by Hubbard, Oberste-Vorth, Papadopol, and Veseloc in \cite{Hubbard1986TheHenonMappingintheComplexDomain},
\cite{HubbardObersteVorth1994HenonMappingsintheComplexDomainItheGlobalTopologyofDynamicalSpace},
\cite{HubbardObersteVorth1995HenonMappingsintheComplexDomainIIProjectiveandInductiveLimitsofPolynomials},
and \cite{HubbardPapadopolVeselov2000ACompactificationofHenonMappingsinC2AsDynamicalSystems},
by Bedford, Lyubich, and Smillie in \cite{BedfordSmillie1991FatouBieberbachDomainsArisingfromPolynomialAutomorphisms},
\cite{BedfordSmillie1991PolynomialDiffeomorphismsofC2CurrentsEquilibriumMeasureandHyperbolicity},
\cite{BedfordSmillie1991PolynomialdiffeomorphismsofmathbbC2IIStablemanifoldsandrecurrence},
\cite{BedfordSmillie1992PolynomialDiffeomorphismsofC2IIIErgodicityExponentsandEntropyoftheEquilibriumMeasure},
\cite{BedfordLyubichSmillie1993DistributionofPeriodicPointsofPolynomialDiffeomorphismsofC2},
\cite{BedfordLyubichSmillie1993PolynomialDiffeomorphismsofC2IVtheMeasureofMaximalEntropyandLaminarCurrents},
\cite{BedfordSmillie1998PolynomialDiffeomorphismsofC2VCriticalPointsandLyapunovExponents},
\cite{BedfordSmillie1998PolynomialDiffeomorphismsofC2VIConnectivityofJ},
\cite{BedfordSmillie1999PolynomialDiffeomorphismsofC2VIIHyperbolicityandExternalRays},
\cite{BedfordSmillie1999ExternalRaysintheDynamicsofPolynomialAutomorphismsofC2},
and \cite{BedfordSmillie2002PolynomialDiffeomorphismsofC2VIIIQuasiExpansion},
and by Fornæss and Sibony in \cite{FornessSibony1992ComplexHenonmappingsinmathbbC2andFatouBieberbachdomains},
\cite{FornessSibony1995Classificationofrecurrentdomainsforsomeholomorphicmaps},
and \cite{FornessSibony1999Complexdynamicsinhigherdimension}.

A basic property restricting the dynamics of Hénon maps is the following
\emph{filtration} of the modulus plane:
\begin{prop}[\cite{BedfordSmillie1991PolynomialDiffeomorphismsofC2CurrentsEquilibriumMeasureandHyperbolicity}]
\label{prop:HenonFiltr}Let $H\in\Aut(\mathbb{C}^{2})$ be a Hénon
map of the form $H(z,w)=(p(z)-\delta w,z)$. Then for $R>0$ large
enough, the sets 
\begin{align}
W & :=\{(z,w)\in\mathbb{C}^{2}\mid\max(|z|,|w|)\le R\}\nonumber \\
V_{+} & :=\{(z,w)\in\mathbb{C}^{2}\mid|z|\ge\max(|w|,R)\}\label{eq:HenonFiltr}\\
V_{-} & :=\{(z,w)\in\mathbb{C}^{2}\mid|w|\ge\max(|z|,R)\}\nonumber 
\end{align}
satisfy: $H(V_{+})\subseteq V_{+}$, $H^{-1}(V_{-})\subseteq V_{-}$,
any point in $V_{+}$ converges to $[1:0:0]$ in $\mathbb{P}^{2}$
and any orbit diverging from $\mathbb{C}^{2}$ enters $V_{+}$. Hence
$\bigcup_{n\in\mathbb{N}}H^{\circ(-n)}(V_{+})$ is a Fatou component
``attracting to $[1:0:0]$''.
\end{prop}

The recurrent Fatou components were classified by Bedford and Smillie,
and Fornæss and Sibony:
\begin{thm}[\cite{BedfordSmillie1991PolynomialdiffeomorphismsofmathbbC2IIStablemanifoldsandrecurrence},
\cite{FornessSibony1992ComplexHenonmappingsinmathbbC2andFatouBieberbachdomains},
\cite{FornessSibony1995Classificationofrecurrentdomainsforsomeholomorphicmaps}]
\label{thm:PolAutRecurrent}Let $F\in\Aut(\mathbb{C}^{2})$ be a
polynomial automorphism with Jacobian $\delta=\det df\neq0$ (necessarily
constant) and let $\Omega$ be a $p$-periodic Fatou component of
$F$. Then we have one of the following:
\begin{enumerate}[noitemsep]
\item $|\delta|<1$ and $\Omega$ is the basin of attraction of an attracting
fixed point $P\in\Omega$ under $F^{\circ p}$ (and $F^{\circ p}$
is conjugated to a polynomial normal form on $\Omega\cong\mathbb{C}^{d}$).
\item \label{enu:PolAutLimit1}$|\delta|<1$ and there is a biholomorphic
map $\varphi:\Omega\to A\times\mathbb{C}$ conjugating $F$ to 
\[
(z,w)\mapsto(e^{i\theta}z,\lambda w),
\]
where $\theta\in\mathbb{R}$ and $|\lambda|<1$ such that $\lambda e^{i\theta}=\delta$,
and $A\subseteq\mathbb{C}$ is an annulus or a disk. In particular,
$\omega_{F}(\Omega)=\varphi^{-1}(A\times\{0\})$ is a Siegel disk
or Herman ring.
\item $|\delta|=1$ and $\Omega$ is a rotation domain.
\end{enumerate}
In all three cases, $\Sigma=\omega_{F}(\Omega)\subseteq\Omega$ is
a closed complex submanifold, any two limit maps $\Omega\to\Sigma$
differ by an automorphism of $\Sigma$, and, in particular, $F|_{\Sigma}\in\Aut(\Sigma)$.
\end{thm}

\begin{defn}
A Fatou component as in case \ref{enu:PolAutLimit1} of Theorem~\ref{thm:PolAutRecurrent}
is called a (\emph{recurrent}) \emph{Siegel\index{Siegel cylinder@\emph{Siegel cylinder}}}
or \emph{Herman cylinder}\index{Herman cylinder@\emph{Herman cylinder}},
if $A$ is a disk or an annulus, respectively.
\end{defn}

Siegel cylinders are known to exist via local dynamics, but the existence
of Herman cylinders is an open question.

The classification of invariant Fatou components has been almost completed
in the \emph{moderately dissipative} case, by Lyubich and Peters in
\cite{LyubichPeters2014ClassificationofinvariantFatoucomponentsfordissipativeHenonmaps},
where they showed the only non-recurrent invariant Fatou components
to be attracting:
\begin{thm}[\cite{LyubichPeters2014ClassificationofinvariantFatoucomponentsfordissipativeHenonmaps}]
Let $F\in\Aut(\mathbb{C}^{2})$ be a non-elementary polynomial automorphism
of degree $k\ge2$ with (constant) Jacobian $\delta$ such that $|\delta|<1/k^{2}$.
If $\Omega$ is an invariant non-recurrent Fatou component of $F$
with bounded orbits, then $\Omega$ is attracting to a parabolic-attracting
fixed point $P\in\partial\Omega$ of $F$ with parabolic multiplier
$1$.
\end{thm}

For Hénon maps, the filtration (\ref{eq:HenonFiltr}) excludes (suitable
generalisations of) Baker domains, as well as escaping or oscillating
wandering Fatou components.

The question of existence of (orbitally bounded) wandering domains
was finally settled by Berger and Biebler in \cite{BergerBiebler2020EmergenceofWanderingStableComponents},
where they demonstrated such wandering domains for a family of Hénon
maps. Before that, Hahn and Peters modified the construction in \cite{AstorgBuffDujardinPetersRaissy2016ATwoDimensionalPolynomialMappingwithaWanderingFatouComponent}
to obtain examples of wandering Fatou components in polynomial automorphisms
of $\mathbb{C}^{4}$.

\subsection{General Automorphisms}

In \cite{ArosioBeniniFornessPeters2019DynamicsoftranscendentalHenonmaps},
Arosio, Benini, Fornæss, and Peters extend the classification theorem~\ref{thm:PolAutRecurrent}
for recurrent Fatou components to holomorphic automorphisms of $\mathbb{C}^{2}$
with constant Jacobian:
\begin{thm}[\cite{ArosioBeniniFornessPeters2019DynamicsoftranscendentalHenonmaps}]
Let $F\in\Aut(\mathbb{C}^{2})$ with constant Jacobian $\delta\neq0$
and let $\Omega$ be an invariant recurrent Fatou component for $F$.
Then there exists a holomorphic retraction $\rho:\Omega\to\Sigma$
to a closed complex submanifold $\Sigma\subseteq\Omega$ such that
for all limit maps $h$, there exists $\eta\in\Aut(\Sigma)$ such
that $h=\eta\circ\rho$. Every orbit converges to $\Sigma$ and $F|_{\Sigma}\in\Aut(\Sigma)$.
Moreover
\begin{enumerate}[noitemsep]
\item If $\dim\Sigma=0$, then $\Omega$ is the basin of an attracting
fixed point and biholomorphically equivalent to $\mathbb{C}^{2}$.
\item If $\dim\Sigma=1$, then either 
\begin{enumerate}[noitemsep]
\item $\Sigma$ is biholomorphic to a rotation domain $A\subseteq\mathbb{C}$,
and there exists a biholomorphism $\Omega\to A\times\mathbb{C}$ conjugating
$F$ to 
\[
(z,w)\mapsto(e^{i\theta}z,\delta e^{-i\theta}w),
\]
for some $\theta\in\mathbb{R}$, or 
\item there exists $j\in\mathbb{N}$ such that $F^{\circ j}|_{\Sigma}=\id_{\Sigma}$
and there exists a biholomorphism $\Omega\to\Sigma\times\mathbb{C}$
conjugating $F^{\circ j}$ to 
\[
(z,w)\mapsto(z,\delta^{j}w).
\]
\end{enumerate}
\item $\dim\Sigma=2$, if and only if $|\delta|=1$ and in this case, there
exists a subsequence $\{n_{k}\}_{k}\subseteq\mathbb{N}$ such that
$F^{\circ n_{k}}|_{\Omega}\xrightarrow[k\to+\infty]{}\id_{\Omega}$.
\end{enumerate}
\end{thm}

\selectlanguage{british}

In \cite{JupiterLilov2004InvariantnonrecurrentFatoucomponentsofautomorphismsofmathbbC2},
Jupiter and Lilov take first steps towards narrowing down the possibilities
for non-recurrent invariant components. They split their discussion
according to the maximal rank of limit maps of $\{F^{\circ n}\}_{n}$
on the Fatou component $V$. Their main result on rank 0 is the following:
\begin{thm}[\cite{JupiterLilov2004InvariantnonrecurrentFatoucomponentsofautomorphismsofmathbbC2}]
\label{thm:JupLilRank0}Let $\Omega$ be an invariant non-recurrent
Fatou component of $F\in\Aut(\mathbb{C}^{2})$ such that all limit
maps on $\Omega$ have rank $0$ (i.e.\ are constant). If $\Sigma:=\omega_{F}(\Omega)$
contains more than one point, then:
\begin{enumerate}
\item $\Sigma$ is closed, uncountable and has no isolated points.
\item There is a one-dimensional subvariety $V$ of fixed points of $F$
such that $\Sigma\subseteq V\cap\partial\Omega$.
\item There exists a unique rotation $\alpha\in S^{1}$ such that $F$ has
multipliers $1$ and $\alpha$ at all points in $\Sigma$, and $\alpha$
is a non-diophantine rotation, that is, $\alpha$ does not satisfy
(\ref{eq:Siegel}).
\end{enumerate}
\end{thm}

\begin{rem}
There are no known examples with more than one rank $0$ limit map.
\end{rem}

Limit maps of rank $1$ have regular image:
\begin{lem}[{Lyubich, Peters \cite[Lem.~13]{LyubichPeters2014ClassificationofinvariantFatoucomponentsfordissipativeHenonmaps}}]
Let $\Omega$ be an invariant non-recurrent Fatou component of $F\in\Aut(\mathbb{C}^{2})$
and $F_{\infty}$ a limit map of rank $1$ on $\Omega$. Then the
image $F_{\infty}(\Omega)\subseteq\mathbb{C}^{2}$ is an injectively
immersed Riemann surface.
\end{lem}

\begin{rem}
This does not rule out images that are the regular part of singular
analytic sets. In \cite{BocThalerFornessPeters2015FatouComponentswithPuncturedLimitSets},
Boc Thaler, Fornæss, and Peters construct polynomial endomorphisms
with Fatou components with limit maps equal to the regular part of
an arbitrary analytic set. It is unknown if there are automorphisms
with such limit sets.
\end{rem}

For rank $1$, the main result of \cite{JupiterLilov2004InvariantnonrecurrentFatoucomponentsofautomorphismsofmathbbC2}
is:
\begin{thm}
Let $\Omega$ be an invariant non-recurrent Fatou component of $F\in\Aut(\mathbb{C}^{2})$
and $F_{\infty}$ and $G_{\infty}$ be limit maps on $\Omega$ of
rank $1$. Then the intersection of their images $F_{\infty}(\Omega)\cap G_{\infty}(\Omega)$
is either empty or a (relatively) open subset of both $F_{\infty}(\Omega)$
and $G_{\infty}(\Omega)$.
\end{thm}

Jupiter and Lilov moreover construct explicit automorphisms with Fatou
components with a unique rank-$1$ limit map or with an uncountable
family of rank-$1$ limit maps whose common image is a coordinate
axis. The latter are a subclass of the ``parabolic cylinders'' examined
by Boc Thaler, Bracci, and Peters \cite{BocBracciPeters2019AutomorphismsofmathbbC2withnonrecurrentSiegelcylinders}:
\begin{defn}
An invariant non-recurrent Fatou component $V$ of $F\in\End(\mathbb{C}^{2})$
is called a \emph{parabolic cylinder}\index{parabolic cylinder@\emph{parabolic cylinder}},
if
\begin{enumerate}
\item the closure of $\omega_{F}(V)$ contains an isolated fixed point,
\item there is an injective holomorphic map $\Phi:V\to\mathbb{C}^{2}$ conjugating
$F$ to the translation $(z,w)\mapsto(z+1,w)$,
\item all limit maps of $\{F^{\circ n}\}_{n}$ on $\Omega$ have dimension
$1$.
\end{enumerate}
\end{defn}

The parabolic cylinders in \cite{JupiterLilov2004InvariantnonrecurrentFatoucomponentsofautomorphismsofmathbbC2}
and \cite{BocBracciPeters2019AutomorphismsofmathbbC2withnonrecurrentSiegelcylinders}
arise from parabolic dynamics in the $z_{1}$-direction attracting
to the $z_{2}$-axis on which $F$ acts as a (diophantine) irrational
rotation. The limit maps differ precisely by arbitrary rotations of
the $z_{2}$-axis. They are biholomorphic to $\mathbb{C}^{2}$ via
the conjugating map $\Phi$.

In \cite{Reppekus2019PuncturednonrecurrentSiegelcylindersinautomorphismsofmathbbC2},
we used a blow-up of local one-resonant dynamics to obtain an automorphism
of $\mathbb{C}^{2}$ with a ``punctured'' parabolic cylinder biholomorphic
via $\Phi$ to $\mathbb{C}\times\mathbb{C}^{*}$ and with limit maps
mapping to a punctured axis $\{0\}\times\mathbb{C}^{*}$. The dynamics
in this case are one-resonant so the orbits approach the limit set
on a path spiralling in both variables.
\begin{rem}
There are no known examples of rank-$1$ limit maps with non-identical
images or limit maps of different rank.
\end{rem}

Oscillating wandering Fatou components for general automorphisms of
$\mathbb{C}^{2}$ have been realised by Fornæss and Sibony in \cite{FornaessSibony1998FatouandJuliaSetsforEntireMappingsinCk}.
In \cite{ArosioBeniniFornessPeters2019DynamicsoftranscendentalHenonmaps},
\cite{ArosioBeniniFornessPeters2019DynamicsofTranscendentalHenonMapsII},
and \cite{ArosioBocThalerPeters2019ATranscendentalHenonMapwithanOscillatingWanderingShortC2},
Arosio, Benini, Boc Thaler, Fornæss, and Peters study the subclass
of \emph{transcendental Hénon maps} and give examples of escaping
and oscillating wandering domains as well as (a suitable generalisation
of) Baker domains. In \cite{BocThaler2020AutomorphismsofCmwithBoundedWanderingDomains},
Boc Thaler gave an example of an automorphism of $\mathbb{C}^{2}$
with a bounded (not just orbitally bounded) wandering domain.

\section{\label{chap:Aut}Automorphisms with prescribed local dynamics}

\selectlanguage{british}

Returning to local dynamics, one might ask whether any given local
(invertible) dynamical behaviour can be realised by an automorphism.
The answer is almost always yes, as for any invertible germ, we can
find an automorphism with identical series expansion up to any finite
order. See Forstneri\v{c} \cite[Chap.~4]{Forstneriv2017SteinmanifoldsandholomorphicmappingsThehomotopyprincipleincomplexanalysis}
for a detailed introduction to this phenomenon.

An important class of Automorphisms are \emph{shears}\index{shear}
\begin{equation}
F(z,w)=(z,e^{h(z)}w+f(z)),\quad(z,w)\in\mathbb{C}^{d-1}\times\mathbb{C}\label{eq:shearDef}
\end{equation}
where $h,f:\mathbb{C}^{d-1}\to\mathbb{C}$ are holomorphic functions.
Compositions of (conjugates of) shears uniformly approximate any automorphism
of $\mathbb{C}^{d}$ on compacts (Andersén, Lempert \cite{AndersenLempert1992OntheGroupofHolomorphicAutomorphismsofCn},
Forstneri\v{c}, Rosay \cite{ForstnericRosay1993ApproximationofBiholomorphicMappingsbyAutomorphismsofCn}).
In fact the same is true for arbitrary holomorphic maps on polynomially
convex compacts (Forstneri\v{c} \cite{Forstneriv1999InterpolationbyholomorphicautomorphismsandembeddingsinmathbbCn}).
More importantly for our purposes, the power series expansion of these
approximation can be made to coincide with that of the original map
to any finite order at any finite number of points. In particular,
we have:
\begin{thm}[\cite{Weickert1998AttractingbasinsforautomorphismsofbfC2}, \cite{Forstneriv1999InterpolationbyholomorphicautomorphismsandembeddingsinmathbbCn}]
\label{thm:JetIntWeickFor}For every invertible holomorphic germ
$F\in\Aut(\mathbb{C}^{d},0)$, and every $k\in\mathbb{N}$, there
exists an automorphism $G\in\Aut(\mathbb{C}^{d},0)$ such that 
\[
G(z)=F(z)+O(\norm z^{k})
\]
as $z\to0$.
\end{thm}

With this, any dynamical behaviour that depends on finitely many terms
of the series expansion can be realised by Automorphisms of $\mathbb{C}^{d}$.

The essential ingredients in the proof of the Theorem~\ref{thm:JetIntWeickFor}
were formalised by Varolin in \cite{Varolin1999AGeneralNotionofShearsandApplications,Varolin2001Thedensitypropertyforcomplexmanifoldsandgeometricstructures,Varolin2000ThedensitypropertyforcomplexmanifoldsandgeometricstructuresII}
in the form of the density property on Lie algebras of vector fields,
enabling far reaching generalisations. In particular, \cite[Theorem~5.1]{Varolin2001Thedensitypropertyforcomplexmanifoldsandgeometricstructures},
\cite[Theorem~1]{Varolin2000ThedensitypropertyforcomplexmanifoldsandgeometricstructuresII},
and \cite[Example~1]{Varolin2000ThedensitypropertyforcomplexmanifoldsandgeometricstructuresII}
show:
\begin{thm}[Varolin \cite{Varolin2001Thedensitypropertyforcomplexmanifoldsandgeometricstructures,Varolin2000ThedensitypropertyforcomplexmanifoldsandgeometricstructuresII}]
\label{thm:JetIntFixedAxis}For every invertible germ $F\in\Aut(\mathbb{C}^{2},0)$
pointwise fixing $\{z=0\}$ and every $l\in\mathbb{N}$, there exists
a global automorphism $G\in\Aut(\mathbb{C}^{2})$ such that 
\[
G(z,w)=F(z,w)+zO(\norm{(z,w)}^{l})
\]
as $(z,w)\to0$.
\end{thm}

\section{Attracting invariant Fatou components}

\selectlanguage{british}

\subsection{\label{subsec:LocalBasinsFatou}Local basins and Fatou components}

Invariant Fatou components attracting to a fixed point $p$ can be
identified by understanding the local dynamics near $p$. All the
examples of attracting or parabolic domains we have seen in Chapter~\ref{chap:2DlocDyn}
correspond to Fatou components by the following simple consequence
of Montel's compactness principle~\ref{thm:MontelBdd}:
\begin{lem}
\label{lem:localbasinIsInFatouComp}Let $F\in\Aut(\mathbb{C}^{d},0)$
and $D\subseteq\mathbb{C}^{d}$ be a bounded, $F$-invariant open
set, such that $F^{\circ n}(z)\to0$ for all $z\in D$. Then $D$
is contained in a Fatou component of $F$ attracting to $0$.
\end{lem}

\begin{proof}
Since $\{F^{\circ n}|_{D}:D\to D\}_{n}$ is uniformly bounded, by
Montel's theorem~\ref{thm:MontelBdd}, $D$ is contained in a Fatou
component $V$. Every limit map $F^{\circ n_{k}}\xrightarrow[k\to+\infty]{}F_{\infty}$
is the constant $0$ on $D$, hence, by the identity principle, the
same is true on $V$ and $V$ is attracting to $0$.
\end{proof}
The next step is identifying the Fatou component containing $D$.
The obvious (and minimal) candidate is the following:
\begin{defn}
Let $F\in\Aut(\mathbb{C}^{d},0)$ and $D\subseteq\mathbb{C}^{d}$
be a bounded, $F$-invariant open set, such that $F^{\circ n}(z)\to0$
for all $z\in D$. Then we call $D$ a \emph{local basin\index{local basin@\emph{local basin}}}
for $F$ at $0$ and the \emph{global} \emph{basin\index{global basin@\emph{global} \emph{basin}}}
corresponding to $D$ is 
\[
\Omega:=\bigcup_{n\in\mathbb{N}}F^{\circ(-n)}(D).
\]
\end{defn}

\begin{rem}[Properties]
Let $D$ be a local basin and $\Omega$ the corresponding global
basin.
\begin{enumerate}
\item $\Omega$ is completely $F$-invariant, open and all orbits in $\Omega$
converge to $0$ locally uniformly.
\item If $F$ is a global automorphism, the global basin $\Omega$ is a
growing union of biholomorphic preimages of $D$. Hence, $\Omega$
is connected, contained in the same Fatou component as $D$, and has
the same topology as $D$. In the proofs of the results from Chapter~\ref{chap:2DlocDyn},
the more precise biholomorphic type was obtained via local coordinates
that are compatible with the dynamics of $F$ on $D$, and that then
extend to $\Omega$ via backward iteration of $F$.
\item For two disjoint local basins $D_{1}$ and $D_{2}$, their respective
global basins $\Omega_{1}$ and $\Omega_{2}$ are still disjoint.
\end{enumerate}
\end{rem}

A global basin is not in general equal to the containing Fatou component,
and the Fatou components containing two disjoint global basins may
still coincide:
\begin{example}
Let $F:\mathbb{C}^{2}\to\mathbb{C}^{2}$ with $F(z,w)=(z/2,w/2)$,
then $\Omega_{1}=\{\Re z>0\}$ and $\Omega_{2}=\{\Re z<0\}$ are global
basins, but they are contained in the same Fatou component $\mathbb{C}^{2}$
of $F$ and are both different from the containing Fatou component
$\mathbb{C}^{2}$.
\end{example}

Ruling out these cases can still be transformed into a local problem
thanks to normality:
\begin{lem}
\label{lem:localFatou}Let $F\in\Aut(\mathbb{C}^{d},0)$ and $\varepsilon>0$.
\begin{enumerate}
\item Let $\Omega$ be the global basin to a local basin $D\subseteq B_{\varepsilon}(0)$.
Then $\Omega$ is not a Fatou component, if and only if there exists
a local basin $D'\subseteq B_{\varepsilon}(0)$ intersecting $D$
and not contained in $\Omega$.
\item Let $D_{1},D_{2}\subseteq B_{\varepsilon}(0)$ be local basins for
$F$. Then $D_{1}$ and $D_{2}$ are contained in the same Fatou component
of $F$, if and only if there exists a local basin $D\subseteq\mathbb{C}^{d}$
containing both $D_{1}$ and $D_{2}$.
\end{enumerate}
\end{lem}

\begin{proof}
\begin{enumerate}
\item Let $V$ be the Fatou component containing $D$ and $p\in D$ and
$q\in V\backslash\Omega$. Then there exists a common neighbourhood
$U\subseteq\mathbb{C}^{d}$ of $p$ and $q$ such that $\overline{U}\subseteq V$.
By normality of $\{F^{\circ n}\}_{n}$ on $V$, there exists $n_{0}\in\mathbb{N}$
such that $\bigcup_{n\ge n_{0}}F^{\circ n}(\overline{U})\subseteq B_{\varepsilon}(0)$
and hence $D':=\bigcup_{n\ge n_{0}}F^{\circ n}(U)$ is a local basin
intersecting $D$ in $F^{\circ n_{0}}(p)$ and containing $F^{\circ n_{0}}(q)\notin\Omega$.
\item Let $D_{1}$ and $D_{2}$ be contained in a Fatou component $V$ and
let $p\in D_{1}$ and $q\in D_{2}$. Then there exists a common neighbourhood
$U\subseteq\mathbb{C}^{d}$ of $p$ and $q$ such that $\overline{U}\subseteq V$.
By normality, there exists $n_{0}\in\mathbb{N}$ such that $\bigcup_{n\ge n_{0}}F^{\circ n}(\overline{U})\subseteq B_{\varepsilon}(0)$
and hence $D_{3}:=\bigcup_{n\ge n_{0}}F^{\circ n}(U)$ is a local
basin intersecting $D_{1}$ and $D_{2}$. Hence, $D=D_{1}\cup D_{2}\cup D_{3}\subseteq B_{\varepsilon}(0)$
is a local basin.
\end{enumerate}
\end{proof}

In conclusion, for $F\in\Aut(\mathbb{C}^{d},0)$, once we identify
the open connected components of the realm of attraction $A_{F}(0)\subseteq U$
with respect to a bounded neighbourhood $U\subseteq\mathbb{C}^{d}$
of $0$, we have identified the Fatou components attracting to $0$.

If we want to identify a global basin as a Fatou component, but do
not know the full realm of attraction with respect to a neighbourhood
of the fixed point, a different approach may be necessary. One such
approach will be described in the following section.

\subsection{State of classification}

The state of classification of invariant attracting Fatou components
of automorphisms of $\mathbb{C}^{2}$ is closely tied to the state
of understanding local dynamics near fixed points. All currently known
examples can be derived from a complete understanding of the stable
orbits near the corresponding fixed point. The following corollary
follows immediately via Lemma~\ref{lem:localbasinIsInFatouComp}
from Theorem~\ref{thm:AttrBasinFatBieb}, Lemma~\ref{lem:AttrIffTopAttr},
Theorem~\ref{thm:SemiAttrUeda}, and Theorem~\ref{thm:FatouCCstarMulti}
and summarises results from Rosay and Rudin \cite{RosayRudin1988HolomorphicmapsfrombfCntobfCn},
Ueda \cite{Ueda1986LocalstructureofanalytictransformationsoftwocomplexvariablesI},
Bracci, Raissy, and Stensønes \cite{BracciRaissyStensonesAutomorphismsofmathbbCkwithaninvariantnonrecurrentattractingFatoucomponentbiholomorphictomathbbCtimesmathbbCk1},
and the author \cite{Reppekus2019PeriodiccyclesofattractingFatoucomponentsoftypemathbbCtimesmathbbCd1inautomorphismsofmathbbCd}.
\begin{cor}
Let $F\in\Aut(\mathbb{C}^{2},0)\cap\Aut(\mathbb{C}^{2})$.
\begin{enumerate}
\item \label{enu:recurrentInSum}If $0$ is an attracting fixed point, then
the basin of attraction $\Omega$ of $0$ is a recurrent attracting
Fatou component admitting a biholomorphism $\varphi:\Omega\to\mathbb{C}^{2}$
conjugating $F$ to a polynomial normal form. All recurrent attracting
Fatou components are of this form.
\item If $0$ is an isolated parabolic-attracting fixed point of order $k+1$,
then the only Fatou components attracting to $0$ are $k$ (non-recurrent)
parabolic domains $\Omega_{1},\ldots,\Omega_{k}$, each admitting
a biholomorphism $\varphi:\Omega_{j}\to\mathbb{C}^{2}$ conjugating
$F$ to the translation $(z,w)\mapsto(z+1,w)$.
\item If $F$ is one-resonant of index $(1,1)$ and of the form 
\begin{equation}
F(z,w)=\paren[\Big]{\lambda z\paren[\Big]{1-\frac{(zw)^{k}}{2k}},\overline{\lambda}w\paren[\Big]{1-\frac{(zw)^{k}}{2k}}}+O(\norm{(z,w)}^{l})\label{eq:oneResFormSummary}
\end{equation}
with $l\ge4$, and satisfies the reduced Brjuno condition, then the
only Fatou components attracting to $0$ are $k$ parabolic domains
$\Omega_{1},\ldots,\Omega_{k}$, each admitting a biholomorphism $\varphi:\Omega_{j}\to\mathbb{C}\times\mathbb{C}^{*}$
conjugating $F$ to the translation $(z,w)\mapsto(z+1,w)$.
\end{enumerate}
\end{cor}

\begin{rem}
It is reasonable to assume, that the last part extends to non-degenerate
one-resonant germs that are parabolically attracting and satisfy the
reduced Brjuno condition. The resulting Fatou components would be
of type $\mathbb{C}^{2}$ or $\mathbb{C}\times\mathbb{C}^{*}$ (depending
on the index of one-resonance) with conjugation to a translation.
\end{rem}

Further global basins arise from Theorem~\ref{thm:Hakim2}, Theorem~\ref{thm:RongQuasiParaParaMnf},
and Theorem~\ref{thm:FatouCCstarMulti}. These are candidates for
Fatou components, but the corresponding Fatou component may be larger.
\begin{cor}
\label{cor:FatouCandidates}Let $F\in\Aut(\mathbb{C}^{2},0)\cap\Aut(\mathbb{C}^{2})$.
\begin{enumerate}
\item If $F$ is tangent to the identity of order $k+1$ then, for each
attracting non-degenerate characteristic direction $[v]$, there exist
$k$ parabolic domains $\Omega_{1},\ldots,\Omega_{k}$ tangent to
$v$, each admitting a biholomorphism $\varphi:\Omega_{j}\to\mathbb{C}^{2}$
conjugating $F$ to the translation $(z,w)\mapsto(z+1,w)$.
\item If $F$ is parabolic-elliptic with parabolic multiplier $1$, dynamically
separating and of semi-parabolic order $k+1$, and the unique non-degenerate
characteristic direction $[v]$ is attracting, then there exist $k$
parabolic domains $\Omega_{1},\ldots,\Omega_{k}$ tangent to $v$,
each admitting a biholomorphism $\varphi:\Omega_{j}\to\mathbb{C}^{2}$
conjugating $F$ to the translation $(z,w)\mapsto(z+1,w)$.
\item If $F$ is one-resonant of the form (\ref{eq:oneResFormSummary}),
then there exist $k$ parabolic domains $\Omega_{1},\ldots,\Omega_{k}$,
each admitting a biholomorphism $\varphi:\Omega_{j}\to\mathbb{C}^{2}$
conjugating $F$ to the translation $(z,w)\mapsto(z+1,w)$.
\end{enumerate}
\end{cor}

In all of these cases, we ask a number of open questions:
\begin{question}
\begin{enumerate}
\item Can we rule out orbits outside these parabolic domains converging
to $0$?
\item Can we rule out local basins/attracting Fatou components outside these
parabolic domains?
\item Are these parabolic domains Fatou components?
\item Are these parabolic domains contained in distinct Fatou components?
\end{enumerate}
A positive answer to any of these questions implies a positive answer
for all following questions.
\end{question}

The only recurrent attracting Fatou components are the basins of attracting
fixed points, in particular, they are Fatou-Bieberbach domains. For
the non-recurrent case, we know the following:
\begin{rem}
By \cite[Proposition~5.1]{Ueda1986LocalstructureofanalytictransformationsoftwocomplexvariablesI},
attracting Fatou components are Runge, and, by \cite{Serre1955UneproprietetopologiquedesdomainesdeRunge},
for every Runge domain $D\subseteq\mathbb{C}^{d}$, we have $H^{q}(D)=0$
for $q\ge d$. Hence $\mathbb{C}\times(\mathbb{C}^{*})^{d-1}$ has
the highest possible degree of non-vanishing cohomology for an attracting
Fatou component. 
\end{rem}

Our known results suggest a few questions regarding the general classification:
\begin{question}[Bracci]
Let $F\in\Aut(\mathbb{C}^{2})$ and $V$ a non-recurrent invariant
attracting Fatou component of $F$.
\begin{enumerate}
\item \label{enu:CCstarProds}Is $V$ necessarily biholomorphic to $\mathbb{C}^{2}$
or $\mathbb{C}\times\mathbb{C}^{*}$?
\item Does $V$ admit a conjugation to a translation?
\item \label{enu:fibrebundle}Does $V$ have a fibre-bundle structure preserved
by $F$ (necessarily over $\mathbb{C}$)?
\item \label{enu:KobayashiVanish}Does the Kobayashi metric $k_{V}$ on
$V$ vanish?
\item Can $V$ have more than one ``hole''?
\end{enumerate}
\end{question}

Positive answers to \ref{enu:fibrebundle} and \ref{enu:KobayashiVanish},
would imply \ref{enu:CCstarProds}. The difficulty with this ``internal
dynamics'' approach is, that the fibrations in the known examples
all stem from the local dynamics near the fixed point.

From the opposite direction, we approach the classification by elimination.
Some negative results, that do not require complete understanding
of stable orbits, follow from Theorem~\ref{thm:NishimuraSemiAttrFixedCurve},
Theorem~\ref{thm:2DsnailLemma}, Theorem~\ref{thm:StableCentreMnf},
and Theorem~\ref{thm:Brjuno}:
\begin{cor}
Let $F\in\Aut(\mathbb{C}^{2},0)$. Then under any of the following
conditions, there is no invariant Fatou component of $F$ attracting
to $0$:
\begin{enumerate}
\item $0$ is a parabolic-attracting fixed point contained in a curve of
fixed points.
\item $0$ is elliptic-attracting.
\item $F$ has at least one repelling multiplier.
\item $0$ is elliptic and satisfies the Brjuno condition.
\end{enumerate}
\end{cor}

\begin{rem}[Perspectives]
In dimension two:
\begin{enumerate}
\item In the parabolic-elliptic case, if the elliptic multiplier satisfies
the Brjuno condition, Theorem~\ref{thm:PoeschelBrjuno} provides
a Siegel disk. This might enable an argument similar to that of the
previous section.
\item In the general parabolic-elliptic or one-resonant case, it seems an
understanding of higher dimensional elliptic dynamics is necessary.
However, the dynamics induced by non-Brjuno elliptic multipliers are
still mysterious.
\item In the parabolic case, simplification by elimination of terms is limited.
The dynamics in non-characteristic directions do not seem to admit
one-dimensional analogues and will likely require a highly specific
approach.
\item All three cases from Corollary~\ref{cor:FatouCandidates} have corresponding
non-attracting and degenerate cases, that will require a more detailed
study.
\end{enumerate}
\end{rem}

The general elliptic case is completely open. Finding a generalisation
of the snail Lemma~\ref{lem:SnailLemma} would be a major step. Based
on the parabolic attracting behaviour observed an (ambitious) guess
is the following:
\begin{conjecture}
If there exists a uniform local basin for $F\in\Aut(\mathbb{C}^{2},0)$
at the non-attracting fixed point $0$, then there is a non-empty
product of multipliers of $F$ that is equal to $1$.
\end{conjecture}

One may try to approach this by finding a deformation of a distinguished
boundary of a polydisc around $0$.

\cleardoublepage{}

\phantomsection\addcontentsline{toc}{chapter}{\indexname}\printindex\cleardoublepage{}
\phantomsection\addcontentsline{toc}{chapter}{\bibname}

\newcommand{\etalchar}[1]{$^{#1}$}
\providecommand{\bysame}{\leavevmode\hbox to3em{\hrulefill}\thinspace}
\providecommand{\MR}{\relax\ifhmode\unskip\space\fi MR }
\providecommand{\MRhref}[2]{%
  \href{http://www.ams.org/mathscinet-getitem?mr=#1}{#2}
}
\providecommand{\href}[2]{#2}

\end{document}